\documentclass{amsart}
\usepackage{amscd,amsmath,amssymb,amsfonts,verbatim, xcolor}
\usepackage[all]{xy}
\allowdisplaybreaks
\usepackage{calrsfs}
\DeclareMathAlphabet{\pazocal}{OMS}{zplm}{m}{n}

\newtheorem{theorem}{Theorem}[subsection]
\newtheorem{corollary}[theorem]{Corollary}
\newtheorem{lemma}[theorem]{Lemma}  
\newtheorem{proposition}[theorem]{Proposition}
\theoremstyle{definition}
\newtheorem{definition}[theorem]{Definition}
\theoremstyle{conjecture*}
\newtheorem*{conjecture*}{Conjecture}
\theoremstyle{theorem*}
\newtheorem*{theorem*}{Theorem}
\theoremstyle{emp}

\theoremstyle{question}

\theoremstyle{notat}

\theoremstyle{claim}

\theoremstyle{construction}

\theoremstyle{remark}
\newtheorem{remark}[theorem]{Remark}
\theoremstyle{example}

\newtheorem{innercustomthm}{Theorem}
\newenvironment{customthm}[1]
  {\renewcommand\theinnercustomthm{#1}\innercustomthm}
  {\endinnercustomthm}
\numberwithin{equation}{subsection}
\title{ Infinitesimal Chow Dilogarithm}
\author{S\.{I}nan \"{U}nver}
\address{Ko\c{c} University, Mathematics Department. Rumelifeneri Yolu, 34450, Istanbul, Turkey}
\address{Freie Universit\"{a}t Berlin, Mathematics Department, Arnimallee 3, 14195, Berlin, Germany}
\email{sunver@ku.edu.tr}

\begin{document}
\maketitle
\noindent

\begin{abstract}
Let $C_{2}$ be a smooth and projective curve over the ring of dual numbers  of a field $k.$  Given non-zero rational functions $f,g,$ and $h$ on
$C_{2},$  we define an invariant $\rho(f\wedge g \wedge h) \in k.$ This is an analog of the real analytic Chow dilogarithm and the  extension to  non-linear cycles of the additive dilogarithm of  \cite{Unver}.  Using this construction we state and prove an infinitesimal version of the strong reciprocity conjecture \cite{gon3}. Also using $\rho,$ we define an infinitesimal regulator on algebraic cycles of weight two which generalizes Park's construction in the case of cycles with modulus \cite{P1}.
  \end{abstract}

\section{Introduction}

Let $k$ be a field of characteristic 0 and    $k_{n}:=k[t]/(t^n),$ with $k_2$ being   the ring of dual numbers over $k.$ One expects an abelian category of motives  over any scheme over $k,$ and in particular over  $k_2.$   Goncharov while considering the degeneration of motives attached to hyperbolic manifolds of odd dimension expected that motives over the dual numbers should have  properties similar to euclidean scissors congruence class groups \cite[p. 616]{vol}. Bloch and Esnault further expected that the  real analytic dilogarithm should degenerate to an additive dilogarithm which would give an analog of the volume map on the euclidean scissors congruence class in three dimensions.  This insight was  realised in the construction of an additive dilogarithm for the weight two motivic cohomology complex  based on the localisation sequence in their fundamental paper \cite{BE2}.  Later an additive dilogarithm  for the Bloch complex was constructed in \cite{Unver}.  In a precise sense, the current paper is an extension of the construction of \cite{Unver} to the non-linear case. The construction in this non-linear case is much more involved and requires many new ideas.  In the introduction, we will  try to  explain the main ideas behind the construction.  

Suppose that $X/\mathbb{C}$ is a smooth, projective curve over $\mathbb{C}$ and $f_{1},\, f_2, \,$ and $f_3$  are rational functions on $X.$ Letting 
$$
r_{2}(f_{1},f_{2},f_{3}):={\rm Alt}_{3}(\frac{1}{6} \log |f_1| \cdot d \log |f_{2}| \wedge d \log |f_{3}| -\frac{1}{2}\log |f_1|\cdot d \arg f_{2} \wedge d \arg f_{3} ),
$$
with the property that $d r_{2}(f_{1},f_{2},f_{3})={\rm Re}( {\rm Alt}_{3}(d \log (f_1)\wedge d \log (f_2) \wedge d \log (f_3) )),$ one has, up to a constant multiple, the Chow dilogarithm map \cite[p. 4]{gon3}:
$
\rho_{\mathbb{R}}: \Lambda^{3} \mathbb{C}(X) ^{\times} \to \mathbb{R}
$ given by 
\begin{eqnarray}\label{realchow}
\rho_{\mathbb{R}}(f_1 \wedge f_2 \wedge f_3) := \int_{X(\mathbb{C})} r_{2}(f_{1},f_{2},f_{3}).
\end{eqnarray}

The main construction of this paper is the analog $\rho$ of this map  for $\mathbb{R}$ replaced with $k_2,$ which is expected in the spirit of the above discussion about the analogy between mixed tate motives over $k_2$ and euclidean scissors congruence class groups.    In the main set-up of this paper $X/\mathbb{C}$ is replaced with a smooth and projective curve $C_{2}/k_{2}.$  

The main challenge in constructing such a map is to find an analog of the integral (\ref{realchow}) for a curve $C_{2}/k_2.$ A related problem is that the integrand involves taking the real part of an algebraic form. Therefore, we have to find an algebraic construction in the context of motives over $k_2$ which will be an analog of this. We have encountered such a function in \cite{Unver}: $\log^{\circ}: k[[t]] ^{\times} \to k[[t]]$  defined as $\log^{\circ}(a):=\log (a/a(0)).$  We would like to think of $\log ^{\circ}$ as the  analog of   a  branch of the logarithm  or as the real part of the logarithm. In order to find an expression which is an analog of (\ref{realchow}), we are again motivated by the construction in \cite{Unver}: the construction in this paper restricted to the case of $\mathbb{P}^{1} _{k_2}$ and a triple of  linear fractional transformations  gives the construction in \cite{Unver}. The precise statement can be found in Theorem \ref{compunver}. Even though both the arguments in \cite{Unver}  and the ones in the  current paper use the idea of  liftings of the functions  to some thickening of $k_2,$ the approach in \cite{Unver} is computational, here it is more geometric. We briefly outline the main idea.

In \cite{Unver}, we defined the additive dilogarithm  $\ell i_2: B_{2}(k_2) \to k,$  by 
$
\ell i_2(\{s+at\} _{2})=-\frac{a^3}{2s^2 (1-s)^2}.
$
In order to show that this function indeed satisfies the functional equation for the dilogarithm and deduce its basic properties, we interpreted this formula as follows. Let $\ell: \Lambda^{2} _{\mathbb{Z}} k[[t]]^{\times} \to k$ denote the map  (\ref{ellmap}), which in this case takes the form  $\ell(u\wedge v):=\log^{\circ}(u)[t^2]\cdot \log^{\circ}(v)[t]-\log^{\circ}(u)[t]\cdot \log^{\circ}(v)[t^2],$ where for $a \in k[[t]], $ $a[t^n]$ denotes the coefficient of $t^n$ in $a.$ Then the map $B_{2}(k[[t]]) \to k$ defined by sending 
$\{\alpha\}_{2} $   to $\ell(1-\alpha) \wedge \alpha),$ factors through the canonical projection $B_{2}(k[[t]]) \to B_{2}(k_2),$ to give $\ell i_2.$ 

Let us now turn to the context of this paper of a curve $C_{2}/k_2.$ When considering general rational functions on $C_{2},$ we  put a slight restriction on the type of singularities that our functions can have. We allow those functions on $C_{2}$ which are locally a product of a unit and a power of an element that reduces to a uniformizer on the closed fiber. The precise definition is in Definition 3.1.1. For a choice of liftings of uniformizers $\mathcal{P}_{2},$ we call a rational function on $C_{2},$ $\mathcal{P}_{2}$-good  in \textsection 3.2,   if it satisfies this hypothesis with respect to the uniformizers in $\mathcal{P}_{2}.$   Letting $k(C_{2},\mathcal{P}_{2})^{\times} $ denote the set of rational functions which are $\mathcal{P}_{2}$-good, the infinitesimal Chow dilogarithm $\rho$ is a map
$
\rho: \Lambda^{3} k(C_{2},\mathcal{P}_{2} )^ \times \to k.
 $
 
The main idea in constructing $\rho$ is as follows. If  $C_{2}/k_{2}$ has a global lifting $\tilde{C}/k[[t]]$ to a smooth and projective curve and the functions $(f,g,h)$ have global liftings $(\tilde{f},\tilde{g},\tilde{h})$ to functions on $\tilde{C}$ which are good with respect to a system of uniformizers $\tilde{\mathcal{P}}:=\{\tilde{\pi}_{c}|c \in |\tilde{C}_s| \}$ on $\tilde{C}$ that  lift $\mathcal{P}_{2}$ then  
\begin{eqnarray}\label{simplerho}
\rho(f,g,h)=\sum _{c \in |\tilde{C}_s|} {\rm Tr}_k(\ell(res_{\tilde{\pi}_c} (\tilde{f} \wedge \tilde{g} \wedge \tilde{h}))), 
\end{eqnarray}
where  $res_{\tilde{\pi}_c}$ is the residue map (\ref{goncres}).

The existence of a global lifting of $C_{2}$ together with good liftings of the functions $f,g,$ $h$ is too much to ask for. Even if we knew that such liftings  exist, we need to show that the expression (\ref{simplerho}) is independent of the choices of the liftings which is a non-trivial task.  In the following, we do not assume the existence of any lifting. Instead, we choose local  liftings of $C_{2}$ at all of its  points, including the generic point. These are known to exist by the smoothness assumption. On the generic lifting we choose arbitrary liftings of the functions; on the local liftings we choose good liftings of the functions. The general expression for $\rho$ is then given in Definition 3.2.2. Sections 3.1 and 3.2 are devoted to showing that this sum makes sense and is independent of all the choices. 

In  Definition 3.2.2, there is a differential form $\omega$ which measures the defect of choosing different liftings. We would like to think of different liftings as the analogs of the different branches of the functions under consideration and $\omega $ as measuring the difference between the  choice of two different branches. This differential is in some sense the main technical object.  
The map $\omega$  attaches a 1-form to the following data:  smooth affine schemes  $\pazocal{A}_{i}/k[[t]]$ of relative dimension one, an isomorphism $\chi: \pazocal{A}_{1}/(t^2) \xrightarrow{\sim} \pazocal{A}_{2}/(t^2)$ and triples of functions $(f_i,g_i,h_i)$ in $\pazocal{A}_{i} ^{\times}$ whose reductions modulo $t^2$ map to each other via $\chi.$ The 1-form is given by the formula (2.3.1). However, this description  depends on many  choices and it is of fundamental importance for the  construction of the regulator $\rho$ to show that this construction of $\omega$  is independent of all the choices. These are proven in Lemma 2.3.1 and Lemma 2.3.3. It is at this point that we have to assume that our objects have dimension at most 1. The statements about independence of the choices are not true without this assumption.  The residue property (P8) in \textsection \ref{residue} allows us to compute the residue of $\omega$ when both of its  arguments are good liftings. This property is used when proving that the regulator $\rho$ does not depend on the choice of local liftings in Proposition 3.1.2. 

 The proof of the infinitesimal version of the strong reciprocity conjecture uses the relative Bloch group relation (P11). In order to state this relation, we  first construct a relative version of the standard Bloch group $B_2,$ and a relative version of the part of the weight three Bloch complex in degrees 2 and 3.  The Proposition 2.4.5 then states that $\omega$ defines a map from the Milnor part of the relative Bloch complex of weight three. Here and the discussions below regarding the strong reciprocity conjecture, we would like to emphasize that we consider only the part of the Bloch complex of weight three. There is in fact a term in degree one which contains a version of the Bloch group $B_{3},$ but this term has no consequence for the present paper since the strong reciprocity only refers to the terms of degree 2 and 3. Defining the analog of $B_3$ would take us too far outside of the material of this paper, this question will be dealt with in more generality in a future work.   

The construction of the infinitesimal Chow dilogarithm with the desired properties is in fact more or less equivalent to the infinitesimal version of the strong reciprocity conjecture of Goncharov \cite{gon3}, by using the results of \cite{Unver} on the structure of $B_{2}(k_2):$

\begin{customthm}{3.4.4}  The infinitesimal part of the residue map 
$$
res_{|C|}: \Gamma(k(C_{2},\mathcal{P}_{2}),3) \to \Gamma(k,2)^{\circ}
$$
is homotopic to 0. More precisely, there is a map $h: \Lambda^{3}k(C_2,\mathcal{P}_{2}) ^{\times} \to B_{2}(k_2)^{\circ}$ which makes the diagram 
$$
\xymatrix{ B_2 (k(C_2,\mathcal{P}_{2}))\otimes k(C_2,\mathcal{P}_{2})^{\times} \ar[r] ^-{ \Delta }\ar[d] ^{res_{|C|}}&\Lambda^{3}k(C_2,\mathcal{P}_{2}) ^{\times} \ar[d]^{res_{|C|}} \ar@{.>}[dl] _{h} \\  B_{2}(k_2)^{\circ} \ar[r]^{\delta^{\circ}} & (\Lambda^{2}k_2^{\times})^{\circ}}
$$
commute and has the property that $h(k_{2} ^{\times} \wedge \Lambda ^{2} k(C_2,\mathcal{P}_{2}) ^{\times})=0.$ 
\end{customthm}

In the last section we use the discussions of this paper  to construct a regulator for algebraic cycles of weight two over $k[[t]].$ That this regulator vanishes on boundaries is a direct consequence of the definition. We prove that if two smooth cycles are the same modulo $t^2$  then they have the same regulator value and also we prove that the regulator vanishes on products. Finally in \textsection 4.4.5, we compare our construction in the case of cycles to the construction of Park in the case of cycles with modulus 2.  Here we emphasize that we use the strong sup modulus when defining the additive Chow groups.

There are many questions that are left unanswered in this paper. The most important one in our opinion  is the generalization of  the above construction to higher weights, moduli and to characterstic $p.$  In characteristic $p,$ even in weight two and modulus two, we expect a new component for our Chow dilogarithm which is of  arithmetic nature. Another fundamental question is to understand the relation of the theory of this paper with that of \cite{griff}, which was in fact  one of our main motivations for this paper. These will be the subjects of future work.  A related  question is to study the regulator on algebraic cycles in more detail and  define an equivalence relation on cycles over $k[[t]]$ which would give a construction of the motivic cohomology complex for truncated polynomial rings in terms of algebraic cycles.  We will pursue this approach in our joint work with J. Park. 

{\bf Outline.} The paper is organized as follows. After a recollection of basic facts about splittings in \textsection 2.1, in the remaining part of  \textsection 2, we construct the map $\omega.$  Its uniqueness is proven in \textsection 2.2, and existence in \textsection 2.3. The fundamental properties of $\omega$ are proven in \textsection 2.4. 

In \textsection 3.1 and \textsection 3.2,  we give the construction of the regulator $\rho.$ In \textsection 3.3 and 3.4, we prove the infinitesimal version of the strong reciprocity conjecture. In \textsection 3.5, we relate the construction of this paper to that of \cite{Unver}.

In \textsection 4,  we apply the construction to the case of algebraic cycles of weight two over $k[[t]],$ and construct an infinitesimal regulator on these cycles.  

{\bf Notation.} For a functor $F$ from $k$-algebras to abelian groups, let $F(k_2)^{\circ}$ denote the direct summand of $F(k_{2})$ which is the kernel of the projection $F(k_2)\to F(k)$ corresponding to the $k$-algebra homomorphism $k_{2}=k[t]/(t^2)\to k$ which sends $t$ to $0.$ If $\pazocal{A}$ is a $k[[t]]$ algebra then we denote $\pazocal{A}/(t^n)$ by $\pazocal{A}|_{t^n}.$ Similarly, if $f \in \pazocal{A},$ then we let $f|_{t^n}$ denote its image in $\pazocal{A}|_{t^n}.$

 {\it Acknowledgements.} The author  thanks J. Park very much  for many  discussions related to the contents of this paper, especially regarding  \textsection 4. He thanks H. Esnault for  mathematical discussions and for being a wonderful host to him as an Humboldt Experienced Researcher at Freie Universit\"{a}t, where part of this work was carried out. Finally, he would like to thank the Humboldt Foundation for support.

\section{Algebraic construction and properties of the canonical 1-form}

In this section we would like to present an algebraic construction of the 1-form that we referred to above. We give this construction in a slightly more general context than is necessary for our purposes for the reasons that it makes the construction more transparent and this general construction will be needed for future generalisations of this work.

\subsection{Splittings}

In the first subsection, we review the notion of a {\it splitting}, the existence of which is guaranteed by the infinitesimal lifting properties of  smooth affine schemes. This notion enables us to make explicit computations by  choosing local coordinates.  

For a ring $A$ and an ideal $I$ of $A,$ let $\underline{A}:=A/I,$  a {\it splitting} of the projection  $\pi:A \to \underline{A},$ is a map $\varphi:\underline{A}\to A$ such that $\pi\circ \varphi=id.$ In the special case that we are interested in, the existence of a splitting is guaranteed by Grothendieck \cite[III-Corollaire 5.6]{SGA1}:

\begin{lemma}\label{groth}
Let $k$ be a field,   $A$  a noetherian $k$-algebra, and  $I:=(x)$ an ideal of $A,$ such that   $x$ is not a zero-divisor in $A,$   $A$ is complete with respect to the $I$-adic topology and $\underline{A}$ is smooth over $k.$ Then there is a continuous isomorphism $\underline{A} [[t] \to A,$ which sends  $t$ to the element $x$   and  such that the composition $\underline{A} \hookrightarrow \underline{A} [[t]] \to A \to \underline{A}$ is the identity map on $\underline{A}.$ 
\end{lemma}

\begin{remark}\label{rem split} (i) It is easy to see that for any ring $B,$ there is a 1-1 correspondence between splittings $\varphi: B \to  B[[t]]$ of $\pi:B[[t]] \to B$ and automorphisms $\alpha: B[[t]]  \to B[[t]] $ such that $B\hookrightarrow B[[t]]  \overset{\alpha}{\to} B[[t]]  \overset{\pi} {\to} B $ is the identity and $\alpha(t)=t.$ 

(ii) Using   Lemma \ref{groth}, the above remark  generalizes to any $A$ and $I$ which satisfy the hypothesis of the lemma. Namely,  there is a 1-1 correspondence between splittings $\varphi: \underline{A} \to A$ of $\pi:A\to \underline{A}$ and isomorphisms $\alpha:\underline{A}[[t]]\to A$ such that $\alpha(t)=x,$  and  $\underline{A}\hookrightarrow \underline{A}[[t]]  \overset{\overline{\varphi}}{\to} A\overset{\pi} {\to} \underline{A} $ is the identity. Given $\overline{\varphi}$ the corresponding $\varphi$ is obtained by restricting to $\underline{A}.$ 

(iii) The  splittings behave functorially with respect to localisation. 
More precisely, if  $\varphi: \underline{A} \hookrightarrow A$ is a splitting and  $g \in \underline{A},$ then $\varphi$ induces a splitting $\varphi _g: \underline{A}_g \hookrightarrow A_{\varphi (g)}.$
\end{remark}

We will use the second remark above constantly in order to choose a local system of coordinates. Given a splitting $\varphi,$ there is a unique isomorphism $\overline{\varphi}: \underline{A}[[t]] \to A$ as above such that $\overline{\varphi}|_{\underline{A}}=\varphi.$ We will use $\overline{\varphi} ^{-1}: A \to \underline{A}[[t]] ,$ as a local chart.

%We now consider the following special case; suppose $A$ is a noetherian $k$-algebra with a regular sequence $x_1, x_2 \in A$ such that the following conditions hold:
%\begin{enumerate}
%\item $A$ is complete with respect to the $(x_1, x_2)$-adic topology.
%\item each of the quotients $B_1:= A/ (x_1), B_2:= A/(  x_2)$ and $C:= A/ (x_1, x_2)$ are smooth over $k$.
%\end{enumerate}

%AAAAA

In general, the splittings as above are not unique. We  prove uniqueness in a case that we need, in order to  handle the 0-dimensional cycles. This result is  well-known, but we give a  proof since it is easier than finding a suitable reference. 
\begin{lemma}\label{uni-qcf-cf}
Suppose that  $k$ is  a perfect  field,  and  $A$ and $I$ are as in Lemma \ref{groth} with $\underline{A}$ a finite extension of $k.$   Then  there exists a unique splitting $\varphi: \underline{A} \to A$. 
\end{lemma}

\begin{proof}
Since $\underline{A}$ is smooth over $k,$ the existence follows from Lemma \ref{groth}. For any $k$-algebra $\tilde{B}$ with a square-zero ideal $I,$ and a $k$-algebra homomorphism $\phi:\underline{A} \to \tilde{B}/I,$ the set of liftings of $\phi$ to a $k$-algebra homomorphism $\tilde{\phi}:\underline{A} \to \tilde{B}$ is a psuedo-torsor under $Hom_{\underline{A}}(\Omega^{1}  _{\underline{A}/k},I).$ Since under the assumptions $\Omega^{1} _{\underline{A}/k}=0,$ there is at most one lifting.  Applying this inductively to the rings $\tilde{B}=A/I^{n},$ starting with the identity map for $n=1$ gives the result. 
 \end{proof}

\subsection{Uniqueness of the canonical 1-form $\omega$}

Let $A:=k[[t]],$ we will be dealing with a full subcategory  $\mathcal{C}^{1} _{A}$ of the category of $A$-algebras. The objects of this category consist of $A$-algebras $\pazocal{A},$ which are smooth of relative dimension 0 or 1, are an integral domain and have the property that the residue fields of the closed points of $\pazocal{A}/(t)$ are  one of the following type: a finite extension of $k,$ if the relative dimension is 1; a finitely generated field of transcendence degree 1 over $k,$ or a field of Laurent series $k'((s))$ over a finite extension $k'$ of $k,$ if the relative dimension is 0.     Note that the  condition on the residue fields of the  closed points  is satisfied if $\pazocal{A}$ is an algebra of finite type over $A$ or is the localisation or completion at a point in the closed fiber of such an algebra.

  If $\pazocal{A}$ is an object of  $\mathcal{C}^{1} _{A},$ we denote  the quotient $\pazocal{A}/(t)$ by $\underline{\pazocal{A}}.$ If we have an isomorphism $\chi:\tilde{\pazocal{A}}|_{t^2} \to \hat{\pazocal{A}}|_{t^2},$ and a pair $(\tilde{p},\hat{p}),$ where  $\tilde{p}:=( \tilde{y}_{1},\tilde{y}_{2},\tilde{y}_{3})$   and  $\hat{p}:=( \hat{y}_{1},\hat{y}_{2},\hat{y}_{3})$   such that     $\tilde{y}_{i} \in \tilde{\pazocal{A}}^{\times},$ $\hat{y}_{i} \in \hat{\pazocal{A}} ^{\times},$ and   $\chi (\tilde{y}_{i}|_{t^2}) =\hat{y}_{i}|_{t^2} $ in $\hat{\pazocal{A}}|_{t^2},$ for $1\leq i \leq 3,$ then we will construct a canonical element  in $\Omega^{1}_{\underline{\hat{\pazocal{A}}}/k}.$ We describe this construction in detail in the next few sections.

Given an isomorphism $\chi: \tilde{\pazocal{A}}|_{t^2} \to \hat{\pazocal{A}}|_{t^2},$ we let  $\beta (\tilde{\pazocal{A}},\hat{\pazocal{A}},\chi)$ denote the abelian group   of  pairs $(\tilde{y},\hat{y}) \in \tilde{\pazocal{A}} ^{\times} \times \hat{\pazocal{A}} ^{\times}$ such that $\chi(\tilde{y}|_{t^2})=\hat{y}|_{t^2}.$ We let $\times ^{n}\beta  (\tilde{\pazocal{A}},\hat{\pazocal{A}},\chi)$ denote the $n$-fold cartesian product of $\beta (\tilde{\pazocal{A}},\hat{\pazocal{A}},\chi)$ with itself. Note that this last group can also be thought of as the group of  pairs $(\tilde{p},\hat{p})$ of $n$-tuples of elements of $\tilde{\pazocal{A}}^{\times}$ and $\hat{\pazocal{A}}^{\times}$ such that if $\tilde{p}=(\tilde{y}_{1},\cdots, \tilde{y}_{n})$ and $\hat{p}=(\hat{y}_{1},\cdots, \hat{y}_{n}),$ then $\chi(\tilde{y}_i|_{t^2})=\hat{y}_i|_{t^2} \in \hat{\pazocal{A}}|_{t^2},$ for all $1 \leq i \leq n.$ We abbreviate this last condition as $\chi(\tilde{p}|_{t^2})=\hat{p}|_{t^2} \, $. We will freely switch between these two notations without mentioning it.

We will construct a  map 
$$
\omega(\cdot,\cdot   ,\chi): \times ^3 \beta(\tilde{\pazocal{A}},\hat{\pazocal{A}},\chi)\to \Omega^1 _{ \underline{\hat{\pazocal{A}}} /k } ,						
$$
with the following properties.

 (P1) {\it Normalising condition/definition.} Let $k'/k$ be a finite  extension.  Suppose that $\hat{\pazocal{A}}=\tilde{\pazocal{A}}=k'((s))[[ t]]$ and $\chi=id.$ Then because of the assumptions $\hat{y}_{i} -\tilde{y}_{i} \in(t^2),$ for $1 \leq i \leq 3.$  Therefore, we can write uniquely 
 $\hat{y}_i=\alpha_{0i}e^{t\alpha_{1i}+t^2\hat{\alpha}_{2i}+\cdots }$ and $\tilde{y}_i=\alpha_{0i}e^{t\alpha_{1i}+t^2\tilde{\alpha}_{2i}+\cdots },$ with $\alpha_{ji}, \hat{\alpha}_{ki},\tilde{\alpha}_{ki} \in k'((s)),$ for $0 \leq j \leq 1, \,2 \leq k,$ and $1\leq i \leq 3. $  

Then we have

\begin{eqnarray}\label{mainform}
\tilde{\omega}(\tilde{p},\hat{p},id) := \sum _{\sigma \in S_{3}} (-1) ^\sigma \alpha_{1 \sigma(1)} ( \tilde{\alpha}_{2\sigma(3)}-\hat{\alpha}_{2\sigma(3)} ) \cdot d \log (\alpha_{0\sigma(2)}) \in \Omega^{1} _{k'((s))/k}.
\end{eqnarray}

 (P2) {\it Completion.}   Suppose that we have  $\tilde{\pazocal{A}}$ and $\hat{\pazocal{A}}$ as above together with an  isomorphism $\chi:\tilde{\pazocal{A}}|_{t^2} \to \hat{\pazocal{A}}|_{t^2},$ and  pairs $(\tilde{p},\hat{p}),$ where  $\tilde{p}:=( \tilde{y}_{1},\tilde{y}_{2},\tilde{y}_{3})$   and  $\hat{p}:=( \hat{y}_{1},\hat{y}_{2},\hat{y}_{3})$   such that     $\tilde{y}_{i} \in \tilde{\pazocal{A}} ^{\times},$ $\hat{y}_{i} \in \hat{\pazocal{A}}^{\times},$ and   $\chi (\tilde{y}_{i}|_{t^2}) =\hat{y}_{i}|_{t^2} $ in $\hat{\pazocal{A}}|_{t^2},$ for $1\leq i \leq 3.$ 
 
 Let  $\tilde{\pazocal{B}}$ and $\hat{\pazocal{B}}$ denote the completions of  $\tilde{\pazocal{A}}$ and $\hat{\pazocal{A}}$ along the closed fiber. Note that $\chi$ induces an isomorphism 
 $$
  \tilde{\pazocal{B}}|_{t^2}\simeq  \tilde{\pazocal{A}}|_{t^2} \to \hat{\pazocal{A}}|_{t^2} \simeq \hat{\pazocal{B}}|_{t^2},
 $$
 which we will  denote by $\psi.$ Let $\tilde{q}$ and $\hat{q}$ denote the images of $\tilde{p}$ and $\hat{p}$ in the respective completions $\tilde{\pazocal{B}}$ and $\hat{\pazocal{B}}$. Then we have 
 $$
 \omega(\tilde{p},\hat{p},\chi)=\omega(\tilde{q},\hat{q},\psi)
 $$
 
 (P3) {\it Functoriality.}  Suppose that we have  a diagram 
 \[
 \begin{CD}
\tilde{\pazocal{A}}_1 @>{\chi_1}>> \hat{\pazocal{A}}_1 \\
@V{f}VV  @V{g}VV\\
\tilde{\pazocal{A}}_2 @>{\chi_2}>>  \hat{\pazocal{A}}_2 ,
 \end{CD}
 \]
 where $\chi_{i}$  are isomorphisms, and 
 such that there exists a closed point in $\tilde{\mathfrak{p}}_2 \in {\rm Spec}\,  \underline{\hat{\pazocal{A}}}_2 ,$ with the property that $g$ induces an isomorphism from the localisation at $(t)$ of the completion  of $\hat{\pazocal{A}}_1$ at $g^{-1}(\tilde{\mathfrak{p}}_2)$ to the localisation at $(t)$ of the completion of $\hat{\pazocal{A}}_2$ at $\tilde{\mathfrak{p}}_2.$ Note that this also implies that $f$ induces an isomorphism from the localisation at $(t)$ of the completion  of $\tilde{\pazocal{A}}_1$ at $(\chi_2 \circ f)^{-1}(\tilde{\mathfrak{p}}_2)$ to the localisation at $(t)$ of the  completion of $\tilde{\pazocal{A}}_2$ at $\chi_{2} ^{-1}(\tilde{\mathfrak{p}}_2).$
 
Let $(\tilde{p},\hat{p}) \in \times ^{3} \beta(\tilde{\pazocal{A}},\hat{\pazocal{A}},\chi_{1} |_{t^2}),$ then  
 we have 
 $$
\underline{g}_{*}( \omega(\tilde{p},\hat{p},\chi_1 |_{t^2}))=\omega(f(\tilde{p}),g(\hat{p}),\chi_2 |_{t^2}),
 $$
 where $\underline{g}_{*}: \Omega^1 _{ \underline{\hat{\pazocal{A}} }_1 /k } \to \Omega^1 _{ \underline{\hat{\pazocal{A}} }_2 /k }$ is the map induced by $g.$

\begin{proposition}
The properties (P1)-(P3) uniquely determine $\omega.$ 
\end{proposition}

\begin{proof} 
Suppose that we would like to determine the value of $\omega(\tilde{p},\hat{p},\chi),$ for a general $(\tilde{p},\hat{p},\chi)$ as above. By property $(P2),$ we will assume that $\tilde{\pazocal{A}}$ and $\hat{\pazocal{A}}$ are complete with respect to the ideal defining the closed fiber. 

Next note that by smoothness and completeness, we have isomorphisms $\tilde{\pazocal{A}}\xrightarrow{\sim} \underline{\tilde{\pazocal{A}}}[[t]]$ and $\hat{\pazocal{A}}\xrightarrow{\sim} \underline{\hat{\pazocal{A}}}[[t]],$ which are the identity maps modulo $(t).$  Through these isomorphisms it follows that  
$\chi$ corresponds to an isomorphism $\underline{\tilde{\pazocal{A}}}[[t]]/(t^2) \xrightarrow{\sim} \underline{\hat{\pazocal{A}}}[[t]]/(t^2)$ and therefore can be lifted, non-uniquely, to an isomorphism 
$
\overline{\chi}:  \tilde{\pazocal{A}} \to \hat{\pazocal{A}}.
$
Choosing isomorphisms $\tilde{\phi}:\tilde{\pazocal{A}}\xrightarrow{\sim} B[[t]]$ and $\hat{\phi}:\hat{\pazocal{A}}\xrightarrow{\sim} B[[t]]$ and using (P3) with $\chi_{2}:= \hat{\phi} \circ \overline{\chi} \circ \tilde{\phi}^{-1},$ we assume without loss of generality that $\tilde{\pazocal{A}}=\hat{\pazocal{A}}=B[[t]].$ 

The completion of the local ring of $B[[t]]$ at a closed point is of the form $k'[[s]][[t]]$ for a finite extension $k'/k$ of $k,$ and its localisation at $(t)$ is  $k'((s))[[t]].$ By the functoriality (P3) with respect to the pair of inclusions $B[[t]] \to k'((s))[[t]]$ and the injectivity of the map $\Omega^{1} _{B/k} \to \Omega^{1} _{k'((s))/k},$ we reduce to the case when $\tilde{\pazocal{A}}=\hat{\pazocal{A}}=k'((s))[[t]],$ and an  arbitrary $A$-automorphism $\overline{\chi}$ of  $k'((s))[[t]].$ 

Finally, taking all the rings in (P3) to be $k'((s))[[t]],$ $\chi_1=f=\overline{\chi}$ and $\chi_2=g=id,$ by (P3), we reduce to the situation in (P1), where the value of the 1-form is given by  an explicit formula. This finishes the proof. 
\end{proof}

In the next  sections we will first prove the existence of $\omega$ with the properties (P1) and (P3) and then prove that it satisfies the further properties (P3)-(P10)

\subsection{Existence of $\omega$ and proof of the properties (P1)-(P3)}

\subsubsection{Construction of $ \omega(\tilde{p},\hat{p},\chi)$}

Let $\tilde{\pazocal{A}}$ and $\hat{\pazocal{A}}$ be objects of $\mathcal{C}^{1} _{A},$ and 
$\chi:\tilde{\pazocal{A}}/({t^2}) \to \hat{\pazocal{A}}/({t^2}),$ an isomorphism over $A,$ and $(\tilde{p},\hat{p}) \in \times ^3 \beta(\tilde{\pazocal{A}},\hat{\pazocal{A}},\chi)$ as above. We will define $ \omega(\tilde{p},\hat{p},\chi)$ by first making some choices. Later, we will show that this definition is independent of all the choices.

First, we replace $\tilde{\pazocal{A}} $ and $\hat{\pazocal{A}}$ with their completions along the closed fiber. Note that this does not change $\underline{\hat{\pazocal{A}}}$ and hence $\Omega^{1} _{\underline{\hat{\pazocal{A}}}/k}$ in which $\omega$ takes values.   Similarly, $\chi$ and the triples of points $\tilde{p}$ and $\hat{p}$ give corresponding points on the completions. Therefore, without loss of generality, we will assume that $\tilde{\pazocal{A}} $ and $\hat{\pazocal{A}}$ are complete with respect to $(t).$ From now on, in this section, we will always assume that the $A$-algebras under consideration are complete with respect to the ideal $(t).$ By the smoothness assumptions we have {\it non-canonical} isomorphisms $\tilde{\pazocal{A}}\simeq  \underline{\tilde{\pazocal{A}}}[[t]]$ and $\hat{\pazocal{A}}\simeq  \underline{\hat{\pazocal{A}}}[[t]].$

Let $\overline{\chi}:\tilde{\pazocal{A}} \xrightarrow{\sim} \hat{\pazocal{A}}$  be any lifting of $\chi,$ and $\varphi: \underline{\hat{\pazocal{A}}} \to \hat{\pazocal{A}}$ be any splitting of the canonical projection.  Denote by $\overline{\varphi}$ the corresponding isomorphism $\underline{\hat{\pazocal{A}}}[[t]] \xrightarrow{\sim} \hat{\pazocal{A}}.$ 
We then we define $\omega(\tilde{p},\hat{p},\chi)$  as follows: 
$$
\omega(\tilde{p},\hat{p},\chi):= \Omega( \overline{\varphi}^{-1}( \overline{\chi}(\tilde{p})), \overline{\varphi}^{-1}(\hat{p}) ),
$$
where we will define $\Omega$ below. 

Let  $(\tilde{q},\hat{q})$ be a pair with  $\tilde{q},\hat{q} \in \times ^3\beta(\underline{\hat{\pazocal{A}}}[[t]])$ and are congruent modulo $(t^2).$ Let us write $\tilde{q}=(\tilde{y}_{1},\tilde{y}_{2},\tilde{y_{3}})$ and $\hat{q}=(\hat{y}_{1},\hat{y}_{2},\hat{y_{3}}).$ Then by assumption $\hat{y}_{i}-\tilde{y}_i \in (t^2),$ for all $1 \leq i \leq 3.$   We write $\hat{y}_i=\alpha_{0i}e^{t\alpha_{1i}+t^2\hat{\alpha}_{2i}+\cdots }$ and $\tilde{y}_i=\alpha_{0i}e^{t\alpha_{1i}+t^2\tilde{\alpha}_{2i}+\cdots },$ with $\alpha_{ji}, \hat{\alpha}_{ki},\tilde{\alpha}_{ki} \in \underline{\hat{\pazocal{A}}} ,$ for $0 \leq j \leq 1, \,2 \leq k,$ and $1\leq i \leq 3,$ as in (P1).   We then define

\begin{eqnarray}\label{defomega}
\Omega (\tilde{q},\hat{q}) := \sum _{\sigma \in S_{3}} (-1) ^\sigma \alpha_{1 \sigma(1)} ( \tilde{\alpha}_{2\sigma(3)}-\hat{\alpha}_{2\sigma(3)} ) \cdot d \log (\alpha_{0\sigma(2)}) \in \Omega^{1} _{\underline{\hat{\pazocal{A}}}/k},
\end{eqnarray}
 exactly as in (P1). 

This construction  a priori depends on $\overline{\chi}$ and $\varphi.$ Therefore, we will denote the $\omega $ defined above temporarily by $\omega_{\overline{\chi},\varphi}.$ We will show below that, in fact, it is independent of $\overline{\chi}$ and $\varphi.$  In order to show independence with respect to the choice of the splitting $\varphi,$ we will start with the following lemma which deals with the split case. 

\begin{lemma}
Let $\pazocal{A}:=k'((s))[[t]],$ where $k'/k$ is a finite extension. Suppose that $(\tilde{p},\hat{p}) \in \times ^3 \beta (\pazocal{A},\pazocal{A},id),$ and $\psi:\pazocal{A} \to \pazocal{A}$ be any continuous $A$-algebra isomorphism which induces the identity map on $\underline{\pazocal{A}}.$ Then 
$$
\Omega(\tilde{p},\hat{p})=\Omega(\psi(\tilde{p}), \psi(\hat{p})).
$$
\end{lemma}

\begin{proof} 
Since $\psi$ is assumed to be an $A$-algebra homomorphism, it is completely determined by its restriction to $k'((s)). $ Since $\psi$ is identity on $k,$ reduces to identity modulo $(t),$ and $\Omega^{1} _{k'/k}=0,$ $\psi$ is identity on $k'.$   Therefore, it is completely determined by its value on $s.$ Suppose that $\psi(s) =s+a t+b t^2+O(t^3),$ with $a, b \in k'((s)).$ Then we have that $\psi(\alpha_0 e^{\alpha_1 t +\alpha_2 t^2})=$
\begin{eqnarray}\label{psiexponent}
&=&(\alpha_0+\alpha_0 '  (a t+b t^2)+\frac{\alpha_0 ''}{2}(at)^2)e^{(\alpha_1+\alpha_1' at)t+\alpha_2t^2}  +O(t^3) \nonumber \\
&=&\alpha_0 e^{\frac{\alpha_0'}{\alpha_0}at+ ( \frac{\alpha_0'}{\alpha_0}b+\frac{\alpha_0''}{2\alpha_0}a^2 -\frac{1}{2}(\frac{\alpha_0 '}{ \alpha_0}a)^2)t^2} e^{\alpha_1t+(\alpha_1' a+\alpha_2)t^2} +O(t^3) \nonumber \\
&=&\alpha_0e^{(\alpha_1+\frac{\alpha_0'}{\alpha_0}a)t + (\alpha_2+\alpha_1' a+\frac{\alpha_0'}{\alpha_0}b+\frac{\alpha_0''}{2\alpha_0}a^2 -\frac{1}{2}(\frac{\alpha_0 '}{ \alpha_0}a)^2 ) t^2}+O(t^3),
\end{eqnarray}
where for a function $f \in k'((s)),$ we let $f' \in k'((s))$ denote the derivative with respect to $s.$ More precisely, $f' \in k'((s))$ is the unique series such that $df=f'ds$ in $\hat{\Omega}^{1} _{k'((s))/k}.$ 

Then using (\ref{psiexponent}) and  (\ref{defomega}) the difference $\Omega(\psi(\tilde{p}), \psi(\hat{p}))-\Omega(\tilde{p},\hat{p})$ is  given by 

\begin{eqnarray*}
& &\sum _{\sigma \in S_{3}}(-1)^{\sigma} a\cdot \frac{\alpha_{0\sigma(1)}'}{\alpha_{0\sigma(1)}}\cdot  ( \tilde{\alpha}_{2\sigma(3)}-\hat{\alpha}_{2\sigma(3)} ) \cdot d \log (\alpha_{0\sigma(2)})\\
& &=\sum _{\sigma \in S_{3}}(-1)^{\sigma} a\cdot \frac{\alpha_{0\sigma(1)}'}{\alpha_{0\sigma(1)}}\cdot  ( \tilde{\alpha}_{2\sigma(3)}-\hat{\alpha}_{2\sigma(3)} ) \cdot \frac{\alpha_{0\sigma(2)} ' }{\alpha_{0\sigma(2)}}ds
=0.
\end{eqnarray*}

\end{proof}

\begin{corollary}\label{indepphi}
If $\overline{\chi}:\tilde{\pazocal{A}} \xrightarrow{\sim} \hat{\pazocal{A}}$  is any lifting of $\chi,$ and $\varphi, \varphi': \underline{\hat{\pazocal{A}}} \to \hat{\pazocal{A}}$ are two splittings of the canonical projection $\hat{\pazocal{A}} \to \underline{\hat{\pazocal{A}}}$ then 
$$
\omega_{\overline{\chi}, \varphi '} =\omega_{\overline{\chi}, \varphi }. 
$$
\end{corollary}

\begin{proof}
Note that $\omega_{\overline{\chi}, \varphi }(\tilde{p},\hat{p},\chi)= \Omega( \overline{\varphi}^{-1}( \overline{\chi}(\tilde{p})), \overline{\varphi}^{-1}(\hat{p}) )$ and $\omega_{\overline{\chi}, \varphi} ' (\tilde{p},\hat{p},\chi)= \Omega( \overline{\varphi} '^{-1}( \overline{\chi}(\tilde{p})), \overline{\varphi} '^{-1}(\hat{p}) ).$ Therefore, letting $\tilde{q}:= \overline{\varphi}^{-1}( \overline{\chi}(\tilde{p}))$ and $\hat{q}:=\overline{\varphi}^{-1}(\hat{p}) \in \underline{\hat{\pazocal{A}}}[[t]]$ and $\psi:= \overline{\varphi} '^{-1} \circ \overline{\varphi},$ we need to show that $\Omega (\tilde{q},\hat{q})=\Omega (\psi(\tilde{q}),\psi(\hat{q})).$ This is exactly the statement of the above lemma when $\underline{\hat{\pazocal{A}}} \simeq k'((s)).$ However, we can always reduce to this case by choosing a closed point of $\hat{\pazocal{A}}$ and noting that the  field of fractions of the completion of $\underline{\hat{\pazocal{A}}}$ at this closed point is isomorphic to $k'((s))$ and the induced map $\Omega^{1} _{\underline{\hat{\pazocal{A}}} /k} \to \Omega^{1} _{k'((s))/k} $ is an injection.
\end{proof}

Because of the corollary above, from now on we will use the notation $\omega_{\overline{\chi}}$ instead of $\omega_{\overline{\chi}, \varphi}.$ Next we show independence with respect to the lifting $\overline{\chi}$ of $\chi.$ 
 
\begin{lemma}
If $\overline{\chi}$ and $\overline{\chi}'$ are  liftings of $\chi$ then $\omega_{\overline{\chi}}=\omega_{\overline{\chi}'}.$ Therefore, $\omega$ is well-defined and does not depend on $\varphi$ and $\overline{\chi}.$
 \end{lemma}

\begin{proof} 
Using a splitting $\varphi: \underline{\hat{\pazocal{A}}}\to \hat{\pazocal{A}}$ as in the definition of $\omega,$ we  reduce to the case when $\tilde{\pazocal{A}}=\hat{\pazocal{A}}=\underline{\hat{\pazocal{A}}}[[t]]$ and $\chi=id.$ Using the argument at the end of the proof of Corollary \ref{indepphi}, we reduce further to the case when $\underline{\hat{\pazocal{A}}}=k'((s)).$ 
 If $\overline{\chi}$ and $\overline{\chi}'$ are two liftings of $id,$ we need to prove that: 
$$
\Omega(\overline{\chi}(\tilde{p}), \hat{p})=\Omega(\overline{\chi}'(\tilde{p}), \hat{p}).
$$
Letting $\overline{\chi}'=\iota \circ \overline{\chi},$ we have $\iota|_{t^2}=id$ and we need to prove that for any  $\tilde{p}$ and $\hat{p}$ such that $\tilde{p}|_{t^2}=\hat{p}|_{t^2},$ we have 
$$
\Omega(\iota(\tilde{p}),\hat{p})=\Omega(\tilde{p},\hat{p}).
$$
 We have $\iota(s)=s+bt^2+O(t^3),$ with $b \in k'((s)),$ and   by the formula (\ref{psiexponent}), we have 
$$
\iota(\alpha_0 e^{\alpha_1 t +\alpha_2 t^2})=\alpha_0e^{\alpha_1 t + (\alpha_2+\frac{\alpha_0'}{\alpha_0}b) t^2}+O(t^3).
$$
Then the difference $\Omega(\iota(\tilde{p}),\hat{p})-\Omega(\tilde{p},\hat{p})$ is given by 
\begin{eqnarray*}
& &\sum _{\sigma \in S_{3}}(-1)^{\sigma}  \alpha_{1 \sigma(1)} \cdot  \big( b \frac{\alpha_{0\sigma(3)}' }{\alpha_{0\sigma(3)}} \big) \cdot d \log (\alpha_{0\sigma(2)})=b\sum _{\sigma \in S_{3}}(-1)^{\sigma}  \alpha_{1 \sigma(1)} \cdot   \frac{\alpha_{0\sigma(3)}' }{\alpha_{0\sigma(3)}}  \frac{\alpha_{0\sigma(2)}'}{\alpha_{0\sigma(2)}}ds=0.  \\
\end{eqnarray*}
This finishes the proof of the lemma. 
\end{proof}

 \subsection{The additional properties (P4)-(P10) of the map $\omega$}

 In this section,  we will show that the construction above satisfies the following fundamental properties which will be needed below. 
 
 \subsubsection{(P4)  Interchange of components.} If we interchange $\tilde{p}$ and $\hat{p}$ then they are related by the following formula 

$$
\underline{\chi}_*(\omega(\hat{p},\tilde{p},\chi ^{-1}))=-\omega(\tilde{p},\hat{p},\chi),
$$
where $\underline{\chi}_{*}: \Omega^1 _{ \underline{\tilde{\pazocal{A}}} /k } \to  \Omega^1 _{ \underline{\hat{\pazocal{A}}} /k }$ is the  map induced by $\chi$.

\begin{proof} 
First note that $\Omega$ has the property that $\Omega(\tilde{q},\hat{q})=-\Omega(\hat{q},\tilde{q}).$ Let $\overline{\chi}:\tilde{\pazocal{A}} \xrightarrow{\sim} \hat{\pazocal{A}}$  be a lifting of $\chi,$ and $\varphi: \underline{\hat{\pazocal{A}}} \to \hat{\pazocal{A}}$ a splitting of the canonical projection 
as above. Then we have, by definition,  $
\omega(\tilde{p},\hat{p},\chi)= \Omega( \overline{\varphi}^{-1}( \overline{\chi}(\tilde{p})), \overline{\varphi}^{-1}(\hat{p}) ).
$

On the other hand, in order to compute $\omega(\hat{p},\tilde{p},\chi ^{-1}),$ we can use the  splitting  $ \psi:=\overline{\chi} ^{-1} \circ \varphi \circ \underline{\chi}$ from  $\underline{\tilde{\pazocal{A}}}$ to $\tilde{\pazocal{A}},$ and the lifting $\overline{\chi}^{-1}$ of $\chi^{-1}.$   Note that $\overline{\psi}$ can be described as follows. Let $\overline{\underline{\chi}}$ denote the isomorphism from $\underline{\tilde{\pazocal{A}}}[[t]] $ to $\underline{\hat{\pazocal{A}}}[[t]] ,$ induced by $\underline{\chi}: \underline{\tilde{\pazocal{A}}} \to \underline{\hat{\pazocal{A}}}.$ Then $\overline{\psi}= \overline{\chi}^{-1} \circ\overline{\varphi} \circ\overline{\underline{\chi}}$ and we have $\omega(\hat{p},\tilde{p},\chi ^{-1})=$
$$
\Omega ( \overline{\psi}^{-1} ( \overline{\chi}^{-1}(\hat{p})), \overline{\psi} ^{-1}(\tilde{p} ))=\Omega ( \overline{\underline{\chi}}^{-1} ( \overline{\varphi}^{-1}(\hat{p})), \overline{\underline{\chi}}^{-1}( \overline{\varphi}^{-1}( \overline{\chi}(\tilde{p})))).
$$
By transport of structure, $\underline{\chi}_{*}$ of the last expression is $\Omega ( \overline{\varphi}^{-1}(\hat{p}),  \overline{\varphi}^{-1}( \overline{\chi}(\tilde{p}))).$ Therefore, 
$$
\underline{\chi}_{*}(\omega(\hat{p},\tilde{p},\chi ^{-1}))=\Omega ( \overline{\varphi}^{-1}(\hat{p}),  \overline{\varphi}^{-1}( \overline{\chi}(\tilde{p})))=-\Omega ( \overline{\varphi}^{-1}( \overline{\chi}(\tilde{p})),\overline{\varphi}^{-1}(\hat{p}))=-\omega(\tilde{p},\hat{p},\chi),
$$
which finishes the proof. 
\end{proof}

 \subsubsection{ (P5) Transitivity.} If we have $\pazocal{A}_{1}, \pazocal{A}_{2}$ and $\pazocal{A}_{3}$ and $\chi_{12}: \pazocal{A}_{1}|_{t^2} \to \pazocal{A}_{2}|_{t^2}$
 and $\chi_{23}: \pazocal{A}_{2}|_{t^2} \to \pazocal{A}_{3}|_{t^2}$ then 
 
 $$
\underline{\chi}_{23 *}(\omega(p_1,p_2,\chi_{12}))+ \omega(p_2,p_3,\chi_{23})=\omega(p_1,p_3,\chi_{13}),
 $$
 where $\chi_{13}= \chi _{23} \circ \chi_{12}$
 
 \begin{proof}
 First note that $\Omega(q_{1},q_2)+\Omega(q_2.q_3)=\Omega(q_1,q_3).$ Let $\overline{\chi}_{12}:\pazocal{A}_1 \to \pazocal{A}_2$ is a lifting of $\chi_{12}$ and $\overline{\chi}_{23}:\pazocal{A}_2 \to \pazocal{A}_3$ is a lifting of $\chi_{23}.$ Then $\overline{\chi}_{13}:=\overline{\chi}_{23} \circ \overline{\chi}_{12}$ is a lifting of $\chi_{13}.$ Let $\varphi: \underline{\pazocal{A}}_{3}\to \pazocal{A}_{3}$ be  a splitting. Then we have 
 \begin{eqnarray}\label{trans1}
 \omega(p_2,p_3,\chi_{23})=\Omega(\overline{\varphi} ^{-1}( \overline{\chi}_{23}(p_2)),\overline{\varphi} ^{-1}(p_3)).
 \end{eqnarray}
 
 Note that $\varphi_{2}:=\overline{\chi}_{23} ^{-1} \circ \varphi \circ\underline{\chi}_{23}:\underline{\pazocal{A}}_{2} \to \pazocal{A}_{2}$ is a splitting with $\overline{\varphi}_{2}=\overline{\chi}_{23} ^{-1} \circ \overline{\varphi} \circ \overline{\underline{\chi}}_{23}.$ Then we have 
 $$
 \omega(p_1,p_2,\chi_{12})=\Omega(\overline{\varphi}_2 ^{-1}( \overline{\chi}_{12}(p_1)),\overline{\varphi}_2 ^{-1}(p_2))=\Omega(\overline{\underline{\chi}}_{23} ^{-1}(\overline{\varphi} ^{-1}( \overline{\chi}_{13}(p_1))),\overline{\underline{\chi}}_{23} ^{-1}(\overline{\varphi} ^{-1}( \overline{\chi}_{23}(p_2)))),
 $$
 which implies that 
  \begin{eqnarray}\label{trans2}
  \underline{\chi}_{23,*}( \omega(p_1,p_2,\chi_{12}))=\Omega(\overline{\varphi} ^{-1}( \overline{\chi}_{13}(p_1)),\overline{\varphi} ^{-1}( \overline{\chi}_{23}(p_2))).
  \end{eqnarray}
 Similarly, 
   \begin{eqnarray}\label{trans3}
 \omega(p_1,p_3,\chi_{13})=\Omega(\overline{\varphi} ^{-1}( \overline{\chi}_{13}(p_1)),\overline{\varphi} ^{-1}( p_3)).
  \end{eqnarray}
 Summing (\ref{trans1}), (\ref{trans2}), and (\ref{trans3}), gives the result.  
 \end{proof}

  \subsubsection{(P6) Multilinearity.} The map $\omega(\cdot, \cdot, \chi)$ is multilinear on each of the three  components of $\times ^{3}\beta(\tilde{\pazocal{A}}, \hat{\pazocal{A}},\chi).$ 
\begin{proof}
This follows  from the same statement for $\Omega.$ Namely, if $(\tilde{y}_{i},\hat{y}_{i}) \in \beta(\underline{\hat{\pazocal{A}}}[[t]], \underline{\hat{\pazocal{A}}}[[t]],id) ,$ and $(\tilde{y}_{i}',\hat{y}'_{i}) \in \beta(\underline{\hat{\pazocal{A}}}[[t]], \underline{\hat{\pazocal{A}}}[[t]],id) ,$ for $1\leq i \leq 3,$ then 
$$
\Omega ((\tilde{y}_{1}\tilde{y}_{1}',\tilde{y}_2,\tilde{y}_3), (\hat{y}_{1}\hat{y}_{1}',\hat{y}_2,\hat{y}_3))=\Omega ((\tilde{y}_{1},\tilde{y}_2,\tilde{y}_3), (\hat{y}_{1},\hat{y}_2,\hat{y}_3))+\Omega ((\tilde{y}_{1}',\tilde{y}_2,\tilde{y}_3), (\hat{y}_{1}',\hat{y}_2,\hat{y}_3)), 
$$
$$
\Omega ((\tilde{y}_1,\tilde{y}_{2}\tilde{y}_{2}',\tilde{y}_3), (\hat{y}_1,\hat{y}_{2}\hat{y}_{2}',\hat{y}_3))=\Omega ((\tilde{y}_{1},\tilde{y}_2,\tilde{y}_3), (\hat{y}_{1},\hat{y}_2,\hat{y}_3))+\Omega ((\tilde{y}_{1},\tilde{y}_2',\tilde{y}_3), (\hat{y}_{1},\hat{y}_2',\hat{y}_3)), 
$$
and
$$
\Omega ((\tilde{y}_1,\tilde{y}_2,\tilde{y}_{3}\tilde{y}_{3}'), (\hat{y}_1,\hat{y}_2,\hat{y}_{3}\hat{y}_{3}'))=\Omega ((\tilde{y}_{1},\tilde{y}_2,\tilde{y}_3), (\hat{y}_{1},\hat{y}_2,\hat{y}_3))+\Omega ((\tilde{y}_{1},\tilde{y}_2,\tilde{y}_3'), (\hat{y}_{1},\hat{y}_2,\hat{y}_3')).
$$

 \end{proof}
 
  \subsubsection{(P7) Anti-symmetry.} The map $\omega(\cdot,\cdot,\chi)$ is anti-symmetric with respect to the action of $S_3$ on  $\times ^{3}\beta(\tilde{\pazocal{A}}, \hat{\pazocal{A}},\chi).$ 
 
 \begin{proof}
This follows from the same statement for $\Omega.$ More precisely, if  $(\tilde{y}_{i},\hat{y}_{i}) \in \beta(\underline{\hat{\pazocal{A}}}[[t]], \underline{\hat{\pazocal{A}}}[[t]],id) ,$ for $1\leq i\leq 3,$ then 
$$
\Omega ((\tilde{y}_{\sigma(1)},\tilde{y}_{\sigma(2)},\tilde{y}_{\sigma(3)}), (\hat{y}_{\sigma(1)},\hat{y}_{\sigma(2)},\hat{y}_{\sigma(3)}))=(-1)^{\sigma}\Omega ((\tilde{y}_{1},\tilde{y}_2,\tilde{y}_3), (\hat{y}_{1},\hat{y}_2,\hat{y}_3)),
$$
for any $\sigma \in S_{3}.$
 \end{proof}
 
\begin{corollary}
The properties (P6) and (P7) imply that $\omega(\cdot,\cdot,\chi) $ induces a map 
$$\omega(\cdot,\cdot,\chi):\Lambda^3 _{\mathbb{Z}} \beta(\tilde{\pazocal{A}},\hat{\pazocal{A}},\chi)\to \Omega^{1}_{\underline{\hat{\pazocal{A}}}/k},$$ which we will denote by the same notation.  
\end{corollary}
 
  \subsubsection{(P8)  Residue.}\label{residue} Suppose that $\tilde{\pazocal{A}}^{\circ}$ and $\hat{\pazocal{A}}^{\circ}$ are objects of $\mathcal{C}^{1} _{A}$ whose reductions $\underline{\tilde{\pazocal{A}}}^{\circ}$ and $\underline{\hat{\pazocal{A}}}^{\circ}$ are discrete valuation rings.    Let   $\tilde{\pazocal{A}}$ and $\hat{\pazocal{A}}$ be their localisations at the prime ideal $(t).$ Suppose that $\chi:\tilde{\pazocal{A}}^{\circ}|_{t^2} \xrightarrow{\sim}  \hat{\pazocal{A}}^{\circ} |_{t^2}.$ Note that this gives, after localisation, an isomorphism $\chi: \tilde{\pazocal{A}}|_{t^2} \xrightarrow{\sim}  \hat{\pazocal{A}}|_{t^2}.$ Suppose that    $\hat{\pi} \in \hat{\pazocal{A}}^\circ$ (resp. $\tilde{\pi} \in \tilde{\pazocal{A}}^\circ$),   such that its  reduction to   $\underline{\hat{\pazocal{A}}}^\circ$  (resp. $\underline{\tilde{\pazocal{A}}}^\circ$)  is a uniformizer.  Assume that $\chi(\hat{\pi}|_{t^2})=\tilde{\pi}|_{t^2}$ in   $\hat{\pazocal{A}}^{\circ} |_{t^2}.$ 

\begin{definition}\label{gooddefn}
 We let $\beta(\hat{\pazocal{A}}^{\circ},\hat{\pi}):=\{\hat{y} \in \hat{\pazocal{A}} ^{\times}  | \hat{y}=\hat{\pi} ^{a}u, \; {\rm for \; some \; } a \in \mathbb{Z},  \; {\rm  and\;} u \in  (\hat{\pazocal{A}} ^{\circ})^{\times}\}$ and $\times ^n \beta(\hat{\pazocal{A}}^{\circ},\hat{\pi}),$ the $n$-fold product of $\beta(\hat{\pazocal{A}}^{\circ},\hat{\pi})$ with itself. We say that $\hat{y} \in \hat{\pazocal{A}}  $ is  $\hat{\pi}$-{\it good} if $\hat{y} \in \beta(\hat{\pazocal{A}}^{\circ},\hat{\pi}).$
\end{definition}
Note that $\times ^n \beta(\hat{\pazocal{A}}^{\circ},\hat{\pi})  \subseteq  \times ^n \hat{\pazocal{A}} ^{\times}.$ We let $\times ^n\beta(\tilde{\pazocal{A}} ^{\circ},\hat{\pazocal{A}}^{\circ} ,\chi,\tilde{\pi},\hat{\pi} )  \subseteq \times ^n \beta(\tilde{\pazocal{A}},\hat{\pazocal{A}} ,\chi)$ denote the subgroup of those pairs of $n$-tuples which lie  in $ \times ^n \beta (\tilde{\pazocal{A}}^{\circ} ,\tilde{\pi}) $ and $\times ^n \beta(\hat{\pazocal{A}}^{\circ} ,\hat{\pi}).$  We claim that on this set, the 1-form $\omega$  has only logarithmic singularities at the closed point of $ {\rm Spec} \, \underline{\hat{\pazocal{A}}}^\circ $: 
$$
 \omega : \times ^3 \beta(\tilde{\pazocal{A}}^{\circ},\hat{\pazocal{A}}^{\circ},\chi,\tilde{\pi},\hat{\pi} )  \to \Omega^1 _{ \underline{\hat{\pazocal{A}}}_{\log}^\circ  /k }   \subseteq \Omega^1 _{ \underline{\hat{\pazocal{A}}} /k },
$$
where, if $\underline{\hat{\pi}}$ denotes the reduction of $\hat{\pi}$ to $\underline{\hat{\pazocal{A}}},$ then $\Omega^1 _{ \underline{\hat{\pazocal{A}}}_{\log}^\circ  /k }:=\big(\Omega^1 _{ \underline{\hat{\pazocal{A}}}^\circ  /k }  +   \underline{\hat{\pazocal{A}}}^\circ \cdot d \log (\underline{\hat{\pi}})\big) \subseteq  \Omega^1 _{ \underline{\hat{\pazocal{A}}} /k}.$  In the remainder of this section, we will prove the above statement and give a formula for the residue of $\omega.$ 

First note that  we have the following  version:
\begin{eqnarray}\label{goncres}
res_{\hat{\pi}}:  \Lambda ^{n} \beta(\hat{\pazocal{A}} ^{\circ},\hat{\pi} ) \to \Lambda^{n-1} (\hat{\pazocal{A}}^{\circ}/(\hat{\pi}))^{\times}, 
\end{eqnarray}
of the residue map in \cite[\textsection 1.14]{gon2}. 
This map is characterised by the following properties. It is  multilinear, anti-symmetric, vanishes when restricted to $\Lambda^n (\hat{\pazocal{A}} ^{\circ})^{\times};$  and 
$$
res_{\hat{\pi}}(\hat{\pi}, u_{2}, \cdots, u_{n})=(\underline{u}_{2},\cdots, \underline{u}_{n}),
$$
where if $u_i \in (\hat{\pazocal{A}} ^{\circ})^{\times} , $
 for $2 \leq i \leq n,$ then $\underline{u}_i$ denotes  reduction modulo $(\hat{\pi}).$  We will need the fact  that this construction of $res_{\hat{\pi}}$ in fact  depends only on the ideal $(\hat{\pi})$ generated by $\hat{\pi}$ in  $\hat{\pazocal{A}}^{\circ}.$

\begin{lemma}
If  $\hat{\pi}'=v \cdot \hat{\pi},$ for some  unit  $v \in (\hat{\pazocal{A}} ^{\circ})^{\times},$ then $res_{\hat{\pi}'}=res_{\hat{\pi}}.$ 
\end{lemma}

\begin{proof} 
First note that under the assumptions $\beta(\hat{\pazocal{A}} ^{\circ},\hat{\pi}')=\beta(\hat{\pazocal{A}} ^{\circ},\hat{\pi})$ and hence the sources and the targets of  $res_{\hat{\pi}'}$ and $res_{\hat{\pi}}$ are the same and  the statement makes sense. In order to prove the statement, by the multi-linearity, anti-symmetry and the vanishing of both  $res_{\hat{\pi}'}$ and $res_{\hat{\pi}}$  on $\Lambda^n(\hat{\pazocal{A}} ^{\circ})^{\times},$ we need to check that they agree on elements of the form $(\hat{\pi}',u_2,\cdots, u_{n})$ with $u_i \in (\hat{\pazocal{A}} ^{\circ})^{\times}.$ On this element, we have $res_{\hat{\pi}'}(\hat{\pi}',u_2,\cdots, u_{n})=(\underline{u}_{2},\cdots, \underline{u}_{n})$ and $res_{\hat{\pi}}(\hat{\pi}',u_2,\cdots, u_{n})=res_{\hat{\pi}}(v\cdot \hat{\pi},u_2,\cdots, u_{n})=res_{\hat{\pi}}(v,u_2,\cdots, u_{n})+ res_{\hat{\pi}}( \hat{\pi},u_2,\cdots, u_{n})  =(\underline{u}_{2},\cdots, \underline{u}_{n}),$ which finishes the proof.
\end{proof}

Let $\mathcal{C}^{0} _{A}$ denote the full subcategory of the category of $A$-algebras whose objects consist of \'{e}tale $A$-algebras which are integral domains with the residue field of the closed point a finite extension $k'$ of $k.$   Given such an $A$-algebra $B,$ the completion $\hat{B}$ along the ideal $(t)$ is canonically isomorphic to $k'[[t]],$  where $k'$ is the residue field of the closed point, so we write $\psi:\hat{B}\xrightarrow{\sim} k'[[t]].$ Let $[t^2\wedge t]: \Lambda ^{2}  _{k'} k'[[t]] \to k'$ denote the map that sends $a\wedge b$ to the coefficient of $t^2\wedge t$ in $a\wedge b.$

We have a canonical map 
\[
\begin{CD}
\ell:\Lambda ^{2}  _{\mathbb{Z}}B^{\times} \to \Lambda ^{2}  _{\mathbb{Z}}\hat{B}^{\times} @>{\Lambda^{2} \psi}>>\Lambda ^{2}  _{\mathbb{Z}} k'[[t]]^{\times} @>{\Lambda ^{2}  \log ^{\circ}  }>> \Lambda ^{2}  _{k'} k'[[t]] @>{[t^2\wedge t]}>>k',
\end{CD}
\]
for every object $B$ of $\mathcal{C}^{0} _{A}.$ 
 With the above notation,   $\hat{\pazocal{A}}^{\circ}/(\hat{\pi})$ is an object of $\mathcal{C}^{0} _{A},$ and we have a canonical map: 
\begin{eqnarray}\label{ellmap}
\ell: \Lambda^2 _{\mathbb{Z}}(\hat{\pazocal{A}}^{\circ}/(\hat{\pi}))^{\times}  \to k':=\hat{\pazocal{A}}^{\circ}/(\hat{\pi},t ) .
\end{eqnarray}

\begin{proposition}
If $(\tilde{p},\hat{p},\chi) \in  \Lambda ^3 \beta(\tilde{\pazocal{A}}^{\circ},\hat{\pazocal{A}}^{\circ},\chi,\tilde{\pi},\hat{\pi} ) ,$ then  we have $\omega(\tilde{p},\hat{p},\chi) \in \Omega^1 _{ \underline{\hat{\pazocal{A}}}_{\log}^\circ  /k } $ with  residue given by 

\begin{eqnarray}\label{resform}
res_{\underline{\hat{\pi}}} \omega(\tilde{p},\hat{p},\chi) = \underline{\chi} \circ\ell\circ res_{\tilde{\pi}} (\tilde{p}) -\ell\circ res_{\hat{\pi}} (\hat{p}).
\end{eqnarray}
\end{proposition}

\begin{proof}

First, without loss of generality, we  replace $\hat{\pazocal{A}}^{\circ}$ and $\tilde{\pazocal{A}}^{\circ}$  with their completions at their closed point corresponding to $(\hat{\pi},t)$ and $(\tilde{\pi},t).$  Then we choose $\overline{\chi}:\tilde{\pazocal{A}}^{\circ} \to \hat{\pazocal{A}}^{\circ}$ as the unique  lifting of $\chi$ which sends $\tilde{\pi}$ to $\hat{\pi}$ and is an $A$-algebra isomorphism.  Choosing also an isomorphism $k'[[s]] \simeq \underline{\hat{\pazocal{A}}}^{\circ},$ which sends $s$ to $\underline{\hat{\pi}},$ we reduce the problem to the case when 
$\hat{\pazocal{A}}^{\circ}=\tilde{\pazocal{A}}^{\circ}=k'[[s,t]],$ $\chi=id$ and $\hat{\pi}=\tilde{\pi}=s.$

Let  $\tilde{p}=(\tilde{y}_1,\tilde{y}_{2},\tilde{y}_{3})$ and  $\hat{p}=(\hat{y}_1,\hat{y}_{2},\hat{y}_{3}).$ Then $\omega(\tilde{p},\hat{p},\chi)=\Omega(\tilde{p},\hat{p}).$ If  $\tilde{y}_i,\hat{y}_{i}\in k'[[s,t]]^{\times},$ for all $1 \leq i \leq 3$ then $\Omega(\tilde{p},\hat{p}) \in \Omega^{1} _{k'[[s]]/k}$ and hence $res_{s} \Omega(\tilde{p},\hat{p}) =0.$ In this case, we also have $res_{s}\tilde{p}=res_{s}\hat{p}=0 \in \Lambda^{2} k[[s]]^{\times}. $ Therefore, we have the equality in the (\ref{resform}) in this case.  

By the above property, the multi-linearity and  anti-symmetry of  both sides of (\ref{resform}), and the definition of $s$-goodness,  we are reduced to proving the equality in case when $\tilde{p}=(s ,\tilde{y}_{2},\tilde{y}_{3})$ and  $\hat{p}=(s ,\hat{y}_{2},\hat{y}_{3}),$ with $\tilde{y}_{i},\hat{y}_{i} \in k'[[s,t]] ^{\times}.$ We can then write  $\hat{y}_i=\alpha_{0i}e^{t\alpha_{1i}+t^2\hat{\alpha}_{2i}+\cdots }$ and $\tilde{y}_i=\alpha_{0i}e^{t\alpha_{1i}+t^2\tilde{\alpha}_{2i}+\cdots },$ with $\alpha_{1i}, \hat{\alpha}_{ri},\tilde{\alpha}_{ri} \in k'[[s]],$ for $2 \leq r,$ and $2\leq i \leq 3; $ and $\alpha_{0i} \in k[[s]] ^{\times},$ for $2\leq i \leq 3.$    Then we have $
res_{s}\Omega(\tilde{p},\hat{p}) = \alpha_{1 3}(0) ( \tilde{\alpha}_{22} (0)-\hat{\alpha}_{22}(0) )  -\alpha_{1 2}(0) ( \tilde{\alpha}_{23}(0)-\hat{\alpha}_{23}(0) ) 
$

$$
=(\alpha_{1 3} (0) \tilde{\alpha}_{22} (0)-\alpha_{1 2} (0)  \tilde{\alpha}_{23} (0))-( \alpha_{1 3} (0)\hat{\alpha}_{22}(0) -\alpha_{1 2}(0)  \hat{\alpha}_{23}(0) ).
$$
On the other hand, $res_{s}(\tilde{p})=\big(\alpha_{02}(0) e^{t\alpha_{12}(0)+t^2\tilde{\alpha}_{22}(0)+\cdots }\big) \wedge \big(\alpha_{03}(0) e^{t\alpha_{13}(0)+t^2\tilde{\alpha}_{23}(0)+\cdots }\big) \in \Lambda ^{2} _{\mathbb{Z}} k'[[t]] ^{\times}, $ and hence $(\Lambda^{2} \log ^{\circ})(res_{s}(\tilde{p}))=(t\alpha_{12}(0)+t^2\tilde{\alpha}_{22}(0)+\cdots ) \wedge (t\alpha_{13}(0)+t^2\tilde{\alpha}_{23}(0)+\cdots) .$ This implies that $\ell(res_{s}(\tilde{p}))= \tilde{\alpha}_{22} (0) \alpha_{1 3} (0)-\alpha_{1 2} (0)  \tilde{\alpha}_{23} (0).$ Similarly, $\ell(res_{s}(\hat{p}))= \hat{\alpha}_{22} (0) \alpha_{1 3} (0)-\alpha_{1 2} (0)  \hat{\alpha}_{23} (0).$ Therefore, we have 
$$
res_{s}\Omega(\tilde{p},\hat{p})=\ell(res_{s}(\tilde{p}))-\ell(res_{s}(\hat{p})),
$$
 which finishes the proof of the proposition. 
\end{proof}
 
 \subsubsection{(P9) Difference relation.} Suppose that we have objects $\tilde{\pazocal{A}}_{i}$ and $\hat{\pazocal{A}}_{i}$  in $\mathcal{C}^{1} _{A}, $ for $i=1,2,$ as above, together with isomorphisms $\alpha_i:\tilde{\pazocal{A}}_{i}/(t^2) \xrightarrow{\sim}\hat{\pazocal{A}}_{i}/(t^2), $ for $i=1,2,$ and $\chi:\tilde{\pazocal{A}}_{2}/(t^2) \xrightarrow{\sim}\tilde{\pazocal{A}}_{1}/(t^2), $  and $\psi:\hat{\pazocal{A}}_{2}/(t^2) \xrightarrow{\sim}\hat{\pazocal{A}}_{1}/(t^2), $ with $\alpha_1\circ\chi=\psi\circ\alpha_2.$ Moreover, suppose that $(\tilde{p}_2,\tilde{p}_{1}) \in \times ^3\beta(\tilde{\pazocal{A}}_{2},\tilde{\pazocal{A}}_{1},\chi ),$
   $(\hat{p}_2,\hat{p}_{1}) \in \times ^3 \beta(\hat{\pazocal{A}}_{2},\hat{\pazocal{A}}_{1},\psi),$ and  $(\tilde{p}_i,\hat{p}_{i}) \in \times ^3 \beta(\tilde{\pazocal{A}}_{i},\hat{\pazocal{A}}_{i},\alpha_i ),$ for $i=1,2.$ Then we have

 $$
\underline{\alpha}_{1*}(\omega(\tilde{p}_{2},\tilde{p}_{1}, \chi))- \omega(\hat{p}_{2},\hat{p}_{1}, \psi)=\underline{\psi}_{*}(\omega(\tilde{p}_{2},\hat{p}_{2}, \alpha_2))-\omega(\tilde{p}_{1},\hat{p}_{1}, \alpha_1).
$$

\begin{proof} 
In order prove the identity above, first, without loss of generality,  we  replace all the rings above with their completions at a closed point in their special fibers such that these points correspond to each other under the given isomorphisms $\alpha_1, \,\alpha_2, \, \chi$ and $\psi.$ Choose liftings, $\overline{\alpha}_{i}$ and $\overline{\chi}.$ Since $\alpha_1 \circ \chi=\psi \circ \alpha_2,$ $\overline{\alpha}_{1} \circ \overline{\chi} \circ \overline{\alpha}_2^{-1}$ is a lifting of $\psi,$  and we can choose $\overline{\psi}:=\overline{\alpha}_{1} \circ \overline{\chi} \circ \overline{\alpha}_2^{-1}$ as the lifting of $\psi.$ With this choice, we have $\overline{\alpha}_1 \circ \overline{\chi}=\overline{\psi} \circ \overline{\alpha}_2.$

Let $\varphi: \underline{\hat{\pazocal{A}}}_{1} \to \hat{\pazocal{A}}_{1}$ be a splitting. Then $\overline{\psi} ^{-1}\circ\varphi \circ\underline{\psi}:\underline{\hat{\pazocal{A}}}_{2} \to \hat{\pazocal{A}}_{2}$ and $\overline{\alpha}_1 ^{-1}\circ\varphi \circ\underline{\alpha}_1:\underline{\tilde{\pazocal{A}}}_{1} \to \tilde{\pazocal{A}}_{1},$ with the corresponding  isomorphisms given by 
$$
\overline{\psi} ^{-1}\circ \overline{\varphi} \circ \overline{\underline{\psi}}:\underline{\hat{\pazocal{A}}}_{2}[[t]] \to \hat{\pazocal{A}}_{2} \;\;\;{\rm and} \;\;\; \overline{\alpha}_1 ^{-1}\circ\overline{\varphi} \circ\overline{\underline{\alpha}}_1:\underline{\tilde{\pazocal{A}}}_{1}[[t]] \to \tilde{\pazocal{A}}_{1}.
$$
By the definition of $\omega,$ we have 
$$
\omega(\tilde{p}_{1},\hat{p}_{1}, \alpha_1)=\Omega(\overline{\varphi}^{-1}(\overline{\alpha}_1(\tilde{p}_1)),\overline{\varphi}^{-1}(\hat{p}_{1})),
$$
$$
\omega(\hat{p}_{2},\hat{p}_{1}, \psi)=\Omega(\overline{\varphi}^{-1}(\overline{\psi}(\hat{p}_2)),\overline{\varphi}^{-1}(\hat{p}_{1})),
$$
$$
\omega(\tilde{p}_{2},\hat{p}_{2}, \alpha_2)=\Omega( (\overline{\psi} ^{-1}\circ \overline{\varphi} \circ \overline{\underline{\psi}})^{-1}(\overline{\alpha}_2(\tilde{p}_{2})) , (\overline{\psi} ^{-1}\circ \overline{\varphi} \circ \overline{\underline{\psi}})^{-1}(\hat{p}_{2}) ),
$$
hence 
$$
\underline{\psi}_{*}(\omega(\tilde{p}_{2},\hat{p}_{2}, \alpha_2))=\Omega( (\overline{\psi} ^{-1}\circ \overline{\varphi})^{-1}(\overline{\alpha}_2(\tilde{p}_{2})) , (\overline{\psi} ^{-1}\circ \overline{\varphi} )^{-1}(\hat{p}_{2}) )=\Omega(\overline{\varphi}^{-1}(\overline{\psi}(\overline{\alpha}_{2}(\tilde{p}_2))),\overline{\varphi}^{-1}(\overline{\psi}(\hat{p}_2)));
$$
and 
$$
\omega(\tilde{p}_{2},\tilde{p}_{1}, \chi)=\Omega( (\overline{\alpha}_1 ^{-1}\circ\overline{\varphi} \circ\overline{\underline{\alpha}}_1) ^{-1}(\overline{\chi}(\tilde{p}_2)) ,(\overline{\alpha}_1 ^{-1}\circ\overline{\varphi} \circ\overline{\underline{\alpha}}_1) ^{-1}(\tilde{p}_1)),
$$
hence 
$$
\underline{\alpha}_{1*}(\omega(\tilde{p}_{2},\tilde{p}_{1}, \chi))=\Omega( (\overline{\alpha}_1 ^{-1}\circ\overline{\varphi} ) ^{-1}(\overline{\chi}(\tilde{p}_2)) ,(\overline{\alpha}_1 ^{-1}\circ\overline{\varphi} ) ^{-1}(\tilde{p}_1))=\Omega(\overline{\varphi}^{-1}(\overline{\psi}(\overline{\alpha}_{2}(\tilde{p}_2))),\overline{\varphi}^{-1}(\overline{\alpha}_1(\tilde{p}_1))).
$$
The above expressions, together with the basic identity, 
$$
\Omega(\tilde{q}_2,\tilde{q}_1)-\Omega(\hat{q}_2,\hat{q}_1)=\Omega(\tilde{q}_2,\hat{q}_2)-\Omega(\tilde{q}_1,\hat{q}_{1}),
$$
finishes the proof of the property above. 
\end{proof}

 \subsubsection{(P10) Modulus.} The map $\omega(\cdot,\cdot,\chi)$ depends on $\tilde{p}$ and $\hat{p}$ only  modulo $(t^3).$ 
 
\begin{proof}
The statement follows from the same statement for $\Omega,$ which is seen directly from its definition. 
\end{proof}

 \subsubsection{(P11)  Relative Bloch group relation.} Suppose that $B$ is a local ring with infinite residue field and $I  \subsetneq B$ an ideal, we will define a relative Bloch group of $(B,I).$  Let us write $\beta(B,I):=\{(\tilde{b},\hat{b})| \hat{b}-\tilde{b} \in I \;{\rm and } \; \hat{b},\, \tilde{b} \in B^{\times} \},$ and $B^{\flat }:=\{b \in B| b\cdot (1-b) \in B^{\times} \}.$ Let $\beta^{\flat}(B,I)$ denote the intersection $(B^{\flat} \times B^{\flat} ) \cap \beta(B,I)$ inside $B ^{\times}\times B^{\times}.$ 
  
  Let us define the relative Bloch group $B_{2}(B,I)$ as the abelian group  generated by the symbols $\{ (\tilde{b},\hat{b})\}_{2}$ for every $(\tilde{b},\hat{b}) \in \beta^{\flat}(B,I),$ modulo the relations generated by the analog of the five term relation for the dilogarithm: 
$$
\{ (\tilde{x},\hat{x})\}_{2}-\{ (\tilde{y},\hat{y})\}_{2}+\{ (\tilde{y}/\tilde{x},\hat{y}/\hat{x})\}_{2}-\{ (\frac{1-\tilde{x}^{-1}}{1-\tilde{y}^{-1}},\frac{1-\hat{x}^{-1}}{1-\hat{y}^{-1}})\}_{2}+\{ (\frac{1-\tilde{x}}{1-\tilde{y}},\frac{1-\hat{x}}{1-\hat{y}})\}_{2}
$$
for every $(\tilde{x},\hat{x}),(\tilde{y},\hat{y}) \in \beta^{\flat}(B,I)$ such that $(\tilde{x}-\tilde{y},\hat{x}-\hat{y})\in \beta(B,I).$ As in the classical case, we obtain a complex 
$$
\delta: B_{2} (B,I) \to \Lambda^{2} \beta(B,I),
$$
which sends $(\tilde{b},\hat{b})$ to $(1-\tilde{b},1-\hat{b})\wedge (\tilde{b},\hat{b}).$  This map induces the following map: 
$$
\Delta: B_{2} (B,I) \otimes \beta(B,I) \to \Lambda^{3} \beta(B,I),
$$
which sends $\{ (\tilde{x},\hat{x})\}_{2} \otimes (\tilde{y},\hat{y})$ to $\delta(\{ (\tilde{x},\hat{x})\}_{2})\wedge (\tilde{y},\hat{y}).$

Let $B:=k'((s))[[t]],$ and $I=(t^2),$ then the two above definitions of $\beta$ are related as follows:  $\beta(B,(t^2))=\beta(B,B,id),$ where $id:B/(t^2)\to B/(t^2)$ is the identity map. In particular, we have a map $\omega:\Lambda^{3} \beta(B,(t^2))\to \Omega^{1} _{\underline{B}/k}.$ 

\begin{proposition}
With $B:=k'((s))[[t]],$ $\omega$ vanishes on  $\Delta(B_{2} (B,(t^2)) \otimes \beta(B,t^2))  \subseteq \Lambda^{3} \beta(B,(t^2)).$ 
\end{proposition}

\begin{proof} 
We need to show that $\omega$ vanishes on terms such as 
 $$((1-\tilde{f},\tilde{f},\tilde{g}),(1-\hat{f},\hat{f},\hat{g})),$$ with $(\tilde{f},\hat{f}) \in \beta^{\flat}(B,(t^2))$  and $(\tilde{g},\hat{g}) \in \beta(B,(t^2)).$

By the transitivity property (P5), it is enough to show that $\omega$ vanishes on the following terms:

\begin{eqnarray}\label{blochgrp1}
((1-\tilde{f},\tilde{f},\tilde{g}),(1-\tilde{f},\tilde{f},\hat{g}))
\end{eqnarray}

and

\begin{eqnarray}\label{blochgrp2}
((1-\tilde{f},\tilde{f},\hat{g}),(1-\hat{f},\hat{f},\hat{g})).
\end{eqnarray}

In order to  prove that $\omega$ vanishes on (\ref{blochgrp1}), let $\tilde{f}=\alpha_{0} e^{\alpha_1 t+\tilde{\alpha}_{2} t^2+\cdots},$  $\tilde{g}:=\beta_{0} e^{\beta_1 t+\tilde{\beta}_{2} t^2+\cdots}$ and $\hat{g}:=\beta_{0} e^{\beta_1 t+\hat{\beta}_{2} t^2+\cdots}.$ Then we have 
$$
1-\tilde{f}=(1-\alpha_0)e^{\frac{\alpha_0\alpha_1}{\alpha_0-1}t + (\frac{\alpha_0 \tilde{\alpha}_2}{\alpha_0-1}-\frac{\alpha_0\alpha_1 ^2}{2(\alpha_0-1)^2}   )t^2+\cdots }.
$$
Then property (P1) gives that $\omega$ evaluated on (\ref{blochgrp1}) is equal to 
$$
\Big( \frac{\alpha_0\alpha_1}{\alpha_0-1}\cdot d \log(\alpha_0)-\alpha_1 \cdot d \log (1-\alpha_0) \Big)\cdot (\tilde{\beta}_2-\hat{\beta}_{2})=0.
$$

In order  prove that $\omega$ vanishes on (\ref{blochgrp2}), let $\hat{f}=\alpha_{0} e^{\alpha_1 t+\hat{\alpha}_{2} t^2+\cdots}.$ Again by (P1), we compute that $\omega$ evaluated on (\ref{blochgrp2}) is equal to 

\begin{eqnarray*}
& &\Big( \frac{\alpha_0}{\alpha_0-1}(\tilde{\alpha}_2-\hat{\alpha}_2)\alpha_1- \frac{\alpha_0\alpha_1}{\alpha_0-1} (\tilde{\alpha}_2-\hat{\alpha}_2)       \Big) d\log(\beta_0)\\
& &+\Big(d \log (1-\alpha_0)(\tilde{\alpha}_2-\hat{\alpha}_2)- \frac{\alpha_0}{\alpha_0-1}(\tilde{\alpha}_2-\hat{\alpha}_2)d \log (\alpha_0)    \Big)\beta_1=0.
\end{eqnarray*}
 This finishes the proof.
 \end{proof}

\begin{corollary}\label{relblochom}
Suppose that $\pazocal{A}$ is an object of $\mathcal{C}^{1} _{A},$ and $(\tilde{p},\hat{p})=((1-\tilde{f},\tilde{f},\tilde{g}),(1-\hat{f},\hat{f},\hat{g}))$ with $(\tilde{p},\hat{p}) \in \times ^{3}\beta(\pazocal{A},\pazocal{A},id).$ Then 
$\omega(\tilde{p},\hat{p},id)=0.$ 
\end{corollary}

\begin{proof}
This follows from the above proposition after completing at a closed point in the special fiber and then localizing at the prime ideal $(t).$ 
\end{proof}

\section{Inifinitesimal regulator and the strong reciprocity theorem} 

In this section, we will give the construction of $\rho(f\wedge g \wedge h),$ where $f,g$ and $h$ are rational functions on a smooth proper curve $C_{2} $ over $k[t]/(t^2).$ The construction will rely   on the map $\omega$ of the previous section. In the first subsection,  we will give a local definition which depends on certain choices. In the next subsection, we will define the global version which is independent of all the choices. 

\subsection{Construction  of $\varepsilon$} In this subsection we will define a map $\varepsilon$ which will be defined on the triples of functions  on a generic lifting of the ring under consideration.

 \subsubsection{Notion of goodness.}\label{section goodness}  Let $S:={\rm Spec}(A),$ $S_{n}:={\rm Spec} (A_{n}),$ with $A_{n}:=k[[t]]/(t^n).$ Let $\mathcal{C}^{1} _{A_2}$ denote the following full subcategory of the category of $A_{2}$-algebras. The objects of  $\mathcal{C}^{1} _{A_2}$ consist of $A_{2}$-algebras $\pazocal{A}_{2}$ such that $\pazocal{A}_{2}$ is a  smooth $A_{2}$-algebra of  relative dimension 0 or 1, $\underline{\pazocal{A}}_{2}$ is an integral domain, the natural multiplication by $t$ map induces an isomorphism $\underline{\pazocal{A}}_{2} \xrightarrow{\sim} (t)$ of $\pazocal{A}_{2}$-modules and the residue fields of the closed points of $\underline{\pazocal{A}}_{2}$ are one of the following type: a finite extension of $k,$ if the relative dimension is 1; a finitely generated field of transcendence degree 1 over $k,$ or a field of Laurent series $k'((s))$ over a finite extension $k'$ of $k,$ if the relative dimension is 0.  Note that these algebras $\pazocal{A}_{2}$ are precisely the algebras that are obtained as $\pazocal{A}/(t^2)$ with $\pazocal{A}$ an object of $\mathcal{C}^{1} _{A}.$ 

In the rest of this section, assume that $\pazocal{A}_{2} ^{\circ}$ is an object of $\mathcal{C}^{1} _{A_{2}} $ such that $\underline{\pazocal{A}}_{2} ^{\circ}$ is a discrete valuation ring. We call  an element $\pi_2 \in \pazocal{A}_{2} ^{\circ},$  {\it a uniformizer}, if its reduction $\pi_1 \in \underline{\pazocal{A}}_{2} ^{\circ},$
 is a uniformizer,  in the usual sense.   This is equivalent to requiring that the natural map from $k'[[\pi_2,t]]/(t^2)$ to the completion of $\pazocal{A}_2 ^{\circ}$ at the maximal ideal being an isomorphism, where $k'$ is the residue field of the closed point. Let $\pazocal{A}_{2}$ denote the localization of $\pazocal{A}_{2} ^{\circ}$ at the prime ideal $(t),$ similar to the notation in \textsection \ref{residue}. If $B$ is a general $A$-algebra with $\underline{B}=B/(t)$ an integral domain, we also use the notation $B_{\eta}$ to denote the localisation of $B$ at the prime ideal $(t).$  
 
 \begin{definition}\label{gooddefn2}
 We let $\beta(\pazocal{A}^{\circ}_2,\pi_2):=\{y \in \pazocal{A}^{\times} _2  | y=\pi_2 ^{a}u, \; {\rm for \; some \; } a \in \mathbb{Z},  \; {\rm  and\;} u \in  (\pazocal{A}_2 ^{\circ})^{\times}\}$ and  say that $y \in \pazocal{A}_2  $ is  $\pi_2$-{\it good} if $y \in \beta(\pazocal{A} _{2}^{\circ},\pi_2).$
\end{definition}
 Note that if $y \in (\pazocal{A}_{2} ^{\circ})^{\times},$ then it is $\pi_{2}$-good for any choice of a uniformizer $\pi_{2}.$ On the other hand, if for example $\pazocal{A}_{2} ^{\circ}=k[[s,t]]/(t^{2})$ and $f=(s+t)/s \in \pazocal{A}_{2}=k((s))[t]/(t^2)$ then $f$ is not good for any choice of a uniformizer.

\subsubsection{Fixed generic lifting.}\label{fixedgenericlifting} Let $\tilde{\pazocal{A}}$ be an object in $\mathcal{C}^{1} _{A}$ together with  a fixed isomorphism 
$$\alpha:\tilde{\pazocal{A}}_2:=\tilde{\pazocal{A}}/(t^{2}) \xrightarrow{\sim} \pazocal{A}_{2} .$$ 
Assume further that $\tilde{\pazocal{A}}$ is complete with respect to the ideal $(t).$ Note that $\tilde{\pazocal{A}}$ is a lifting of $\pazocal{A}_{2}$ only and we do {\it not} assume that $\tilde{\pazocal{A}}$ is the localisation of a lifting of $\pazocal{A}^{\circ} _{2}.$ Let now 
$$
\beta(\tilde{\pazocal{A}},\alpha^{-1}(\pi_2)):= \{ y \in \tilde{\pazocal{A}} ^{\times }| \alpha(y|_{t^{2}}) \in \beta (\pazocal{A}_{2} ^{\circ},\pi_2) \},
$$
denote the elements in $\tilde{\pazocal{A}}$ whose reductions modulo $(t^2)$ are $\alpha^{-1}(\pi_2)$-good. Note that we do {\it not} require that the elements in $\tilde{\pazocal{A}}$ are good with respect to a uniformizer. In fact such a requirement would not have made sense since there is neither a  fixed  integral structure nor a choice of a  uniformizer on this integral structure on $\tilde{\pazocal{A}}.$ There is such a structure only modulo $(t^2).$  This  is  a crucial point in what follows. 
 
We will define a function 
$$
\varepsilon: \times ^{3} \beta(\tilde{\pazocal{A}},\alpha^{-1}(\pi_2)) \to k',
$$
where $k'$ is the residue field of the closed point of $\pazocal{A}^{\circ} _{2}.$ 

In order to do this, first we choose a lifting $\widetilde{\pazocal{B}^{\circ}}$ of $\pazocal{A}^{\circ} _2.$ In other words, $\widetilde{\pazocal{B}^{\circ}}$ is an object of $\mathcal{C}^{1} _{A},$ complete with respect to $(t),$ together with a fixed isomorphism 
$$
\tilde{\gamma}: (\widetilde{\pazocal{B}^{\circ}})_{2}:=\widetilde{\pazocal{B}^{\circ}}/(t^2) \xrightarrow{\sim} \pazocal{A} ^{\circ} _{2}.
$$
Let $\tilde{\pi}$ be a uniformizer on $\widetilde{\pazocal{B}^{\circ}}$ such that  $\tilde{\gamma}(\tilde{\pi}|_{t^{2}})=\pi_{2}.$ Moreover let $(\widetilde{\pazocal{B}^{\circ}})_{\eta}$ denote the localization of $\widetilde{\pazocal{B}^{\circ}}$ at the ideal $(t).$ Then $(\widetilde{\pazocal{B}^{\circ}})_{\eta}$  is a lifting of $\pazocal{A}_{2}$ together with the isomorphism $\tilde{\gamma}_{\eta},$ which is the localization of $\tilde{\gamma}:$ 
$$
\tilde{\gamma}_{\eta}:  (\widetilde{\pazocal{B}^{\circ}})_{\eta}/(t^2) = (\widetilde{\pazocal{B}^{\circ}})_{2,\eta}\xrightarrow{\sim} \pazocal{A}_{2}.
$$
Because of the choice of a uniformizer $\tilde{\pi}$ on $\widetilde{\pazocal{B}^{\circ}},$ we have the notion of $\tilde{\pi}$-goodness for elements of  $(\widetilde{\pazocal{B}^{\circ}})_{\eta},$ the set of which we  denoted by $\beta(\widetilde{\pazocal{B}^{\circ}},\tilde{\pi})$  in Definition \ref{gooddefn}. 

Now suppose that we start with $\tilde{p} \in  \times ^{3} \beta(\tilde{\pazocal{A}},\alpha^{-1}(\pi_2)).$ By assumption, $\alpha(\tilde{p} |_{t^2}) \in \times ^{3} \beta(\pazocal{A}_{2} ^{\circ},\pi_2).$ On the other hand, since $\tilde{\gamma}(\tilde{\pi}|_{t^2})=\pi_2,$ there is an element, not unique, say $\tilde{q} \in \times ^3\beta(\widetilde{\pazocal{B}^{\circ}},\tilde{\pi}),$ such that $\tilde{\gamma}_{\eta}(\tilde{q}|_{t^2})=\alpha(\tilde{p} |_{t^2}).$ 
This implies that $(\tilde{p},\tilde{q}) \in \times ^{3}\beta(\tilde{\pazocal{A}},(\widetilde{\pazocal{B}^{\circ}})_{\eta}, \tilde{\gamma}_{\eta} ^{-1}\circ\alpha).$ Therefore, by the construction of the previous section, we have 
$$
\omega(\tilde{p},\tilde{q},\tilde{\gamma}_{\eta} ^{-1}\circ\alpha ) \in \Omega^{1} _{\underline{\widetilde{\pazocal{B}}}/k},
$$
where $\underline{\widetilde{\pazocal{B}}}:=(\widetilde{\pazocal{B}^{\circ}})_{\eta}/(t).$ If we let $\underline{\widetilde{\pazocal{B}}}^{\circ}:=(\widetilde{\pazocal{B}^{\circ}})/(t)$ then by assumption $\underline{\widetilde{\pazocal{B}}}^{\circ}$ is a discrete valuation ring with field of fractions $\underline{\widetilde{\pazocal{B}}}.$ Therefore, the residue 
$
res(\omega(\tilde{p},\tilde{q},\tilde{\gamma}_{\eta} ^{-1}\circ\alpha ))$ of $\omega(\tilde{p},\tilde{q},\tilde{\gamma}_{\eta} ^{-1}\circ\alpha )$  is a well-defined element of the residue field $ \underline{\widetilde{\pazocal{B}}}^{\circ}/(\tilde{\pi})
 $
 of the closed point of $\underline{\widetilde{\pazocal{B}}}^{\circ}.$ This construction depends on the liftings $\widetilde{\pazocal{B}^{\circ}}$ and $\tilde{q},$ and the liftings $\tilde{\pazocal{A}}$ and  $\tilde{p}.$  Recall the residue map 
$$
res: \Lambda^{3}\beta(\widetilde{\pazocal{B}^{\circ}},\tilde{\pi}) \to  \Lambda^{2}(\widetilde{\pazocal{B}^{\circ}}/(\tilde{\pi}))^{\times}
$$
in (\ref{goncres})
and the map $\ell$ 
$$
\ell: \Lambda^{2}(\widetilde{\pazocal{B}^{\circ}}/(\tilde{\pi}))^{\times} \to \underline{\widetilde{\pazocal{B}}}^{\circ}/(\tilde{\pi}).
$$

We define 
$$
\varepsilon_{\tilde{q}}(\tilde{p}):=\ell(res(\tilde{q}) )+ res (\omega(\tilde{p},\tilde{q} ,\tilde{\gamma}_{\eta} ^{-1}\circ \alpha)).
$$
We first show that the expression above is independent of the choice of local liftings $\tilde{q}.$ 

\begin{proposition}\label{indeplocal}
Suppose that we have another lifting $\widehat{\pazocal{B}^{\circ}}$ of $\pazocal{A}^{\circ} _{2},$ with the isomorphism 
$$
\hat{\gamma}: (\widehat{\pazocal{B}^{\circ}})_{2}:=\widehat{\pazocal{B}^{\circ}}/(t^2) \xrightarrow{\sim} \pazocal{A} ^{\circ} _{2},
$$
 $\hat{\pi},$  a uniformizer on $\widehat{\pazocal{B}^{\circ}}$ such that  $\hat{\gamma}(\hat{\pi}|_{t^{2}})=\pi_{2},$ and  $\hat{q} \in \beta(\widehat{\pazocal{B}^{\circ}},\hat{\pi}),$ a $\hat{\pi}$-good lifting  of $\tilde{p},$ i.e. $\hat{\gamma}_{\eta}(|\hat{q}_{t^2})=\alpha(\tilde{p}|_{t^2}),$ then we have  
$$
\underline{\hat{\gamma}}(\varepsilon_{\hat{q}}(\tilde{p}))=\underline{\tilde{\gamma}}(\varepsilon_{\tilde{q}}(\tilde{p})),
$$
where $\underline{\tilde{\gamma}}$ (resp. $\underline{\hat{\gamma}}$ ) is the map from $\underline{\widetilde{\pazocal{B}}}^{\circ}/(\tilde{\pi})$ (resp. $\underline{\widehat{\pazocal{B}}}^{\circ}/(\hat{\pi})$) to  $\pazocal{A} ^{\circ} _{2}/(t,\pi_2)=k'$ induced by $\tilde{\gamma} $ (resp. $\hat{\gamma}$). 
\end{proposition}

\begin{proof} 
By the transitivity property (P5) applied on the generic point, we have 
$$
\underline{\tilde{\gamma}}_{\eta,*}(\omega(\tilde{p},\tilde{q} ,\tilde{\gamma}_{\eta} ^{-1}\circ \alpha))+ \underline{\hat{\gamma}}_{\eta,*}\omega(\tilde{q} ,\hat{q} , \hat{\gamma}_{\eta} ^{-1} \circ \tilde{\gamma}_{\eta} )= \underline{\hat{\gamma}}_{\eta,*}\omega(\tilde{p},\hat{q},\hat{\gamma}_{\eta} ^{-1}\circ \alpha).
$$
Taking residues in the above expression, the statement that we would like to prove reduces to: 
 
$$
\underline {\tilde{\gamma}}\circ\ell(res_{\tilde{\pi}} (\tilde{q}) )- \underline {\hat{\gamma}}\circ\ell(res _{\hat{\pi}}(\hat{q}))= \underline {\hat{\gamma}} \circ res_{\underline{\hat{\pi}}}( \omega(\tilde{q} ,\hat{q} ,  \hat{\gamma}^{-1}\circ\tilde{\gamma})).
$$
Note that in the last expression since $\hat{q}$ and $\tilde{q}$ are good liftings, i.e.
$$
(\tilde{q},\hat{q}) \in \times^{3} \beta(\widetilde{\pazocal{B} ^{\circ}},\widehat{\pazocal{B} ^{\circ}}, \hat{\gamma}^{-1}\circ\tilde{\gamma},\tilde{\pi},\hat{\pi} ),
$$
the residue  $res_{\underline{\hat{\pi}}}( \omega(\tilde{q} ,\hat{q} ,  \hat{\gamma}^{-1}\circ\tilde{\gamma}))$ can be computed in terms of $\ell,$ by the residue property (P8),  and hence gives the above expression. 
\end{proof}

\begin{definition}\label{defepsilon}
With the notation as above we  define $\varepsilon: \times ^{3} \beta(\tilde{\pazocal{A}},\alpha^{-1}(\pi_2)) \to k' ,$ by the formula 
$$
\varepsilon(\tilde{p}):=\underline{\tilde{\gamma}}(\varepsilon_{\tilde{q}}(\tilde{p}))=\underline{\tilde{\gamma}}(\ell(res(\tilde{q}) )+ res (\omega(\tilde{p},\tilde{q} ,\tilde{\gamma}_{\eta} ^{-1}\circ \alpha))).
$$
By the previous proposition, this definition is independent of the choice of the  local good lifting $\tilde{q}.$ In fact, by the multi-linearity and the anti-symmetry properties of $\omega,$  $\ell$ and $res,$ $\varepsilon$ induces a map 
$$
\varepsilon: \Lambda ^{3} \beta(\tilde{\pazocal{A}},\alpha^{-1}(\pi_2)) \to k'.
$$
\end{definition}

\subsection{The construction of $\rho$} Suppose that $C_{2}/S_{2}$ is a smooth and proper curve. We let $C:=C_{2} \times _{S_{2}}S_{1}$ be its special fiber, $\eta,$ the generic point of $C,$ and $|C|,$ the set of closed points of $C.$ Similarly, for an open subscheme  $U_{2} \subseteq C_{2},$ we let $U$ to be the special fiber.  We call an element of $k(C_{2}) ^{\times}:=\pazocal{O}_{C_{2},\eta} ^{\times},$ a  rational function on $C_{2}.$      We call $\mathcal{P}_2=\{ \pi_{2,c}|c \in |C| \} $ {\it a system of uniformizers}, if for every  $c\in |C| ,$ $\pi_{2,c} \in \hat{\pazocal{O}}_{C_2,c}$ is a uniformizer. Here $\hat{\pazocal{O}}_{C_{2},c}$ is the completion of the local ring $\pazocal{O}_{C_{2},c}$ at its maximal ideal.    We say that a rational function $f$ is $\mathcal{P}_{2}${\it -good}, if $f$ is $\pi_{2,c}$-good for every $c \in |C|.$ We denote the group of $\mathcal{P}_{2}$-good rational functions by $k(C_{2},\mathcal{P}_2)^{\times}.$  Similarly, we say that $p=(f,g,h)$ is  $\mathcal{P}_{2}${\it -good}, if all of $f,g,$ and $h$ are. If we let $|p|:=|div(f|_{C})|\cup |div(g|_{C})| \cup |div(h|_{C})| \subseteq |C|, $ then $p$ is $\pi_{2,c}$-good for $c \notin |p|,$ for any choice of a uniformizer $\pi_{2,c}.$ On the other hand, for example,  a triple of functions of the form $(s,s+t,\cdot)$ is not $\pi_2$-good on the spectrum of $k[[s,t]]/(t^2),$ for any choice of a uniformizer $\pi_2.$

Fix $p=(f,g,h)\in  \times^{3} k(C_{2},\mathcal{P}_2) ^{\times}.$ Choose a lifting  $\tilde{\pazocal{A}}$ of $\pazocal{A}_{2}:=\pazocal{O}_{C_{2},\eta},$ as in \textsection 3.1, together with a fixed isomorphism 
$$
\alpha:\tilde{\pazocal{A}}_2:=\tilde{\pazocal{A}}/(t^{2}) \xrightarrow{\sim} \pazocal{A}_{2}.
$$
Choose any lifting $\tilde{p}$ of $p$ to $\tilde{A}.$ For any $c \in |C|,$ we apply the construction in \textsection 3.1, with $\tilde{A}^{\circ} _{2}:=\hat{\pazocal{O}}_{C_{2},c}.$  
Since, by assumption, $p$ is $\mathcal{P}_{2}$-good, we have, for any $c \in |C|,$ 
$\tilde{p} \in \times ^{3}\beta(\tilde{\pazocal{A}}, \alpha^{-1}(\pi_{2,c})).$ 
Therefore corresponding to $c,$ we have $\varepsilon(\tilde{p};c) \in k(c).$ For any  element $y$ in $k(c),$ let ${\rm Tr}_{k}(y)$ denote its trace from $k(c) $ to $k.$

We define $\rho(\tilde{p})$ as:

\begin{eqnarray}\label{defnrhogeneric}
\rho(\tilde{p})=\sum_{c \in |C|} {\rm Tr}_k\varepsilon(\tilde{p};c ) \in k.
\end{eqnarray}
We will need the following proposition in order to show that this sum  makes sense, and to show that it defines the function that we are seeking. 

\begin{proposition}\label{indepgen}
For any  lifting $\tilde{p},$ the sum in the definition  (\ref{defnrhogeneric})  of $\rho(\tilde{p})$ is finite. Moreover, for any other   lifting 
$\hat{p},$ we have $\rho(\tilde{p})=\rho(\hat{p}).$
\end{proposition}

\begin{proof} 
First, let us choose an affine open set $U_2 \subseteq C_2$ such that $U_2$ has a lifting $\tilde{U}$ and $p|_{U_2}$ has a lifting $\tilde{p}_{U}$ to $\tilde{U}, $ which is good with respect to a system of parameters $\mathcal{P}_{2}(\tilde{U})$ on $\tilde{U}$ which lift $\mathcal{P}_{2}|_{U_2}.$  More precisely, $\tilde{U}$ is an affine scheme whose ring of functions $\tilde{A}_{U}$ is an object of $\mathcal{C}^{1} _{A}$ together with an isomorphism 
$$
\alpha_{U_2}: \tilde{\pazocal{A}}_{U}|_{t^2}\xrightarrow{\sim} \pazocal{A}_{2,U},
$$ 
where $\pazocal{A}_{2,U}$ is the ring of functions of $U_{2};$  $\mathcal{P}_{2}(\tilde{U})$ is a system of parameters on $\tilde{\pazocal{A}}_{U}$ which lift $\alpha_{U_2} ^{-1}(\mathcal{P}_{2}|_{U_2});$ and $\tilde{p}_{U} \in \tilde{\pazocal{A}}_{U},$ is a lifting of $\alpha_{U_2}^{-1}(p|_{U_2}).$

 Then on the generic point, we have the generic lifting $\tilde{p}_{U,\eta}.$ Let us look at the summands that appear in the expression (\ref{defnrhogeneric}) for  $\rho(\tilde{p}_{U,\eta}).$ These are the terms 
$$
{\rm Tr}_k\varepsilon(\tilde{p}_{U,\eta};c )={\rm Tr}_k(\ell(res_c(\tilde{q}_{c}) )+ res_{c} \omega(\tilde{p}_{U,\eta},\tilde{q}_{c} ,  \tilde{\gamma}_{c,\eta} ^{-1} \circ \alpha_{U_2,\eta})),
$$
where $\tilde{q}_{c}$ is any local good choice of lifting of $p$ at $c.$ Again to be more precise, $\widetilde{\pazocal{B}^{\circ} _c }$ is an object of $\mathcal{C}^{1} _{A},$ together with an isomorphism 
$$
\tilde{\gamma}_{c}: \widetilde{\pazocal{B}^{\circ} _c }|_{t^2} \xrightarrow{\sim} \hat{\pazocal{O}}_{C_{2},c},  
$$
and $\tilde{q}_{c}$ is a lifting of $\tilde{\gamma}_{c,\eta}^{-1}(p_{\eta}),$ which is good with respect to a uniformizer in $\widetilde{\pazocal{B}^{\circ} _c }$ which lifts $\tilde{\gamma}_{c}^{-1}(\pi_{2,c}).$

In case $c\in U,$ $\tilde{p}_{U,\eta},$ by our assumption,  is a  good lifting at $c.$ Since $\tilde{q}_{c}$ is also a good lifting,  by the residue property (P8) we have
$$
{\rm Tr}_k res_c \omega(\tilde{p}_{U,\eta},\tilde{q}_{c} , \tilde{\gamma}_{c,\eta} ^{-1} \circ \alpha_{U_2,\eta}))={\rm Tr}_k\ell(res_{c} \tilde{p}_{U})-{\rm Tr}_k\ell( res_c \tilde{ q_{c}}) 
$$
and therefore ${\rm Tr}_k\varepsilon(\tilde{p}_{U,\eta};c )={\rm Tr}_k\ell(res_c \tilde{p}_{U}).$ Since $\tilde{p}_{U}$ is a good lifting of $p|_{U_2},$  right hand side  is non-zero only when $c \in |p|.$ Hence for $c \in U,$ the terms $\varepsilon(\tilde{p}_{U,\eta}Ê;c )$ are non-zero for only finitely many $c.$ Since $C\setminus U$ is a finite set, we conclude that the sum (\ref{defnrhogeneric})  for $\rho(\tilde{p}_{U,\eta})$ is finite. 

Now let $\hat{p}_{\eta}$ and $\tilde{p}_{\eta}$ be  arbitrary generic liftings. In other words, we have  isomorphisms 
$$
\hat{\alpha}: \hat{\pazocal{A}}|_{t^2} \xrightarrow{\sim} \pazocal{A}_{2}
$$
and 
$$
\tilde{\alpha}: \tilde{\pazocal{A}}|_{t^2} \xrightarrow{\sim} \pazocal{A}_{2},
$$
and elements $\hat{p}_{\eta} \in \hat{\pazocal{A}}$ and $\tilde{p}_{\eta} \in \tilde{\pazocal{A}}$ which lift $\hat{\alpha}^{-1}(p_\eta)$ and $\tilde{\alpha}^{-1}(p_\eta).$ 
The differences 
$\varepsilon(\tilde{p}_{\eta}Ê;c)-\varepsilon(\hat{p}_{\eta}Ê;c )$ can be computed by choosing a good local lifting $\tilde{q}_{c}.$ Namely,  we have 
\begin{eqnarray*}
& &\varepsilon(\tilde{p}_{\eta}Ê;c )-\varepsilon(\hat{p}_{\eta}Ê;c )=\underline{\tilde{\gamma}_{c}}(\varepsilon_{\tilde{q}_{c}}(\tilde{p}_{\eta}))-\underline{\tilde{\gamma}_{c}}(\varepsilon_{\tilde{q}_{c}}(\hat{p}_{\eta}))=\underline{\tilde{\gamma}_{c}}(\varepsilon_{\tilde{q}_{c}}(\tilde{p}_{\eta})-\varepsilon_{\tilde{q}_{c}}(\hat{p}_{\eta}))\\
&=&\underline{\tilde{\gamma}_{c}}(res( \omega(\tilde{p}_{\eta},\tilde{q}_{c} ,\tilde{\gamma}_{c,\eta} ^{-1}\circ \tilde{\alpha}_{c}))-res (\omega(\hat{p}_{\eta},\tilde{q}_{c} , \tilde{\gamma}_{c,\eta} ^{-1}\circ \hat{\alpha}_{c}))).
\end{eqnarray*}
In the last expression, $\tilde{\alpha}_{c}$ (resp. $\hat{\alpha}_{c}$) is the map induced by $\tilde{\alpha}$ (resp.  $\hat{\alpha}$) after completing at $c; $ and $\tilde{\gamma}_{c,\eta}$ is the map induced by $\tilde{\gamma}_{c}$ after localising at $\eta.$
By the transitivity property (P5),  
$$
 \omega(\tilde{p}_{\eta},\tilde{q}_{c} ,\tilde{\gamma}_{c,\eta} ^{-1}\circ \tilde{\alpha}_{c})-\omega(\hat{p}_{\eta},\tilde{q}_{c} , \tilde{\gamma}_{c,\eta} ^{-1}\circ \hat{\alpha}_{c})= (\underline{\tilde{\gamma}_{c,\eta}} ^{-1}\circ \underline{\hat{\alpha}_{c}})_{*}\omega(\tilde{p}_{\eta}, \hat{p}_{\eta} ,  \hat{\alpha}_{c} ^{-1} \circ \tilde{\alpha}_{c}).
$$
Therefore, after taking residues, we have 
$$
\varepsilon(\tilde{p}_{\eta}Ê;c )-\varepsilon(\hat{p}_{\eta}Ê;c )=\underline{\hat{\alpha}_{c}}(res \, \omega(\tilde{p}_{\eta}, \hat{p}_{\eta} ,  \hat{\alpha}_{c} ^{-1} \circ \tilde{\alpha}_{c}) )=res\, (\underline{\hat{\alpha}_{c}}_{*}(\omega(\tilde{p}_{\eta}, \hat{p}_{\eta} ,  \hat{\alpha}_{c} ^{-1} \circ \tilde{\alpha}_{c}) )) 
$$
By the functoriality property (P3), $\underline{\hat{\alpha}_{c}}_{*}\omega(\tilde{p}_{\eta}, \hat{p}_{\eta} ,  \hat{\alpha}_{c} ^{-1} \circ \tilde{\alpha}_{c})$ is the restriction to the formal neighbourhood of $c,$ of the rational 1-form $\underline{\hat{\alpha}}_{*}(\omega(\tilde{p}_{\eta}, \hat{p}_{\eta} ,  \hat{\alpha} ^{-1} \circ \tilde{\alpha})).$ Using this in the last equality above, and taking traces,  we obtain: 
\begin{eqnarray}\label{trepsilongen}
{\rm Tr}_k\varepsilon(\tilde{p}_{\eta}Ê;c )-{\rm Tr}_k\varepsilon(\hat{p}_{\eta}Ê;c )={\rm Tr}_{k} \,res_{c}\, \underline{\hat{\alpha}}_{*}(\omega(\tilde{p}_{\eta}, \hat{p}_{\eta} ,  \hat{\alpha} ^{-1} \circ \tilde{\alpha})),
\end{eqnarray}
for every $c \in |C|.$

Since the rational 1-form $\underline{\hat{\alpha}}_{*}(\omega(\tilde{p}_{\eta}, \hat{p}_{\eta} ,  \hat{\alpha} ^{-1} \circ \tilde{\alpha}))$ has only finitely many singularities, by the identity  (\ref{trepsilongen}), the finiteness of the sum for $\rho(\tilde{p}_{\eta})$ and for $\rho(\hat{p}_{\eta})$ are equivalent. Since we know the finiteness of the sum for one lifting, namely for $\tilde{p}_{U,\eta},$ we therefore know it for all the liftings. This proves the first statement. 

In order to prove the second statement, we note that by summing (\ref{trepsilongen}) above over all $c \in |C|$: 
$$
 \rho(\tilde{p}_{\eta})-\rho(\hat{p}_{\eta})=\sum _{c \in |C|}{\rm Tr}_k \, res_{c}\, \underline{\hat{\alpha}}_{*}(\omega(\tilde{p}_{\eta}, \hat{p}_{\eta} ,  \hat{\alpha} ^{-1} \circ \tilde{\alpha})).
$$
 Since the sum of the residues of a 1-form on a curve is 0, we have $ \rho(\tilde{p}_{\eta})=\rho(\hat{p}_{\eta}).$
\end{proof}

\begin{definition} For $p \in \times ^{3} k(C_{2},\mathcal{P}_{2})^{\times},$  we  define the value of the regulator $\rho$ on $p$ as
\begin{eqnarray}\label{regcurves}
\rho(p):= \sum_{c\in |C|}{\rm Tr}_k(\ell(res_c(\tilde{q}_{c}) )+ res_{c} \omega(\tilde{p}_{\eta},\tilde{q}_{c} ,  \tilde{\gamma}_{c,\eta} ^{-1} \circ \alpha_c )),
\end{eqnarray}
where $\tilde{p}_{\eta}$ is a lifting of $p_{\eta},$ through a generic lifting $\tilde{\pazocal{A}}$ and an isomorphism 
$
\alpha: \tilde{\pazocal{A}}/(t^2)\xrightarrow{\sim} \pazocal{O}_{C_{2},\eta};
 $
and $\tilde{q}_{c}$ is a local good choice of a lifting, through a local lifting $\widetilde{\pazocal{B}^{\circ}_{c}}$ and an isomorphism 
$
\tilde{\gamma}_{c}: \widetilde{\pazocal{B}^{\circ}_{c}}/(t^2) \xrightarrow{\sim}  \hat{\pazocal{O}}_{C_{2},c}.
$ By the above Proposition \ref{indeplocal} and Proposition \ref{indepgen}  this sum is finite and is independent of the choice of a generic lifting $\tilde{p}_{\eta}$ and the choice of good local liftings $\{\tilde{q} _{c}|c \in C\}.$ 
\end{definition}

By the multi-linearity and the anti-symmetry of the functions appearing in the definition of $\rho,$ we see that  $\rho$ induces a map 
$$
\rho: \Lambda^{3} k(C_{2},\mathcal{P}_{2})^{\times} \to k.
$$

The following  corollary gives a more explicit expression for $\rho(p).$ Choose  a cover $\{ U_{i}\}_{1 \leq i \leq r}$ of $C_{2}$ by  open sets   such that there are liftings $\tilde{U}_{i}$ of $U_{i}$ to smooth schemes over $S$ whose ring of regular functions lie in $\mathcal{C}^{1}_{A},$ and functions  $\tilde{p}_i$  on $\tilde{U}_{i}$ which are good liftings of $p|_{U_{i}},$ with respect to a system of uniformizers which lift that on $U_{i}.$    If $X$ is any subset of $U_{i},$ let $|X|$ denote the set of closed points in $X$ and  we let $r_{X}(\tilde{p}_i)$ denote the sum of the local contributions ${\rm Tr}_k\,\ell(res_x(\tilde{p}_{i})),$ for $x\in |X|.$ Note that since $\tilde{p}_i$ is a good lifting of $p$ on $U_{i},$ these contributions are possibly non-zero only for $x \in |X|\cap |p|.$ Similarly, if $\omega$ is any rational 1-form on $U_{i},$ we let $\omega_{X}$ denote the sum of the residues $\sum _{x \in |X|} {\rm Tr}_k \,res_{x} \omega.$

\begin{corollary}\label{explicitrho}
With notation as above, let $U_{i} ':=U_i \setminus \cup_{i+1 \leq j \leq r} U_{j},$ for $i<r.$ Then we have
\begin{eqnarray}\label{setdefn}
 \rho(p)=r_{U_{r}}(\tilde{p}_{r})+ \sum _{1 \leq i \leq r-1} \big(
r_{U_{i} '}(\tilde{p}_{i})+\omega_{U_i '} (\tilde{p}_{r},\tilde{p}_{i}, \chi_{U_i \cap U_r}) \big),
\end{eqnarray}
where  $\chi_{U_{i}\cap U_{r}}$ denotes the isomorphism between $(\tilde{U}_{r}|_{U_{i}\cap U_{r}})|_{t^2}\xrightarrow{\sim} (\tilde{U}_{i}|_{U_{i}\cap U_{r}})|_{t^2}.$
\end{corollary}

\begin{proof}
In order to  compute $\rho(p) $ we need to choose a generic lifting of $p$ and a collection of good local liftings at each closed point. We let $\tilde{p}_{r,\eta}$ be the generic lifting. We choose  the good local  lftings $\tilde{q}_{c}$ as follows. If $c \in |U_{r}|,$ let $\tilde{q}_{c}=\tilde{p}_{U_{r},c}.$ Otherwise, there is a unique $ 1\leq i \leq r-1  $ such that $c \in |U_{i} '|$ and we let $\tilde{q} _{c}:= \tilde{p}_{i,c}.$     Then we have 
$$
\rho(p)=\rho(\tilde{p}_{r,\eta})=\sum_{c \in |C|}{\rm Tr}_k \varepsilon_{\tilde{q}_c}(\tilde{p}_{r,\eta})=\sum _{C \in |U_{r}|}{\rm Tr}_k\varepsilon_{\tilde{p}_{r,c}}(\tilde{p}_{r,\eta})+\sum _{1 \leq i \leq r-1} \sum _{c \in |U_{i}'|}{\rm Tr}_k \varepsilon_{\tilde{p}_{i,c}}(\tilde{p}_{r,\eta}).
$$
The statement then follows after noting that ${\rm Tr}_k\, \varepsilon_{\tilde{p}_{r,c}}(\tilde{p}_{r,\eta})={\rm Tr}_k \, \ell (res_c  \, \tilde{p}_{r}).$
\end{proof}

\subsection{Bloch group  relation.} In this section, we will construct a residue map from a weight 3 infinitesimal Bloch complex to a weight 2 infinitesimal Bloch complex and study the relation of $\rho$  to this map of complexes. The results of  this section will be used in the proof of the infinitesimal strong reciprocity conjecture. 

\subsubsection{Infinitesimal Bloch group of a discrete valuation ring and the residue map}\label{locressection}
In this subsection we follow the notation in \textsection \ref{section goodness}. We first define the Bloch group $B_{2}(\pazocal{A}_{2} ^{\circ},\pi_{2}).$ 

\begin{definition}
Let $\beta^{\flat}(\pazocal{A}_{2} ^{\circ},\pi_{2}) :=  \{x| x,\, 1-x \in  \beta(\pazocal{A}_{2} ^{\circ},\pi_{2}) \} .$  We define the Bloch group $B_{2}(\pazocal{A}_{2} ^{\circ},\pi_{2})$ to be the quotient of the free abelian group generated by the symbols $\{ x\}_{2},$ with   $x \in \beta^{\flat}(\pazocal{A}_{2} ^{\circ},\pi_{2}) ,$ modulo the relations generated by  $\{ x\}_{2}-\{ y\}_{2}+\{ (y/x)\}_{2}-\{ (1-x^{-1})/(1-y^{-1})\}_{2}+\{ (1-x)/(1-y)\}_{2},$ for  $x,y \in \beta^{\flat}(\pazocal{A}_{2} ^{\circ},\pi_{2}) ,$ with $x-y \in \beta(\pazocal{A}_{2} ^{\circ},\pi_{2}).$ As usual this gives us  a complex
 $
B_{2}(\pazocal{A}_{2} ^{\circ},\pi_{2}) \xrightarrow{\delta} \Lambda ^{2}\beta(\pazocal{A}_{2} ^{\circ},\pi_{2}) ,
 $ 
where $\delta(\{ x\}_{2}):=(1-x)\wedge x.$ 

\end{definition}
This induces part of a weight 3 complex: $B_{2}(\pazocal{A}_{2} ^{\circ},\pi_{2}) \otimes \beta(\pazocal{A}_{2} ^{\circ},\pi_{2})\xrightarrow{\Delta} \Lambda ^{3}\beta(\pazocal{A}_{2} ^{\circ},\pi_{2}) ,$ where $\Delta$ sends $\{ x\}_{2} \otimes y$ to $\delta(x)\wedge y.$ We also have a residue map   that gives a commutative diagram 
\[
\begin{CD}
&&  B_{2}(\pazocal{A}_{2} ^{\circ},\pi_{2}) \otimes \beta(\pazocal{A}_{2} ^{\circ},\pi_{2}) @>{\Delta}>> \Lambda ^{3}\beta(\pazocal{A}_{2} ^{\circ},\pi_{2}) \\
   && @V{res}VV     @V{res}VV\\
 &&  B_{2}(\pazocal{A}_{2} ^{\circ}/(\pi_{2})) @>{\delta}>> \Lambda^2 (\pazocal{A}_{2} ^{\circ}/(\pi_{2}))^{\times},
\end{CD}
\]
as in \cite{gon2}. The residue map on $\Lambda ^{3}\beta(\pazocal{A}_{2} ^{\circ},\pi_{2})$ is defined exactly as in \textsection \ref{residue}. The residue map on $B_{2}(\pazocal{A}_{2} ^{\circ},\pi_{2}) \otimes \beta(\pazocal{A}_{2} ^{\circ},\pi_{2}) $ is characterised by the following properties:

(i) $res(\{x \}_2 \otimes y)=0,$ if $x \notin (\pazocal{A}_{2} ^{\circ})^{\flat}:=\{a| a, 1-a \in (\pazocal{A}_{2} ^{\circ})^{\times} \} $
and 

(ii) $res(\{x \}_2 \otimes y)=n\{ \overline{x} \}_{2}, $ if $x  \in  (\pazocal{A}_{2} ^{\circ})^{\flat} $ and $y=\pi_{2} ^{n} u,$ with $u  \in (\pazocal{A}_{2} ^{\circ})^{\times},$  and $\overline{x}$ is the reduction of $x$ modulo $(\pi_{2}).$

\subsubsection{Global version}

Let us now assume  that $C_{2}/S_{2}$ is  a smooth and proper curve, $C$ denote its special fiber and  $\mathcal{P}_{2}$ a system of uniformizers on $C_{2}.$ For a finite extension $k'/k$ and $n \geq 1,$ we write $k'_n:=k'[t]/(t^n).$ Note that for every $c \in |C|,$ there is a choice of a uniformizer $\pi_{2,c}\in \mathcal{P}_{2},$ and this gives a quotient $\pazocal{O}_{C_{2},c}/(\pi_{2,c}),$ which is canonically isomorphic to $k(c)_{2}.$

\begin{definition}
Let $k(C_{2},\mathcal{P}_{2})^{\flat}:=  \{x| x,\, 1-x \in k(C_{2},\mathcal{P}_{2}) ^{\times} \} .$  Then we define the Bloch group $B_{2}(k(C_{2}),\mathcal{P}_{2})$ as the quotient of the free abelian group generated by the symbols $\{ x\}_{2},$ with   $x \in k(C_{2},\mathcal{P}_{2})^{\flat} ,$ modulo the relations generated by  $\{ x\}_{2}-\{ y\}_{2}+\{ (y/x)\}_{2}-\{ (1-x^{-1})/(1-y^{-1})\}_{2}+\{ (1-x)/(1-y)\}_{2},$ for  $x,y \in k(C_{2},\mathcal{P}_{2})^{\flat} ,$ with $x-y \in k(C_{2},\mathcal{P}_2) ^{\times}.$ This gives   a complex
 $
 B_{2}(k(C_{2}),\mathcal{P}_{2}) \xrightarrow{\delta} \Lambda ^{2} k(C_{2}),\mathcal{P}_{2})^{\times}.
 $ 

\end{definition}
We have natural maps $ k(C_2,\mathcal{P}_2)^{\times } \to \beta(\pazocal{O}_{C_{2},c},\pi_{2,c})$ and $B_{2}(k(C_{2}),\mathcal{P}_{2}) \to B_{2}(\pazocal{O}_{C_{2},c},\pi_{2,c}).$ Combining these with the sum of  the residue maps in \textsection \ref{locressection}, corresponding to $\pazocal{A}^{\circ} _{2}=\pazocal{O}_{C_{2},c},$  over all $c \in |C|$ and identifing $\pazocal{O}_{C_2,c}/(\pi_{2,c})$ with $k(c)_{2},$ we obtain a commutative diagram:
 \[
\begin{CD}
&&  B_{2}(k(C_2,\mathcal{P}_2)) \otimes k(C_2,\mathcal{P}_2)^{\times }@>{\Delta}>> \Lambda^{3} k(C_2,\mathcal{P}_{2})^{\times} \\
   && @V{res}VV     @V{res}VV\\
 && \oplus _{c \in |C|} B_{2}(k(c)_2) @>{\delta}>> \oplus _{c \in |C|} \Lambda^2 k(c)_2 ^{\times}.
\end{CD}
\]

We constructed an additive dilogarithm map $\ell i _2: B_{2}(k_2) \to k,$ for any field $k$  of characteristic 0 in \cite{Unver}. Let us briefly recall this function. It is given by the explicit formula 
$$
\ell i_{2}(\{ s+at \}_{2})=-\frac{a^3}{2s^{2}(1-s)^2},
$$
for   $s\in k^{\flat}:=k^{\times } \setminus \{ 1\}$ and $a \in k,$ on the generators of $B_{2}(k_2).$  It has also the following description using the notation of \textsection \ref{residue}: taking $B:=k[[t]]$ as in \textsection \ref{residue}, we have a map $\ell: \Lambda^{2} k[[t]] ^{\times } \to k.$
The composition $B_{2}(k[[t]]) \xrightarrow{\delta}  \Lambda^{2} k[[t]] ^{\times } \xrightarrow{\ell} k,$ then factors through the natural projection 
$B_{2}(k[[t]]) \twoheadrightarrow  B_{2}(k_2)$ to give $\ell i_{2}:B_{2}(k_2)\to k.$ 

Let $ \ell i_{2,c}$ denote the additive dilogarithm from $ B_{2}(k(c)_2)$ to $k(c),$ and $\ell i_{2,|C|}$  denote the composition  $ \oplus _{c \in |C|} B_{2}(k(c)_2) \xrightarrow{\oplus \ell i_{2,c}} \oplus _{c \in |C|}k(c)  \xrightarrow{\sum {\rm Tr}_{k(c)/k}}  k.$ Then we have the following relation of this function with the regulator $\rho.$ 

\begin{proposition}\label{regdilog}
The functions  $\ell i _{2,|C|} \circ res$ and $\rho \circ \Delta$  from $ B_{2}(k(C_2,\mathcal{P}_2)) \otimes k(C_2,\mathcal{P}_2)^{\times }$ to $k$ are equal.  
\end{proposition}

\begin{proof}

We need to show that both of the functions agree on a general element $\{ f\}_{2} \otimes g \in B_{2} (k(C_{2},\mathcal{P}_2)) \otimes k(C_{2},\mathcal{P}_2)^{\times} . $
We have  $\Delta(\{ f\}_{2} \otimes g)=(1-f) \wedge f \wedge g,$ and hence  $\rho\circ \Delta (\{ f\}_{2} \otimes g)=\rho (p), $ with $p=(1-f,f,g).$ In order to compute $\rho(p),$ we need to choose a generic lifting of $p_{\eta}$ and good local liftings of $p_{c},$ for every $c \in |C|.$

Let $\tilde{\pazocal{A}}$ denote a lifting of $\pazocal{O}_{C_{2},\eta},$ together with an isomorphism 
$\alpha: \tilde{\pazocal{A}}|_{t^2}  \xrightarrow{\sim } \pazocal{O}_{C_{2},\eta};$ and let $\widetilde{\pazocal{B}^{\circ} _{c} }$ be a lifting of $\hat{\pazocal{O}}_{C_2,c},$ together with isomorphisms $\tilde{\gamma}_{c}: \widetilde{\pazocal{B}^{\circ} _{c} }|_{t^2} \xrightarrow{\sim} \hat{\pazocal{O}}_{C_2,c},$ for every $c \in |C|.$ In order to choose a generic lifting of $p_{\eta},$ choose liftings $\tilde{f}_{\eta}$ and $\tilde{g}_{\eta}$ of $\alpha^{-1}(f_{\eta})$ and $\alpha^{-1}(g_{\eta})$ and let $\tilde{p}_{\eta}:=(1-\tilde{f}_{\eta},\tilde{f}_{\eta},\tilde{g}_{\eta}).$ For the local lifting at $c,$ first choose a uniformizer $\tilde{\pi}_c$ on $\widetilde{\pazocal{B}^{\circ} _{c} }$ lifting the given one $\pi_{2,c} $ on $\hat{\pazocal{O}}_{C_2,c}.$ Since $g_c$ is $\pi_{2,c}$-good, let $\tilde{g}_{c} $ be any $\tilde{\pi}_c$-good lifting of $g_c.$  By assumption both $f_c$ and $1-f_c$ are $\pi_{2,c}$-good.  If $c \notin div (f|C) \cup div((1-f)|C),$ and  if   $\tilde{f}_{c}$ is any lifting of $f_{c}$ then both $\tilde{f}_{c}$ and $1-\tilde{f}_{c}$ will be $\tilde{\pi}_{c}$-good. If $f|C$ vanishes at $c$, choose a $\tilde{\pi}_{c}$-good lifting $\tilde{f}_{c}$ of $f_c,$ then since $(1-f)|_C$ is not 0 at $c,$ $1-\tilde{f}_{c}$ becomes a $\tilde{\pi}_{c}$-good lifting of $1-f_c.$ In case, $(1-f)|_{C}$ vainshes at $c,$ choose $\tilde{f}_{c}$ such that $1-\tilde{f}_{c}$ is  a $\tilde{\pi}_c$-good lifting of $1-f_{c},$ then also $\tilde{f}_{c}$ becomes a $\tilde{\pi}_c$-good lifting of $f_c.$  The other remaining case is when $f|_{C}$ and $(1-f)|_{C}$ both have a pole at $c.$ In this case, if we let $\tilde{f}_{c}$ as any $\tilde{\pi}_{c}$-good lifting of $f_{c}$ then $1-\tilde{f}_{c}$ is a $\tilde{\pi}_c$-good lifting of $1-f_c.$   In each of these cases, we let $\tilde{q}_{c}:=(1-\tilde{f}_c,\tilde{f}_c,\tilde{g}_c)$ as the good local lifting of $p_c.$ 
 
 By definition, $\rho(p)=\sum _{c \in |C|}{\rm Tr}_{k} \varepsilon_{\tilde{q}_{c}}(\tilde{p}_{\eta})=
\sum_{c \in |C|}({\rm Tr_{k}} \ell (res\, \tilde{q}_{c})+{\rm Tr_{k}} res_c \, \omega (\tilde{p}_{\eta},\tilde{q}_{c}, \tilde{\gamma}_{c,\eta} ^{-1} \circ \alpha_c) )   .$
First we deal with the second summand. Note that 
$$res_c \, \omega (\tilde{p}_{\eta},\tilde{q}_{c}, \tilde{\gamma}_{c,\eta} ^{-1} \circ \alpha_c)=res_{c}  \, \omega (\alpha_c(\tilde{p}_{\eta}),\tilde{\gamma}_{c,\eta}(\tilde{q}_{c}), id) ,$$
with $\alpha_c(\tilde{p}_{\eta})=(1-\alpha_c(\tilde{f}_{\eta}),\alpha_c(\tilde{f}_{\eta}), \alpha_c(\tilde{g}_{\eta})  )$ and $\tilde{\gamma}_{c,\eta}(\tilde{q}_{c})=(1-\tilde{\gamma}_{c,\eta}(\tilde{f}_{c}), \tilde{\gamma}_{c,\eta}(\tilde{f}_{c}), \tilde{\gamma}_{c,\eta}(\tilde{g}_{c})).$ By the construction of $\tilde{f}_{\eta},\,  \tilde{g}_{\eta}, \tilde{f}_{c}$  and $\tilde{g}_{c},$ we see that  $(\alpha_c(\tilde{p}_{\eta}),\tilde{\gamma}_{c,\eta}(\tilde{q}_{c})) \in \Lambda^3 \beta(\hat{K}_{C_2,c},(t^2)),$ where $\hat{K}_{C_{2},c}$ is the localization of $\hat{\pazocal{O}}_{C_2,c}$ at the generic point of its special fiber.   By the above expressions,  we have: 
$$
(\alpha_c(\tilde{p}_{\eta}),\tilde{\gamma}_{c,\eta}(\tilde{q}_{c}))=\Delta ( \{ (\alpha_c(\tilde{f}_{\eta}), \tilde{\gamma}_{c,\eta}(\tilde{f}_{c})) \}_2 \otimes  (\alpha_c(\tilde{g}_{\eta}), \tilde{\gamma}_{c,\eta}(\tilde{g}_{c}))  ),
$$
with 
$ \{ (\alpha_c(\tilde{f}_{\eta}), \tilde{\gamma}_{c,\eta}(\tilde{f}_{c})) \}_2 \otimes  (\alpha_c(\tilde{g}_{\eta}), \tilde{\gamma}_{c,\eta}(\tilde{g}_{c}))  \in  B_{2} ( \hat{K}_{C_2,c},(t^2)) \otimes  \beta(\hat{K}_{C_2,c},(t^2)).$ The relative Bloch group relation for $\omega,$ namely Corollary  \ref{relblochom}, then implies that $\omega(\alpha_c(\tilde{p}_{\eta}),\tilde{\gamma}_{c,\eta}(\tilde{q}_{c}),id)=0.$ Hence
$$
\rho(p)=\sum_{c \in |C|}{\rm Tr_{k}} \ell (res\, \tilde{q}_{c}).
$$
Let us first look at the expression $res\, \tilde{q}_{c}=res\, (1-\tilde{f}_{c}  ) \wedge \tilde{f}_c \wedge \tilde{g}_c.$  Since $\Delta(\{\tilde{f}_{c} \}_{2} \otimes \tilde{g}_{c} )= (1-\tilde{f}_{c}  ) \wedge \tilde{f}_c \wedge \tilde{g}_c, $  the commutative diagram in \textsection \ref{locressection}, implies that 
$$
res\, \tilde{q}_{c}=\delta \circ res(\{\tilde{f}_{c} \}_{2} \otimes \tilde{g}_{c} ) \in \Lambda ^{2}(\pazocal{O}_{C_2,c}/(\tilde{\pi}_c))^{\times}=\Lambda ^{2}k(c)_2^{\times},
$$  
where the last equality is written using the {\it canonical} isomorphism between the $k$-algebras $\pazocal{O}_{C_2,c}/(\tilde{\pi}_c)$ and $k(c)_2.$ Applying $\ell$ to the above expression we obtain, 
$$
\ell (res\, \tilde{q}_{c})=\ell \circ\delta \circ res(\{\tilde{f}_{c} \}_{2} \otimes \tilde{g}_{c} ).
$$
On the other hand, $\ell \circ\delta= \ell i _{2},$ and $\ell i_{2}$ depends on its argument only  modulo $(t^2),$ as was described in the paragraph before this proposition.  Therefore, 
$$
\ell (res\, \tilde{q}_{c})=\ell i _{2,c} \circ res(\{f_{c} \}_{2} \otimes g_{c} ),
$$
since $\tilde{f}_{c}$ and $\tilde{g}_c$ are liftings of $f_c$ and $g_{c}.$ Summing over all $c \in |C|$ and taking traces, we obtain 
$$
\rho \circ \Delta (\{ f\}_{2} \otimes g)=\rho(p)=\ell i _{2,|C|}\circ res (\{ f\}_{2} \otimes g).
$$
This finishes the proof. 
\end{proof}
 
\subsection{Strong reciprocity conjecture}

In this section, we assume that $k$ is algebraically closed. The following theorem of Suslin is proven in much more generality in \cite{sus}. 

\begin{theorem*} (Suslin Reciprocity)
Let $C/k$ be a smooth and projective curve over $k,$ then the sum of the residue maps
\[
\begin{CD}
K_{3} ^{M} (k(C)) @>{\oplus res_c}>> \oplus _{c \in |C|} K_{2} ^{M} (k) @>{\sum}>> K_{2} ^{M}(k)
\end{CD}
\]
is 0. 
\end{theorem*}

There are Bloch complexes of  weight $n$ defined by Goncharov which conjecturally compute motivic cohomology of weight $n.$ Let $K$ be a field then 
$$
\Gamma(K,n): B_{n}(K) \to B_{n-1}(K)\otimes K^{\times} \to \cdots B_{2}(K) \otimes \Lambda^{n-2} K^{\times} \to \Lambda^n K^{\times}. 
$$
The cohomological complex $\Gamma(n,K)$ is  concentrated in degrees $[1,n],$ with $B_{n}(K)$ in degree 1 and conjecturally $H^{i}(\Gamma(n,K)_{\mathbb{Q}}) \simeq H ^{i}_{\pazocal{M}}(K,\mathbb{Q}(n)).$ Here  $H_{\pazocal{M}} ^{\cdot} (K,\mathbb{Q}(n))$ denotes motivic cohomology of weight $n.$ By Voevodsky's  description in terms of algebraic cycles and the Bloch-Grothendieck-Riemann-Roch theorem,  $H_{\pazocal{M}} ^{i} (K,\mathbb{Q}(n)) \simeq K_{2n-i} (K)^{(n)} _{\mathbb{Q}} .$ This is known when $i=n,$  since then $H^{n}(\Gamma(n,K))\simeq K_{n} ^{M}(K)_{\mathbb{Q}}\simeq K_{n}(K)_{\mathbb{Q}}^{(n)}.$ 

 In what follows only the Bloch group $B_2$ will be relevant. There are several different definitions of these groups and it is not known that these definitions are equivalent. Note that the definition of $B_{2} (A) $ that  we have been using is the free abelian group generated by $\{ a\}_{2},$ with $a \in A^{\flat}:=\{a|a, \, 1-a \in A^{\times} \},$  modulo the 5-term functional equation of the dilogarithm.    

We have a residue map $res_{c}: \Gamma(k(C),3) \to \Gamma(k,2),$ for every $c \in |C|.$ Summing them over all $c \in |C|,$ we obtain the commutative diagram
$$
\xymatrix{B_3(k(C)) \ar[r]  & B_2(k(C))\otimes k(C)^{\times}  \ar[r] ^-{ \Delta}\ar[d] ^{res_{|C|}}&\Lambda^{3}k(C)^{\times} \ar[d]^{res_{|C|}} \\ & B_{2}(k) \ar[r]^{\delta} & \Lambda^{2}k^{\times}.}
$$
By the Suslin reciprocity theorem $Im(res_{|C|}) \subseteq Im (\delta) $ in $\Lambda^{2} k^{\times}.$

\begin{conjecture*} (Strong reciprocity conjecture) The residue map $res_{|C|}: \Gamma(k(C),3) \to \Gamma(k,2),$ is  homotopic to 0. More precisely, there is a map $h,$ with $h(k^{\times } \wedge \Lambda^{2} k(C) ^{\times})=0,$ which gives a commutative diagram:  
$$
\xymatrix{B_3(k(C)) \ar[r]  & B_2(k(C))\otimes k(C)^{\times}  \ar[r] ^-{ \Delta}\ar[d] ^{res_{|C|}}&\Lambda^{3}k(C)^{\times} \ar[d]^{res_{|C|}} \ar@{.>}[dl] _{h} \\ & B_{2}(k) \ar[r]^{\delta} & \Lambda^{2}k^{\times}.}
$$
\end{conjecture*}

We need some preliminary results before proving the infinitesimal version of the strong reciprocity conjecture. 

\begin{proposition}\label{ressum2}
The sum of the local residue maps $res_{|C|}: \Lambda^{2} k(C_2,\mathcal{P}_{2}) ^{\times} \to  k_{2} ^\times$ is 0.
\end{proposition}

\begin{proof}
 This last statement can be reduced to two statements since $k_{2} ^{\times} \simeq k^{\times} \oplus k. $ Fix $f \wedge g \in \Lambda^{2} k(C_2,\mathcal{P}_{2}) ^{\times}.$  That the map $\Lambda^{2} k(C_2,\mathcal{P}_{2}) ^{\times} \to  k ^\times$ is zero follows from applying the classical Weil reciprocity to $C$ and the functions $f|_C$ and $g|_C.$ On the other hand, denote the map $\Lambda^{2} k(C_2,\mathcal{P}_{2}) ^{\times} \to  k $ by $res_{|C|} ^{\circ}.$  It is the sum of $res^{\circ} _{c},$ for $c \in |C|.$ Here $res_{c}^{\circ}(f\wedge g):=res_t \frac{1}{t} res_{\tilde{\pi}_c}(d \log \tilde{f}_c \wedge d \log \tilde{g}_c),$ where $\tilde{\pi}_c$ is a uniformizer lifting $\pi_{2,c}$ and $\tilde{f}_{c}$ and $\tilde{g}_c$ are $\tilde{\pi}_c$-good local liftings of $f_c$ and $g_c.$  By the anti-commutativity of the residues we see that 
$$
res_{|C|} ^{\circ}=-\sum _{c \in |C|}res_{c}res_{t}(\frac{1}{t}d \log \tilde{f}_c \wedge d \log \tilde{g}_c).
$$
On the other hand,  we see that $res_{t}(\frac{1}{t}d \log \tilde{f}_c \wedge d \log \tilde{g}_c)$ defines a rational 1-form $\omega$ on $C$ independent of $c.$ 
The statement then follows from the fact that the sum of the residues of $\omega$ on $C$ being 0. 
\end{proof}

The following is an infinitesimal version of the Suslin reciprocity theorem. 
\begin{proposition}\label{infsuslin}
The composition $\Lambda^{3}k(C_2,\mathcal{P}_{2}) ^{\times} \xrightarrow{res_{|C|}} (\Lambda^{2}k_2^{\times})^{\circ} \to K_{2} ^{M}(k_2) ^{\circ}$ is 0. 
\end{proposition}

\begin{proof} 
There is an isomorphism $K_{2} ^{M}(k_2) ^{\circ}\simeq \Omega^{1} _{k/\mathbb{Q}} ,$ given by sending  $\{ a_1, a_2\}$ to $res_{t} \frac{1}{t} d\log(a_1) \wedge d \log (a_{2}).$  
For any $c \in |C|,$ the composition 
$$
\Lambda^{3}k(C_2,\mathcal{P}_{2}) ^{\times}  \xrightarrow{res_c} (\Lambda^{2}k_2^{\times})^{\circ}  \to \Omega^{1} _{k/\mathbb{Q} } 
$$ 
is given by sending $f\wedge g \wedge h$ to $res_{t}(res_{\tilde{\pi}_c}(\frac{1}{t} d \log(\tilde{f}_c) \wedge d \log (\tilde{g}_c) \wedge d \log (\tilde{h} _c))),$ where $\tilde{\pi}_{c}$ is any  unifomizer on a lifting of  $\hat{\pazocal{O}}_{C_{2},c},$ which lifts $\pi_{2,c}$ and $\tilde{f}_{c},\, \tilde{g}_{c}$ and $\tilde{h}_{c}$ are any $\tilde{\pi}_{c}$-good local liftings of $f, \, g $  and $h. $ Therefore, to prove the statement we need to show that the sum 
\begin{eqnarray}\label{suslinsum}
\sum _{c \in |C|} res_{t}(res_{\tilde{\pi}_c}(\frac{1}{t} d \log(\tilde{f}_c) \wedge d \log (\tilde{g}_c) \wedge d \log (\tilde{h} _c)))=-\sum _{c \in |C|} res_{c}(\tilde{\omega}_c)
\end{eqnarray}
is 0, where $\tilde{\omega}_{c}:=res_{t}(\frac{1}{t} d \log(\tilde{f}_c) \wedge d \log (\tilde{g}_c) \wedge d \log (\tilde{h} _c)). $ On the other hand, $\omega_c$ clearly depends only on $\tilde{f}_{c},\, \tilde{g}_c$ and $\tilde{h}_{c}$ modulo $t^2$ and hence $\tilde{\omega}_c$ is the restriction of an (absolute) rational 2-form $\omega$ on $C.$ Since the last expression in (\ref{suslinsum}) is the sum of all the residues of the  2-form $\omega$ on the curve $C,$ it is equal to 0. 
\end{proof}

\begin{proposition}\label{rho0}
The map $\rho: \Lambda^{3}k(C_{2},\mathcal{P}_2)^{\times} \to k,$  when restricted to $k_{2} ^{\times} \wedge \Lambda^{2}k(C_{2},\mathcal{P}_2)^{\times},$ is $0.$
\end{proposition}

\begin{proof}

Let $r \in k_2  ^{\times} $ and $g, \, h \in k(C_{2},\mathcal{P}_2)^{\times} .$  In order to compute $\rho(r\wedge g \wedge h),$ we need to choose local and  generic liftings of $r, \, g,$ and $h.$ Let us choose a lifting $\tilde{r} \in k[[t]]$ of $r.$ Then we choose the  generic lifting and the good local  liftings all to be equal to $\tilde{r}.$ We also  choose  a generic lifting and good local liftings of $g$ and $h$ arbitrarily, satisfying the hypotheses in the definition (\ref{regcurves}) of $\rho .$ Let us put $\tilde{r}=r_{0}e^{r_{1}t+\tilde{r}_{2}t^2+\cdots},$ with $r_{0} \in k^{\times}$ and $r_{1},\, \tilde{r}_{2}, \cdots \in k.$

First we simplify the expression $\sum _{c \in |C|}\ell(res_c(\tilde{q}_c)), $ where $\tilde{q}_c=(\tilde{r},\tilde{g}_c, \tilde{h}_c).$ Note  that $res_c(\tilde{q}_c)=-\tilde{r} \wedge res_{c}(\tilde{g}_c \wedge \tilde{h}_c).
$
For $\alpha \in k[[t]],$ let $\alpha[t^n]$ denote the coefficient of $t^n$ in $\alpha.$ Then we have, 
$$
\ell (res_c(\tilde{q}_c))=-\ell (\tilde{r} \wedge res_{c}(\tilde{g}_c \wedge \tilde{h}_c))=-\tilde{r}_{2} \cdot \log ^{\circ} (res_{c}(\tilde{g}_c \wedge \tilde{h}_c))[t]+r_{1}\cdot \log ^{\circ} ( res_{c}(\tilde{g}_c \wedge \tilde{h}_c))[t^2].
$$
We can rewrite $\log ^{\circ} (res_{c}(\tilde{g}_c \wedge \tilde{h}_c))[t]$ in terms of residues as:
$$\log ^{\circ} (res_{c}(\tilde{g}_c \wedge \tilde{h}_c))[t]=res_{t} \frac{1}{t}d \log ( res_{c}(\tilde{g}_c \wedge \tilde{h}_c))=res_{t} \frac{1}{t}( res_{c}(d \log (\tilde{g}_c) \wedge d \log (\tilde{h}_c))).$$
By the anti-commutativity of the residues, the last expression can be replaced with 
$$
-res_{c} ( res_{t}(\frac{1}{t}d \log (\tilde{g}_c) \wedge d \log (\tilde{h}_c))).
$$
Similar to the  arguments in the preceding propositions, $res_{t}(\frac{1}{t}d \log (\tilde{g}_c) \wedge d \log (\tilde{h}_c))$ defines a rational 1-form on $C,$ independent of $c,$ and hence the sum of its residues on the curve $C$ is equal to 0. Therefore, 
\begin{eqnarray}\label{firstsimp}
\sum _{c \in |C|}\ell (res_c(\tilde{q}_c))=r_1 \sum _{c\in |C|} \log ^{\circ} ( res_{c}(\tilde{g}_c \wedge \tilde{h}_c))[t^2].
\end{eqnarray}

Let $\alpha$ be  a generic lifting map  $\alpha:\tilde{\pazocal{A}}/(t^2)\xrightarrow{\sim}\pazocal{O}_{C_{2},\eta};$ and $\gamma_c,$ the local  lifting maps $\gamma_{c}:\widetilde{\pazocal{B}^{\circ}_{c}}/(t^2) \xrightarrow{\sim} \hat{\pazocal{O}}_{C_{2},c},$ as in the definition (\ref{regcurves}).
Let us denote a generic lifting of $(r,g,h)$ by $\tilde{p}_{\eta}:=(\tilde{r}, \tilde{g}_{\eta},\tilde{h}_{\eta}).$ Note that we denoted the local liftings by $\tilde{q}_{c}:=(\tilde{r},\tilde{g}_{c},\tilde{h}_{c}).$ 
Let $\psi: \underline{\tilde{\pazocal{A}}} \to \tilde{\pazocal{A}}$ be a splitting and $\overline{\psi}: \underline{\tilde{\pazocal{A}}}[[t]] \xrightarrow{\sim } \tilde{\pazocal{A}},$ the corresponding isomorphism.

Then, by the definition of $\omega,$ we have 
\begin{eqnarray}\label{secondsimp}
res_{c} \omega(\tilde{p}_{\eta},\tilde{q}_{c} ,  \tilde{\gamma}_{c,\eta} ^{-1} \circ \alpha_c )=res_{c}(\Omega(\overline{\psi}^{-1} (\tilde{p}_{\eta}) , \overline{\psi}^{-1} (\overline{\chi}_{c}^{-1} (\tilde{q}_{c}))),\end{eqnarray}
where $\overline{\chi}_{c}$ is any lifting of the isomorphism  $\tilde{\gamma}_{c,\eta} ^{-1} \circ \alpha_{c}.$ We will simplify the last expression. Let us put $\tilde{q}:=\overline{\psi}^{-1} (\tilde{p}_{\eta})$ and $\hat{q}:= \overline{\psi}^{-1} (\overline{\chi}_{c}^{-1} (\tilde{q}_{c})).$ Using the notation in 
(\ref{defomega}), we have $\tilde{q}:=(\tilde{y}_1,\tilde{y}_2,\tilde{y}_3),$ $\hat{q}:=(\hat{y}_1,\hat{y}_2,\hat{y}_3),$ with  $\tilde{y}_{1}=\hat{y}_{1}=\tilde{r}.$ Since $\Omega$ is a 1-form relative to $k,$ the terms involving $d \log (r_{0})$ in the expression of $\Omega$ will vanish. Also since $\tilde{y}_1=\hat{y}_1,$ we see also that the terms involving $\tilde{r}_{2}$ vanish. Therefore, $\Omega$ contains only the terms that have the factor $r_{1}$ in them. More precisely, 
$$
\Omega(\tilde{q},\hat{q})=r_{1} \big( d \log (\alpha_{02}) (\tilde{\alpha}_{23} -\hat{\alpha}_{23}) -  d \log (\alpha_{03}) (\tilde{\alpha}_{22} -\hat{\alpha}_{22})\big).
$$ 
Taking residues we have
\begin{eqnarray}\label{thirdsimp}
\;\;\;\;\;\;\;\;res_c \, \Omega(\tilde{q},\hat{q})=r_{1}\cdot res_{c}( d \log (\alpha_{02})  \tilde{\alpha}_{23}- d \log (\alpha_{03})\tilde{\alpha}_{22}  )-r_{1} \cdot \log ^{\circ} ( res_{c}(\tilde{g}_c \wedge \tilde{h}_c))[t^2],
\end{eqnarray}
where we replaced the term $res_{c} (d \log (\alpha_{02})  \hat{\alpha}_{23}- d \log (\alpha_{03}) \hat{\alpha}_{22}) $ with $ \log ^{\circ} ( res_{c}(\tilde{g}_c \wedge \tilde{h}_c))[t^2]$ using the fact that $\tilde{q}_{c}$ is a good lifting. 

By the definition  (\ref{regcurves}) of $\rho,$ we have 
$$
\rho(r\wedge g \wedge h)=\sum _{c \in |C|} (\ell (res_{c}(\tilde{q}_c)) +res_{c}\omega (\tilde{p}_{\eta},\tilde{g}_{c},\tilde{\gamma}_{c} ^{-1} \circ \alpha_{c})).
$$
Using (\ref{firstsimp}), (\ref{secondsimp}), and (\ref{thirdsimp}), we can rewrite this as 
$$
r_1\sum _{c \in |C| } res_{c}( d \log (\alpha_{02})  \tilde{\alpha}_{23}- d \log (\alpha_{03})\tilde{\alpha}_{22}  ).
$$
However, $ d \log (\alpha_{02})  \tilde{\alpha}_{23}- d \log (\alpha_{03})\tilde{\alpha}_{22}$ is a rational 1-form on $C,$ and hence the sum of its residues is equal to 0. 
This finishes the proof. 
\end{proof}

In order to  prove  the infinitesimal version of the strong reciprocity conjecture,  we then put together the results we have proved so far about $\rho$ in this paper and about the Bloch group $B_{2}(k_2)$ in \cite{Unver}. 
\begin{theorem} (Infinitesimal strong reciprocity conjecture) The infinitesimal part of the residue map 
$$
res_{|C|}: \Gamma(k(C_{2},\mathcal{P}_{2}),3) \to \Gamma(k,2)^{\circ}
$$
is homotopic to 0. More precisely, there is a map $h: \Lambda^{3}k(C_2,\mathcal{P}_{2}) ^{\times} \to B_{2}(k_2)^{\circ}$ which makes the diagram 
$$
\xymatrix{ B_2 (k(C_2,\mathcal{P}_{2}))\otimes k(C_2,\mathcal{P}_{2})^{\times}  \ar[r] ^-{ \Delta }\ar[d] ^{res_{|C|}}&\Lambda^{3}k(C_2,\mathcal{P}_{2}) ^{\times} \ar[d]^{res_{|C|}} \ar@{.>}[dl] _{h} \\  B_{2}(k_2)^{\circ} \ar[r]^{\delta^{\circ}} & (\Lambda^{2}k_2^{\times})^{\circ}}
$$
commute and has the property that $h(k_{2} ^{\times} \wedge \Lambda ^{2} k(C_2,\mathcal{P}_{2}) ^{\times})=0.$ \end{theorem}

\begin{proof}
By the theorem in \cite{Unver}, $B_{2}(k_2)^{\circ} \simeq ker(\delta^{\circ}) \oplus Im (\delta ^{\circ}),$ and $ker(\delta ^{\circ}) \xrightarrow{\sim} k,$ such that the composition with the projection $B_{2}(k_2)^{\circ} \to ker(\delta ^{\circ}),$ is the same as the additive dilogarithm  $\ell i_{2}: B_{2}(k_2)^{\circ} \to  k.$  Suppose that we have such an $h$ as in the statement of the theorem. Let us denote the first component of $h $ by $h_{1}$  and the second component by $h_{2}.$ Then $h_{1}:  \Lambda^{3}k(C_2,\mathcal{P}_{2}) ^{\times} \to ker(\delta ^{\circ})$ and $h_{2}:  \Lambda^{3}k(C_2,\mathcal{P}_{2}) ^{\times} \to Im(\delta ^{\circ}).$  By the commutativity of the lower triangle we must have that $h_{2}=res_{|C|}.$ In order to be able to choose $h_{2}:= res_{|C|},$ we need to know that $Im(res_{|C|}) \subseteq Im (\delta^{\circ}). $ Since the cokernel of $\delta^{\circ}$ is $K_{2} ^{M} (k_2) ^{\circ},$ this is equivalent to Proposition  \ref{infsuslin}.

We know that by Proposition \ref{regdilog}, $ \ell i_{2,|C|} \circ res=\rho \circ \Delta. $ Clearly, from the definitions, we also have $ \ell i_{2,|C|} \circ res=\ell i _{2} \circ res_{|C|}.$ Therefore,  we have $\ell i _{2} \circ res_{|C|}=\rho \circ \Delta ^{\circ}.$ By the commutativity of the upper triangle, we see then that we must have $\ell i_{2}\circ (h_1\circ \Delta ) =\ell i _{2} \circ res_{|C|}=\rho \circ \Delta .$ Therefore, we choose $h_1:= \ell i_{2} ^{-1}\circ \rho, $ with $\ell i _{2} ^{-1}: k \to ker(\delta) ^{\circ}.$ These choices of $h_{1}$ and $h_2$ give us an $h$ that makes the diagram commutative. 

In order to show that $h(k_{2} ^{\times} \wedge \Lambda ^{2} k(C_2,\mathcal{P}_{2}) ^{\times})=0,$ we need to show that $h_1(k_{2} ^{\times} \wedge \Lambda ^{2} k(C_2,\mathcal{P}_{2}) ^{\times})=0,$ and $h_2(k_{2} ^{\times} \wedge \Lambda ^{2} k(C_2,\mathcal{P}_{2}) ^{\times})=0.$ The first statement follows from $\rho(k_{2} ^{\times} \wedge \Lambda ^{2} k(C_2,\mathcal{P}_{2}) ^{\times})=0,$ proved in Proposition \ref{rho0}. For the second statement we need to show that $res_{|C|}(k_{2} ^{\times} \wedge \Lambda ^{2} k(C_2,\mathcal{P}_{2}) ^{\times})=0.$ If $r \in k_{2} ^{\times }$ then $res_{c}(r\wedge f \wedge g)=-r\wedge res_{c}(f \wedge g).$ Therefore, it suffices to prove that $res_{|C|}: \Lambda^{2} k(C_2,\mathcal{P}_{2}) ^{\times} \to  k_{2} ^\times$ is 0. This is precisely Proposition   \ref{ressum2}.
\end{proof}
\subsection{Comparison with the additive dilogarithm} 
In this section, we no longer assume that $k$ is algebraically closed. We will show that when the infinitesimal regulator $\rho$ in this paper  is restricted to the case of $\mathbb{P}^{1} _{2}:=\mathbb{P}^{1} _{k^2}$ and the rational functions are restricted to those of  degree 1, then we obtain the additive dilogarithm $\ell i _2$ of \cite{Unver}. We first compute the regulator $\rho$ for the triple of rational functions on $\mathbb{P}^{1} _{2}$ that is related to the Totaro cycle. 
\begin{lemma}\label{totaro} For $a \in k_{2} ^{\flat},$ the triple of rational functions $(1-x,x,1-\frac{a}{x})$ on  the projective line $\mathbb{P}^{1} _{2}$ over $k_{2},$ with the parameter $x,$ are good with respect to a system of uniformizers $\mathcal{P}_{2}$ and  the regulator value on the  triple $(1-x)\wedge x \wedge (1-\frac{a}{x}) \in \Lambda^{3} k(\mathbb{P}^{1} _{2},\mathcal{P}_{2})^{\times}$  is given by 
$$
\rho((1-x)\wedge x \wedge (1-\frac{a}{x}))=\ell i_{2}(\{ a \}_2 ),
$$
where $\{ a \}_2$ is the element corresponding to $a \in k^{\flat}$ in the Bloch group $B_{2} (k_{2}).$ 
\end{lemma}

\begin{proof}
The first statement follows by choosing the uniformizers $x, \, x-1, \, x-a, $ and $1/x$ at the points $0, \, 1, \, a,$ and  $\infty.$ For the second statement, note that $\Delta(\{ x\}_{2} \otimes (1-\frac{a}{x}))=(1-x) \wedge x \wedge (1-\frac{a}{x})$ and  by  Proposition \ref{regdilog}, $\rho(\Delta(\{ x\}_{2} \otimes (1-\frac{a}{x})))= \ell i_{2, |\mathbb{P}^1|} (res  (\{ x\}_{2} \otimes (1-\frac{a}{x})) )=\ell i_{2}(\{ a \}_2).$
\end{proof}

For $B=k_m$ or $A,$ we let $ \mathbb{P}^{n} _{B}=Proj (B[z_0,\cdots, z_{n}]).$    For $0 \leq i \leq n,$ let $H_{i}$ denote the hyperplane defined by $z_{i}=0.$ Let $z_{l} ^{q} (\mathbb{P}^{n} _{B}) $ denote the free abelian group generated by codimension $q$ linear subspaces of $\mathbb{P}^{n} _{B},$ which are in general position with respect to the  intersections $\cap _{i \in I}H_{i},$ for all $I \subseteq \{0,1, \cdots, n \}.$ We do {\it not} assume that the linear subspaces under consideration come from linear subspaces of $\mathbb{P}^{n} _{k}$ by pull-back. In other words, they are not necessarily constant on $B.$  For $0 \leq i \leq n,$ let  $\partial_{i} : z_{l} ^{q} (\mathbb{P}^{n} _{B}) \to z_{l} ^{q} (\mathbb{P}^{n-1} _{B}) $ denote the map induced by sending $L$ to $L\cap H_{i} .$ The map $\partial :=\sum _{0 \leq i \leq n} (-1) ^i \partial _i$ then defines a complex $z_{l} ^{q} (\mathbb{P}^{\bullet} _{B}) $.    We denote  $\mathbb{P}^{n} _{k_2}$ by $\mathbb{P}^{n} _{2}.$
 
\begin{definition}
Restricting our $\rho$ to linear objects  we obtain a map 
$$
\rho_l: z_{l} ^{2} (\mathbb{P}^{3} _{2}) \to k,
$$
as follows. Given a line $L_2 $ in $\mathbb{P}^{3} _{2},$ in general position with respect  to the intersections of the coordinate hyperplanes as above, then for a  system of uniformizers $\mathcal{P}_{2}$ on $L_2,$  $\frac{z_{1}}{z_{0}}\wedge \frac{z_{2}}{z_{0}}\wedge \frac{z_3}{z_0} \in \Lambda^{3} k(L_2,\mathcal{P}_2)^{\times}$ and we define $\rho_{l} (L_2):= \rho(\frac{z_{1}}{z_{0}}\wedge \frac{z_{2}}{z_{0}}\wedge \frac{z_3}{z_0}).$  
\end{definition}
We can give a more explicit description of the function $\rho_{l}$ above. We first define $\tilde{\rho}_{l}: z_{l} ^2(\mathbb{P}^{3} _{A}) \to k$ as follows. 
Suppose that $\tilde{L}$ is a line in $z_{l} ^2(\mathbb{P}^{3} _{A}),$ then there is a system of uniformizers $\mathcal{P}_{A}$ on $\tilde{L}$ such that $\frac{z_{1}}{z_{0}}\wedge \frac{z_{2}}{z_{0}}\wedge \frac{z_3}{z_0} \in \Lambda^{3} k(\tilde{L},\mathcal{P}_A)^{\times}.$ Let $L$ denote the reduction of $\tilde{L}$ modulo $(t).$ Then the value of $\tilde{\rho}_l$ on $\tilde{L}$ is  defined by 
\begin{eqnarray}\label{linrho}
\tilde{\rho}_{l}(\tilde{L}):=\sum _{c \in |L|}{\rm Tr}_{k} (\ell (res_{c} \frac{z_{1}}{z_{0}}\wedge \frac{z_{2}}{z_{0}}\wedge \frac{z_{3}}{z_{0}})  )=\sum _{c \in   |L| \cap |H|} \ell (res_{c} \frac{z_{1}}{z_{0}}\wedge \frac{z_{2}}{z_{0}}\wedge \frac{z_{3}}{z_{0}}  ),
\end{eqnarray}
where $|H|:=\cup_{0 \leq i \leq 3} |H_i|.$ The last equality follows because the singularities of the above triple of functions lie only on $|L| \cap |H|$ and that the intersections $\tilde{L} \cap H_{i},$ being the intersections of linear subspaces, are $A$-rational and hence the values of the function $\ell $ in the summands are already in $k$ and there is no need to take trace.  

\begin{proposition}
The natural reduction mod $(t^2)$ map from $z_{l} ^2(\mathbb{P}^{3} _{A})$ to $z_{l} ^2(\mathbb{P}^{3} _{2})$ is surjective and the map $\tilde{\rho}_l$ factors through this map to give $\rho_l.$ 
\end{proposition}

\begin{proof} 
Given a line in $L_2$ in $\mathbb{P}^{3} _{2},$ which is in general position with respect to the intersections of the coordinate hyperplanes.  Let $\tilde{L} $ be a linear subspace in $\mathbb{P}^{3} _{A}$ which reduces to $L$ modulo $(t^2).$ Then $\tilde{L} \in z_{l} ^2(\mathbb{P}^{3} _{A}).$ This proves the first statement.   For the second statement we need to compute $\rho_l(L)=\rho(\frac{z_{1}}{z_{0}}\wedge \frac{z_{2}}{z_{0}}\wedge \frac{z_{3}}{z_{0}}).$ But $\tilde{L}$ together with the triple of rational  functions $\frac{z_{1}}{z_{0}}\wedge \frac{z_{2}}{z_{0}}\wedge \frac{z_{3}}{z_{0}}$ on $\tilde{L}$ gives a good global lifting. Therefore when computing  $\rho$ on $L_{2},$ we can use Corollary \ref{explicitrho} with $i=1.$ This proves that $\rho_{l}(L_{2})=\tilde{\rho}_{l}(\tilde{L})$ as desired. 
\end{proof}

\begin{proposition}
The regulator  $\rho_l: z_{l} ^{2} (\mathbb{P}^{3} _{2}) \to k$ factors through the quotient $z_{l} ^{2} (\mathbb{P}^{3} _{2})/\partial(z_{l} ^{2} (\mathbb{P}^{4} _{2})) .$
\end{proposition}

\begin{proof}
By the previous proposition and the fact that the map $z_{l} ^{2}(\mathbb{P}^{4} _{A}) \to z_{l} ^{2}(\mathbb{P}^{4} _{2})$ is a surjection, it suffices to show that $\tilde{\rho}_{l}$ factors through $z_{l} ^{2} (\mathbb{P}^{3} _{A})/\partial(z_{l} ^{2} (\mathbb{P}^{4} _{A})) .$ By the definition (\ref{linrho}) of $\tilde{\rho}_l,$ we see that 
$$
\tilde{\rho}_{l}(\tilde{L})=\ell\big(\big(\frac{z_1}{z_0}\wedge \frac{z_2}{z_0} \big)\Big|_{\tilde{L} \cap H_{3}}\big)-\ell\big(\big(\frac{z_1}{z_0}\wedge \frac{z_3}{z_0} \big)\Big|_{\tilde{L} \cap H_{2}}\big)+\ell\big(\big(\frac{z_2}{z_0}\wedge \frac{z_3}{z_0} \big)\Big|_{\tilde{L} \cap H_{1}}\big)-\ell\big(\big(\frac{z_2}{z_1}\wedge \frac{z_3}{z_1} \big)\Big|_{\tilde{L} \cap H_{0}}\big).
$$
For a scheme $X/A,$ let $X(A)$ denote the set of $A$-valued points of $X.$  Note that we can express $\tilde{\rho}_{l}$ in terms of the map 
$$
f: \mathcal{H}:=\cup_{0 \leq i \leq 3} H_{i} (A) \setminus \cup _{0 \leq i <j \leq 3} (H_{i}\cap H_{j})(A) \to k
 $$
 defined by $f([a_{0},\cdots,0 , \cdots, a_{2}]):=\ell (\frac{a_1}{a_0}\wedge \frac{a_2}{a_0}),$ as follows. First, we extend $f$ linearly to a map from the free abelian group $\mathbb{Z}[\mathcal{H}]$ on $\mathcal{H}.$   Then we note that $\partial (\tilde{L}) \in \mathbb{Z}[\mathcal{H}]$ and 
 $$
 \tilde{\rho}_l(\tilde{L})=-f(\partial(\tilde{L})).
  $$
  The statement that $ \tilde{\rho}_l(\partial (\tilde{P}))=0,$ then follows from $\partial ^{2}=0.$  
\end{proof}

Let $\tilde{C}_{m}(\mathbb{P}^1 _B)$ denote the free abelian group generated by an $m$-tuple of  distinct points in $\mathbb{P}^1 _B (B),$ and  $C_{m}(\mathbb{P}^1 _B)$ be the co-invariants of this group with respect to the natural action of $PGL(2,B).$ Let $d_{i}: \tilde{C}_{m}(\mathbb{P}^1 _B) \to \tilde{C}_{m-1}(\mathbb{P}^1 _B),$ be the map that sends $(c_{0},\cdots ,c_{i},\cdots,c_{m-1})$ to $(c_{0},\cdots ,\hat{c}_{i},\cdots,c_{m-1}),$ and $d:=\sum _{0\leq i \leq m-1}(-1) ^i d_i.$ This defines a complex $(\tilde{C}_{\bullet}(\mathbb{P}^1 _B),d),$ which  descends to a complex $(C_{\bullet}(\mathbb{P}^1 _B),d).$

There is a commutative diagram 
\[
\begin{CD}
z_{l} ^{2}(\mathbb{P}^{4} _{B})@>{\partial}>>z_{l} ^{2}(\mathbb{P}^{3} _{B})\\
@V{f_5}VV  @V{f_4}VV\\
C_{5}(\mathbb{P}^1 _B)@>{d}>> C_{4}(\mathbb{P}^1 _B)
\end{CD}
\]
with surjective vertical maps $f_{4}$ and $f_{5}.$ The map $f_{4}$  is given by sending $\tilde{L}$ to the $4$-tuple of points $(\gamma(\tilde{L} \cap H_{0}),\gamma(\tilde{L} \cap H_{1}),\gamma(\tilde{L} \cap H_{2}),\gamma(\tilde{L} \cap H_{4}))$  on $\mathbb{P}^{1} _{B},$ where $\gamma: \tilde{L} \xrightarrow{\sim } \mathbb{P}^{1} _{B}$  is any isomorphism of $\tilde{L}$ with $\mathbb{P}^{1} _{B}.$ Since we are taking coinvariants with respect to the action of the projective general linear group, this is independent of $\gamma.$ 
Similarly, $f_{5}$ is defined by sending $\tilde{P}$ to $(\gamma(\tilde{L} \cap H_{0}),\gamma(\tilde{L} \cap H_{1}),\gamma(\tilde{L} \cap H_{2}),\gamma(\tilde{L} \cap H_{4}),\gamma(\tilde{L} \cap H_{4}) ),$ where $\tilde{L}$ is any  line in $\tilde{P}$ in general position with respect to the lines $\{ \tilde{P}\cap H_{i} \} _{0\leq i \leq 4} $  in $\tilde{P}$ and  $\gamma: \tilde{L} \xrightarrow{\sim } \mathbb{P}^{1} _{B}$  is any isomorphism of $\tilde{L}$ with $\mathbb{P}^{1} _{B}.$ The above commutative diagram  induces the following surjective map 
\begin{eqnarray}\label{defng}
g: z_{l} ^{2} (\mathbb{P}^{3} _{2})/\partial(z_{l} ^{2} (\mathbb{P}^{4} _{2})) \to C_{4}(\mathbb{P}^1 _2)/d(C_{5}(\mathbb{P}^1 _2)). 
\end{eqnarray}

\begin{proposition}
The linear regulator map $\rho_{l}: z_{l} ^{2} (\mathbb{P}^{3} _{2})/\partial(z_{l} ^{2} (\mathbb{P}^{4} _{2})) \to k,$ factors through the surjection $g$ in (\ref{defng}) and induces a map $\rho_{c}: C_{4}(\mathbb{P}^1 _2)/d(C_{5}(\mathbb{P}^1 _2)) \to k.$ 
\end{proposition}

\begin{proof}
In oder to prove the statement we need to prove that if $\tilde{L}$ and $\tilde{L} ' $ are in $z_{l} ^{2} (\mathbb{P}^{3} _{A})$ with the property that $f_4(\tilde{L})=f_4(\tilde{L} ')$ then $\rho_l (\tilde{L}_2)=\rho_{l}(\tilde{L} ' _{2}),$ where $\tilde{L}_{2}:=\tilde{L}|_{t^2}$ and $\tilde{L}_{2} '=\tilde{L}'|_{t^2}.$  The assumption  $f_4(\tilde{L})=f_4(\tilde{L} ')$ implies that there is an isomorphism 
$\alpha: \tilde{L}\xrightarrow{\sim} \tilde{L} '$ such that $\alpha(\tilde{L} \cap H_{i})=\tilde{L} ' \cap H_{i},$ for all $0 \leq i \leq 3.$ By the definition of $\rho_{l},$ we need to prove that 
\begin{eqnarray}\label{linearrhoeq}
\rho(\frac{z_1}{z_0}\wedge \frac{z_2}{z_0} \wedge \frac{z_3}{z_0} \big|_{\tilde{L}_2})=\rho(\frac{z_1}{z_0}\wedge \frac{z_2}{z_0} \wedge \frac{z_3}{z_0} \big|_{\tilde{L}_{2} '}),
\end{eqnarray} 
where our notation is for emphasizing that on the left hand side inside the parenthesis we consider triples of rational  functions on $\tilde{L}_2$ whereas on the right hand side it is on $\tilde{L}_2 '.$ 

By the functoriality of $\rho$ we have  $\rho(\frac{z_1}{z_0}\wedge \frac{z_2}{z_0} \wedge \frac{z_3}{z_0} \big|_{\tilde{L}_{2} '})=\rho(\alpha^*(\frac{z_1}{z_0}\wedge \frac{z_2}{z_0} \wedge \frac{z_3}{z_0} \big|_{\tilde{L}_{2} '})).$ On the other hand, for each $1 \leq i \leq 3,$ $\alpha^*(\frac{z_{i}}{z_0}\big|_{\tilde{L} '})$ has the same zeroes and poles as $\frac{z_{i}}{z_0}\big|_{\tilde{L} }.$ This implies that there exist $\lambda _{i} \in A^{\times}$ such that 
$\alpha^*(\frac{z_{i}}{z_0}\big|_{\tilde{L} '})=\lambda_i \cdot \frac{z_{i}}{z_0}\big|_{\tilde{L} },$ for all $1 \leq i \leq 3.$  Combining with the above, this implies that 
$$
\rho(\frac{z_1}{z_0}\wedge \frac{z_2}{z_0} \wedge \frac{z_3}{z_0} \big|_{\tilde{L}_{2} '})=\rho ( (\lambda_1 \frac{z_{1}}{z_0} \wedge \lambda_2 \frac{z_{2}}{z_0} \wedge \lambda_3 \frac{z_{3}}{z_0} ) \big|_{\tilde{L}_{2}}).
$$
By Proposition \ref{rho0}, we see that the right hand side is equal to $\rho ( ( \frac{z_{1}}{z_0} \wedge  \frac{z_{2}}{z_0} \wedge  \frac{z_{3}}{z_0} ) \big|_{\tilde{L}_{2}}).$ This implies (\ref{linearrhoeq}) and hence finishes the proof of the proposition. 
 \end{proof}

By Remark 3.8.2 in \cite{Unver}, there is a map from $B_{2}(k_{2})$ to $C_{4}(\mathbb{P}^1 _2)/d(C_{5}(\mathbb{P}^1 _2))$ which sends $\{ a \}_{2}$ to $(0,a,1,\infty)$ and this map is an isomorphism, when the source and the target are tensored with $\mathbb{Q}.$  Let us continue to denote the map from $B_{2}(k_{2}) \to k,$ which corresponds to $\rho_{c}$ via this isomorphism, by the same symbol. Then the following theorem proves that $\rho$ restricted to the case of the projective line and linear fractional transformations is essentially the same as the additive dilogarithm of \cite{Unver}.
 
 \begin{theorem}\label{compunver}
The map $\rho_{c}: B_{2}(k_{2}) \to k$ which is  induced by restriction of  the  regulator $\rho$ to the projective line, is given by 
$$
\rho_{c}(\{ a\}_{2})=-\ell i_{2}(\{ a \}_2),
$$
where $\ell i_{2} (\{ a_{0} +a_{1}t \} _2)=-\frac{a_{1} ^{3}}{2a_{0} ^2(1-a_{0})^2}.$ 
 \end{theorem}
 
 \begin{proof} 
By the definition of $\rho_{c}$ on $B_{2}(k_2),$ the left hand side of of the expression is $\rho_{c}(0,a,1,\infty). $ On the other hand, again by the definition of $\rho_{c}$ we can express it as the regulator of a triple of functions of $\mathbb{P}^{1} _{2}:$ 
$$
\rho_{c}(0,a,1,\infty)=\rho((1-\frac{a}{x}) \wedge (1-\frac{1}{x}) \wedge \frac{1}{x} )= -\rho((1-x )\wedge x \wedge  (1-\frac{a}{x}) )=-\ell i_{2}(\{ a \}_2),
$$
where the last equality follows from the Totaro cycle computation in Lemma \ref{totaro}.
  \end{proof}

\section{Infinitesimal regulator on algebraic cycles }

In this section,  we will construct a regulator map for cycles over $S,$ which has the properties that it vanishes on boundaries and   assumes the same value if the  cycles are equivalent modulo $(t^2),$ for an appropriately defined equivalence relation.

\subsection{Bloch's cubical higher Chow groups}

We first recall Bloch's definition \cite{bloch} of cubical higher Chow groups in the case of a smooth $k$-scheme $X/k.$ Let $\square_k:= \mathbb{P}^{1} _{k} \setminus \{ 1\}$ and $\square ^n _{k}$ the $n$-fold product of $\square_k$ with itself over $k, $ with the coordinate functions $y_1, \cdots, y_n.$ For  a smooth $k$-scheme $X,$ we let   $\square^n _{X} :=X \times_k \square_k ^n.$  A codimension 1 face of $\square^n _{X}$ is a divisor $F_{i} ^a$  of the form $y_{i}=a,$ for $1\leq i \leq n,$ and $a \in \{0,\infty \}.$ A face of $\square^n _{X}$ is either the whole scheme $\square^n _{X}$ or an arbitrary intersection of codimension 1 faces.

 Let $\underline{z}^q (X, n)$ be the free abelian group on the set of codimension $q,$ integral, closed subschemes $Z \subseteq  \square^n _{X}$ which are {\it admissable}, i.e.  which intersect each face properly on $\square^n _{X}.$ For each codimension one face $F_{i} ^a,$ and  irreducible $Z \in \underline{z} ^q (X, n)$, we let $\partial_i ^{a} (Z)$  be the cycle associated to the scheme $Z \cap F_{i} ^{a}.$ We let $\partial:= \sum_{i=1} ^n (-1)^n (\partial_i ^{\infty} - \partial_i ^0)$ on $\underline{z}^q (X, n)$. One checks immediately  that $\partial ^2  = 0.$ We therefore obtain a complex $(\underline{z}^q (X, \cdot),\partial).$

 Let  $\underline{z}^q (X, n)_{\rm degn}$ denote the subgroup of  degenerate cycles, i.e. sums of those obtained by pulling back via one of the standard projections $p_i:\square_X ^n \to \square_X ^{n-1},$ for $0 \leq i \leq n,$ which omits the $i$-th coordinate on $\square_X ^n$ and 
 $z^q (X, \cdot):= \underline{z}^q (X, \cdot)/ \underline{z}^q (X, \cdot)_{\rm degn}$ the corresponding non-degenerate complex. The complex $(z^q (X, \cdot), \partial)$  is called the \emph{higher Chow complex} of $X$ and its homology ${\rm CH}^q (X, n):= {\rm H}_n (z^q (X, \cdot))$ is the higher Chow group of $X$. It is a theorem of Voevodsky  that the higher Chow groups  ${\rm CH}^q (X, n)$  compute the motivic cohomology ${\rm H}_{\mathcal{M}} ^{2q-n} (X, \mathbb{Z}(q)),$ for  smooth varieties $X/k.$

\subsection{Cycles over $S$}  

In the following, we need to work with cycles over $S$ which have finite reduction in a certain sense. We define these groups as follows. Let $\overline{\square}_{k}:= \mathbb{P}^{1} _{k},$  $\overline{\square}_{k} ^{n},$ the $n$-fold product of  $\overline{\square}_{k} $ with itself over $k,$ and  $\overline{\square}_{S} ^{n} :=\overline{\square}_{k} ^{n} \times _k S.$ We define a subcomplex $\underline{z}^q _{f} (S, \cdot) \subseteq \underline{z}^q (S, \cdot)$, as  the subgroup generated by integral, closed subschemes $Z \subseteq \square_S ^n$ which are admissible in the above sense and have {\it finite reduction}, i.e. $\overline{Z}$  intersects each $s\times \overline{F}$ properly on $\overline{\square}_S ^n,$  for  every face $F$ of $\square^n_{k}.$  Modding out by degenerate cycles, we  have a complex $z^q_{f} (S, \cdot).$

We want to emphasize that the more familiar notion of  $s$-admissability is not enough for our purposes. This condition would require that the cycle  $Z \subseteq \square_S ^n$ intersect each    $S \times F$ and  $s\times F$  properly on $\square_S ^n,$  for  every face $F$ of $\square^n_{k}.$ For example, the cycle $(1+t,t)$ in $\underline{z}^{2} (S,2) $ is  $s$-admissable, but does not have finite reduction because its closure contains the point  $(1,0)$ over the special fiber. Our regulator will  not be defined on this cycle.

\subsection{Definition of the regulator} 
An irreducible cycle $p$ in $ \underline{z}_{f} ^2 (S,2)$ is  given by a closed point $p_{\eta} $ of  $\square ^{2} _{\eta}$ whose closure $\overline{p}$  in $\overline{\square} ^{2} _{S}$ does not meet $(\{ 0,\infty\} \times \overline{\square}_{S}) \cup ( \overline{\square}_{S} \times \{0, \infty \}).$ Let $\tilde{p}$ denote the normalisation of $\overline{p} $ and $\{s_1,\cdots, s_m \}$ the closed fiber of $\tilde{p}.$ We have  surjections $\hat{\pazocal{O}}_{\tilde{p},s_i} \to k(s_i).$ Since $k(s_i)/k$ is finite \'{e}tale there is a unique splitting $\sigma_{\tilde{p},s_i}:k(s_i)\to \hat{\pazocal{O}}_{\tilde{p},s_i}.$   We define $
\log ^{\circ} _{\tilde{p},s_i}: \hat{\pazocal{O}}_{\tilde{p},s_i} ^{\times} \to \hat{\pazocal{O}}_{\tilde{p},s_i},
$
by 
$$
\log ^{\circ} _{\tilde{p},s_i}(y)=\log(\frac{y}{\sigma_{\tilde{p},s_i}(y(s_i))}).
$$

Let 
\begin{eqnarray}\label{defnl} 
\;\; \; l(p):=\sum _{1 \leq i \leq m}{\rm Tr}_{k}\Big(res_{\tilde{p}, s_i}\Big(\frac{1}{t^3}\big(\log^{\circ} _{\tilde{p},s_i} (y_1) \cdot d\log(y_2)-\log^{\circ} _{\tilde{p},s_i}(y_2) \cdot d \log(y_1) \big) \Big) \Big).
\end{eqnarray}

%{\color{blue} Note that the multiplicity appears implicitly when the form is pulled back to the normalisation so it does not appear explicitly in the formula above. It would be a good idea to check that the commutativity of the residue  holds with this definition, i.e. when we pass to the normalisation. This is probably what is done in the classical case. }  

The following lemma shows that the value of $l$ on the graph of a function assumes  a familiar form.

\begin{lemma}\label{lemmal} Suppose that  $f_1(t), f_2(t) \in k[[t]] ^{\times}.$ Then $p:=(f_1(t),f_2(t))$ is an irreducible 0-cycle in $\underline{z}_{f} ^{2}(S,2)$ and  
$$
l(p)=\log^{\circ} (f_1)[t]\cdot  \log^{\circ} (f_2)[t^2]-\log^{\circ} (f_2)[t]\cdot  \log^{\circ} (f_1)[t^2].
$$ 
\end{lemma} 

\begin{proof}
Note that with the above notation $\tilde{p}=\overline{p}=p,$ and $m=1$ with $s_{1}=(f_1(0),f_2(0)).$  
Let us write $\log ^{\circ } (f_1)= a_{1}t+a_2t^2+\cdots $ and $\log^{\circ}(f_2)= b_{1}t+b_2t^2+\cdots.$   We have $\log^{\circ} _{\tilde{p},s_1} (y_i)=\log ^{\circ } (f_i),$ and $d \log (y_{i})=d \log ^{\circ} (f_i).$ The expression (\ref{defnl}) then takes the form 
$$
res_{t} \frac{1}{t^3} (   (a_{1}t+a_2t^2+\cdots)  (b_{1}+2b_2t+\cdots)- (b_{1}t+b_2t^2+\cdots)  (a_{1}+2a_2t+\cdots) )dt.
$$
This expression is equal to $a_{1}b_{2}-a_{2}b_{1},$ which is exactly the expression in the statement of the lemma. 
\end{proof}

\begin{definition}
We  will define a regulator 
$
\rho _f: \underline{z}_{f} ^2 (S,3) \to k
$ 
 as the composition $l  \circ \partial ,$ where 
$
l: \underline{z}_{f} ^2 (S,2) \to k
$
is the map defined in (\ref{defnl}).
\end{definition}

\subsection{Properties of the regulator} In this subsection, we prove that the regulator has the expected properties.

\subsubsection{Vanishing on the boundaries} By the construction of the regulator, it is easy to see that it is 0 on the boundaries. 

\begin{proposition}
 The composition $\rho_f\circ\partial: \underline{z}^{2} _{f} (S,4) \to k $ is 0. 
\end{proposition}

\begin{proof} 
This follows  from the fact that $\rho=l\circ\partial$ and that $\partial \circ \partial=0.$ 
\end{proof}

\subsubsection{Anti-symmetry} 

Let $G_{n}$ be the semi-direct product of $S_{n}$ with  $(\mathbb{Z}/2)^n,$ with $S_{n}$ acting on $(\mathbb{Z}/2)^n$ by permuting the factors. Let $\chi:G_{n} \to \mathbb{Z}/2\simeq \{-1, 1\}$ be the homomorphism  which restricts to identity on each $\mathbb{Z}/2$ factor and to the sign character on $S_n.$ There is a natural action of $G_{n}$ on $z^{q} (S,n),$ where $S_{n}$ permutes the coordinates in $\square^{n}_{S}$ and $\mathbb{Z}/2$ in the $i$-th coordinate in $(\mathbb{Z}/2)^n$ acts by switching $0$ and $\infty$ in the $i$-th coordinate in $\square ^n _S.$ Then the regulator has the following anti-symmetry property. 

\begin{proposition}
For $\sigma \in G_{3}$ and $Z \in z^{2} _f(S,3),$ $\rho_f(\sigma(Z))=\chi(\sigma)\rho_f(Z).$  
\end{proposition}  
 
\begin{proof} 
Note that by the description (\ref{defnl}), $l:z_{f} ^{2}(S,2) \to k$ is anti-symmetric with respect to the action of $G_{2}.$ If $\sigma _{i} \in (\mathbb{Z}/2)^3$ is the element which is non-trivial only in the $i$-th coordinate then $ (\partial _{i} ^0 -\partial _{i} ^{\infty}) \sigma _{i}(Z)=-(\partial _{i} ^0 -\partial _{i} ^{\infty}) (Z).$ For $j \neq i,$ $ l((\partial _{j} ^0 -\partial _{j} ^{\infty}) \sigma _{i}(Z))=l (\alpha_j(\sigma_i)((\partial _{j} ^0 -\partial _{j} ^{\infty}) (Z))),$ where 
$\alpha_j:(\mathbb{Z}/2)^3 \to (\mathbb{Z}/2)^2$ is the homomorphism which omits the $i$-th term. Then the anti-symmetry of $l$ gives the desired equality for $\sigma_i.$ To complete the proof we need to show anti-symmetry for $\sigma \in S_3,$ and hence only for $\sigma=(12)$ or $(23),$ since they generate $S_{3}.$ This is then seen by a direct computation.  
\end{proof}

\subsubsection{Modulus property}\label{modulus section}

 Suppose that $Z_i$ for $i=1,2$ are two irreducible cycles in $\underline{z}^{2} _{f} (S,3).$ We say that $Z_{1}$ and $Z_{2}$ are equivalent modulo $t^m$ if the following condition $(M_{m})$ holds:
 
 (i) $\overline{Z}_{i}/S$ are smooth with $(\overline{Z}_i)_{s} \cup (\cup_{j,a} |\partial _j ^{a} Z_i|) $  a strict normal crossings divisor on $\overline{Z}_i.$
  
 and more importantly 
 
 (ii) $\overline{Z}_{1}|_{t^m}=Z_{2} |_{t^m}.$

The main result of this section is the following:

\begin{theorem}\label{modulus theorem}
If $Z_{i} \in \underline{z} _{f} ^{2}(S,3),$ for $i=1,2,$  satisfy the condition $(M_{2})$ then they have the same infinitesimal regulator value: 
$$\rho_f (Z_{1})=\rho_f(Z_{2}).$$
\end{theorem}

\begin{proof} 
By assumption (i) in $(M_{2}),$ we see that the divisors $F_{j} ^{a}$ define a system of uniformizers $\mathcal{P}_S(i)$ on $\overline{Z}_i $ such that the triple of functions $(y_1 \wedge y_2 \wedge y_3)|_{\overline{Z}_i}$ are $\mathcal{P}_S(i)$-good. Therefore, if $\mathcal{P}_2(i)$ denote the reduction mod $(t^2)$ of this system of uniformizers to $\overline{Z}_{i}|_{t^2}$ then $(y_1 \wedge y_2 \wedge y_3)|_{\overline{Z}_{i}|_{t^2}} \in \Lambda^3 k(\overline{Z}_{i}|_{t^2},\mathcal{P}_2(i))^\times$ and $(y_1 \wedge y_2 \wedge y_3)|_{\overline{Z}_i}$ can be used as a global good lifting to compute $\rho((y_1 \wedge y_2 \wedge y_3)|_{\overline{Z}_{i}|_{t^2}})$ using Corollary \ref{explicitrho} with $r=1.$ More precisely, 
\begin{eqnarray}\label{regfinite}
\rho((y_1 \wedge y_2 \wedge y_3)|_{\overline{Z}_{i}|_{t^2}})=\sum _{z \in |(\overline{Z}_{i})_s|}{\rm Tr}_k \ell (res_{\tilde{\pi}_z(i)} ((y_1 \wedge y_2 \wedge y_3)|_{\overline{Z}_{i}})),
\end{eqnarray}
where $\{ \tilde{\pi}_z(i)|   z\in |(\overline{Z}_{i})_s|\}$ is the system of uniformizers $\mathcal{P}_{S}(i).$ 

The only contribution to the sum in (\ref{regfinite}) comes from when $z \in (\partial_j ^{a} Z_i)_s.$ The sum of the  contributions coming from   $z \in (\partial_j ^{a} Z_i)_s$  is $l(\partial_j ^{a} Z_i)$ by Lemma \ref{lemmal}.   Summing over all the faces we see that 
$\rho((y_1 \wedge y_2 \wedge y_3)|_{\overline{Z}_{i}|_{t^2}})=l\circ \partial (Z_{i})=\rho_{f}(Z_{i}).$ Since by the condition (ii) in $(M2),$ $\overline{Z}_{1}|_{t^2}=\overline{Z}_{2}|_{t^2},$ the left hand sides of the last expression is the same for $i=1$ and $i=2$ and therefore the right hand sides are the same. This gives the desired equality.  
\end{proof}

\begin{remark}\label{chow group}
By the same argument, one can  prove a somewhat stronger statement than the above. Namely, that we do not need to assume that $\overline{Z}_{i}/S$  are smooth with the the given divisors, being normal crossings divisors on them. But we assume that this holds after a common imbedded resolution of singularities for $\overline{Z}_i$ in $\overline{\square}^3 _{S}$ and these new smooth cycles are congruent modulo $t^2.$  Our main aim in doing this would be to define the correct higher Chow groups over $k_2,$ by defining the cycle complex to be cycles over $S,$ and setting two cycles to be the same if they have the same reduction modulo $t^2 $ after an appropriate resolution of singularities. We will pursue this approach in future work, the Milnor case of which is done in \cite{pu}.

\end{remark}

\subsubsection{Vanishing on the products} 
 Suppose that $Z$ is an irreducible cycle in $\underline{z}_{f} ^2 (S,3)$ which satisfies condition (i) in \textsection \ref{modulus section}.

\begin{proposition}\label{products}
If there is $1\leq i\leq 3$ such that $y_{i} $ restricted to $Z_{2}:=Z|_{t^2}$ is in $k_2 ^{\times}$ then $\rho_f(Z)=0.$ 
\end{proposition}

\begin{proof}
Without loss of generality assume that $i=1,$ and $y_{1}$ restricted to $Z|_{2}$ is $\alpha \in k_2 ^\times.$ By exactly as in the proof of Theorem \ref{modulus theorem} we see that 
$$\rho_f(Z)=\rho((y_{1}\wedge y_2 \wedge y_3)|_{Z_2})=\rho(\alpha \wedge( y_2 \wedge y_3)|_{Z_2}).$$
On the other hand, the right hand side is 0 by Proposition \ref{rho0}.
\end{proof}

\subsubsection{Comparison with Park's regulator} 

In this section, we will compare our construction to Park's construction in \cite{P1}. We first recall Park's construction in the cases that relate to our discussion. Let $\Diamond _n:=\mathbb{A}^{1} _k \times_k \square ^n _k$ and   $\overline{\Diamond}_n:=\mathbb{A}^{1} _k \times_k \overline{\square} ^n _k,$ with $t$ being the coordinate on $\mathbb{A}^{1} _k.$ The  codimension one face $F_{i} ^{a},$ for $1\leq i \leq n,$ and $a\in \{ 0,\,\infty \}$ is given by $y_i=a.$ For a cycle $Z \subseteq \Diamond ^n,$ its face $\partial _i ^a (Z)$ is defined as the cycle associated to $Z \cap F_{i} ^a.$  As usual $\partial := \sum _{1 \leq i\leq n} (\partial _i ^{0} -\partial _{i} ^{\infty}).$ We will consider two different versions of the additive Chow groups: one with the strong sup modulus condition, denoted by the subscript ``ss," and the other with the sup modulus condition, denoted by the subscript ``s." In \cite{P1}, only the one with the sup modulus condition is considered. 

The groups $c_{p,s}(\Diamond_n;2)$ (resp. $c_{p,ss}(\Diamond_n;2)$) are defined inductively:  

(a) $c_{0,s}(\Diamond_n;2)=c_{0,ss}(\Diamond_n;2)$ is defined to be the free abelian group on the set of  closed points on $\Diamond_n \setminus (\cup_{i,a} F_{i} ^a \cup \{ t=0\}).$ 

(b) $c_{p,s}(\Diamond_n;2)$ (resp. $c_{p,ss}(\Diamond_n;2)$) is defined to be the free abelian group on the set of irreducible $p$-dimensional closed subvarieties $W$ of $\Diamond_n$ which satisfy the following properties: 

$(i)$ $W$ intersects all the faces properly.

$(ii)$ For  $1\leq i \leq n$ and $a \in \{ 0,\, \infty\},$ $\partial _i ^a (W) \in c_{p-1,s}(\Diamond_{n-1};2)$ (resp. $\partial _i ^a (W) \in c_{p-1,ss}(\Diamond_{n-1};2)$).

Let $f:\overline{W} \to \overline{\Diamond}_n$ be the normalisation of the closure of $W$ in $\overline{\Diamond}_n.$  

Then for the sup modulus condition: 

$(iii)_s$ The divisor $\sup _{1\leq i \leq n} (f^{*}\{y_i=1 \})-2\cdot f^* \{t=0\}$ is an effective divisor on $\overline{W}.$ Here for a finite set of Weil divisors $\{ D_{i} \}_{1 \leq i \leq n }$ on a normal variety $X$, $\sup_{1\leq i \leq n}D_i$ is the divisor $D$ such that for any prime divisor $E$ on $X,$ order of $D$ along $E$ is the maximum of the orders of $D_{i}$ along $E.$

For the strong sup modulus condition:

$(iii)_{ss}$ There exists an $i \in \{1, \cdots, n \} $ such that $f^{*}\{y_i=1 \}-2\cdot f^* \{t=0\}$ is an effective divisor on $\overline{W}.$

Dividing $c_{p,s}(\Diamond_n;2)$ (resp. $c_{p,ss}(\Diamond_n;2)$) 
by the subgroup of degenerate cycles one obtains the groups $\pazocal{Z}_ {p,s}(\Diamond_n;2)$ (resp. $\pazocal{Z}_{p,ss}(\Diamond_n;2)).$ Letting $q=n+1-p,$ one denotes the same groups by $\pazocal{Z}_{s} ^q(\Diamond_n;2)$ (resp. $\pazocal{Z}_{ss} ^q(\Diamond_n;2)).$

\begin{remark}
If one would like to define a cycle complex over $k_2 $ computing its motivic cohomology, one might try to do so by considering $z_{f} ^{q}(S,n)/\pazocal{Z} ^{q} _{\cdot}(\Diamond_n;2),$ where $\cdot=s$ or $ss.$ However, we are doubtful that this would give the correct answer. In some sense, this approach only would only make cycles which are close to infinity equal to zero, whereas we think that two cycles which are close to each other in the sense of Remark  \ref{chow group} should be made the same in this quotient group. An opposite reason for our doubt in this definition with the sup modulus condition is that such a definition would imply that the regulator should be 0 on $\pazocal{Z} ^{q} _{s}(\Diamond_n;2).$ However, the example below shows that this is not the case. 
\end{remark}

{\it Example.} Consider  the cycle $Z \in z_{f} ^{2}(S,3)$ given by the parametric equations 
$$
xy=t, \; y_1=1-x^3, \; y_2=\frac{1+2y^2}{1+y^2}, \; y_3=\frac{1+2y}{1+y}.
$$
For simplicity, assume that $k$ is algebraically closed. We would like to compute $\rho_{f}(Z).$ Let us look at the contributions from each of the faces. Since  $\partial_{1} ^{\infty}\overline{Z}=(1,1)$ and $l(1,1)=0,$  this term does not contribute to the regulator. For $i=2,3,$ and $a=0,\, \infty,$ $\partial_{i} ^{a}\overline{Z}=(1-\alpha t^3,\cdot),$ for some $\alpha \in k.$ Since $l(1-\alpha t^3,\cdot)=0,$ these  terms do not contribute to the regulator either. Therefore, $\rho_f(Z)=l(\partial _{1} ^0 Z)=\sum _{\omega^3=1} l(\frac{1+2\omega t^2}{1+\omega t^2},\frac{1+2\omega^2 t}{1+\omega^2 t})=-3.$

This cycle given by the same formula above satisfies the sup modulus condition with modulus 2 and hence gives an element in $Z \in \pazocal{Z}_{s} ^2(\Diamond_3;2)),$  but does not satisfy the strong sup modulus condition with modulus 2 hence $Z \notin \pazocal{Z}_{ss} ^2(\Diamond_3;2)).$   In terms of our viewpoint, it is not surprising that the above cycle does not have 0 regulator value, since there is not a single function $y_i$ which is constant on all of $Z|_{t^2}:$ on one component $y_1$ is constant, on the  component $y_2$ is constant. 

 Note that that there is a natural  map  $\iota: \pazocal{Z} ^{q} _{s}(\Diamond_n;2) \to  z_{f} ^{q} (S,n),$ where $\iota$ maps a cycle $Z$ in  $\mathbb{A}^{1} _k\times_k  \square^{n} _k$ to its completion $\iota(Z)$ along 0 in $\mathbb{A}^{1} _k.$ However, in principle $z_{f} ^{q} (S,n)$ contains much more cycles since they need not satisfy any modulus condition. 
 
 Park defines a map $R: \pazocal{Z} ^{2} _{s}(\Diamond_2;2) \to k.$ We have the following relation between $R$ and $l.$ 

\begin{lemma} 
We have the equality $l \circ \iota=R$ of the  functions from $\pazocal{Z} ^{2} _{s}(\Diamond_2;2)$ to $k.$
\end{lemma}

\begin{proof}
Direct computation of both sides. 
\end{proof}

Since we defined $\rho_{f}$ as the composition $l\circ \partial$ and since $\partial \circ \iota=\iota \circ \partial, $ we have the equality of the following functions on $\pazocal{Z} ^{2} _{s}(\Diamond_3;2):$
\begin{eqnarray}\label{parkunver}
\rho_f\circ \iota=R\circ \partial. 
\end{eqnarray}
 The above example shows that $\rho_f\circ \iota$ is not necessarily 0 on $\pazocal{Z} ^{2} _{s}(\Diamond_3;2).$ On the other hand, if $Z$ is a cycle satisfying the strong modulus condition, in other words $Z$ is in $\pazocal{Z} ^{2} _{ss}(\Diamond_3;2),$ such that $\iota(Z)$ is smooth then $\rho_f\circ \iota$  vanishes on $Z$ \cite[Theorem 3.1]{P1}. This can also  be seen by applying Proposition \ref{products} to $\iota(Z)$ since if $Z$ satisfies the modulus condition $(iii)_{ss}$ then $y_{i}$ restricted to $\iota(Z)|_{t^2}$ is equal to 1.


\begin{thebibliography}{99}





\bibitem{bloch} S. Bloch. {\sl Algebraic cycles and higher $K$-theory\/}, Adv. Math., 61, (1986), no. 3, 267--304.




%\bibitem{Bloch reg} S., Bloch, {\sl Higher regulators, algebraic $K$-theory, and zeta functions of elliptic curves\/}, CRM Monograph Series, \textbf{11}, Amer. Math. Soc., Providence, RI, 2000. x+97pp.

\bibitem{BE2} S. Bloch, H. Esnault. {\sl The additive dilogarithm\/}, Doc. Math., Extra Vol., (2003), 131--155.




%\bibitem{GMS} H. Gangl, S. M\"uller-Stach, {\sl Polylogarithmic identities in cubical higher Chow groups\/}, in Algebraic $K$-theory (Seattle, WA, 1997), pp. 25--40, Proc. Symp. Pure Math., \textbf{67}, Amer. Math. Soc., Providence, RI, 1999.



%\bibitem{gon1} A. Goncharov. {\it Chow polylogarithms and regulators.} Math. Res. Lett. 2, no. 1,  95-112, (1995).

\bibitem{gon2} A. Goncharov.{\it Geometry of configurations,
polylogarithms, and motivic cohomology.} Advances in Math. 114 (1995), 197-318.

\bibitem{vol} A. Goncharov. {\it Volumes of hyperbolic manifolds and mixed Tate motives.} J. Amer. Math. Soc. 12 (1999), no. 2, 569-618.

\bibitem{gon3} A. Goncharov. {\it Polylogarithms, regulators, and Arakelov motivic complexes.} J. Amer. Math. Soc., 18 (2005), no. 1, 1-60.  


 \bibitem{griff} M. Green, P. Griffiths. {\it On the tangent space to the space of algebraic cycles on a smooth algebraic variety.} Annals of Mathematics Studies, 157. Princeton University Press, Princeton, NJ, (2005). vi+200 pp.

\bibitem{SGA1} A. Grothendieck. {\sl Rev\^etements \'etales et groupe fondamental\/}, S\'eminaire de G\'eom\'etrie Alg\'ebrique du Bois Marie 1960--1961 (SGA 1). Dirig\'e par Alexandre Grothendieck. Augment\'e de deux expos\'es de M. Raynaud. Lecture Notes in Mathematics, Vol. 224. Springer-Verlag, Berlin-New York, (1971), xxii+447 pp.
 






%\bibitem{Lomadze} V. G. Lomadze, {\sl On residues in algebraic geometry\/}, Izv. Akad. Nauk SSSR Ser. Math. \textbf{45}, (1981), 1258--1287; English translation in Math. USSR Izv., \textbf{19}, (1982), no. 3, 495--520.


%\bibitem{Mirzaii} B. Mirzaii, {\sl A Bloch-Wigner exact sequence over local rings\/}, J. Algebra, \textbf{476}, (2017), 459--493.



%\bibitem{P2} J. Park, {\sl Algebraic cycles and additive dilogarithm\/}, Int. Math. Res. Not., \textbf{2007}, no. 18, Article ID rnm067.

\bibitem{P1} J. Park. {\sl Regulators on additive higher Chow groups\/}, Amer. J. Math., 131, (2009), no. 1, 257--276.

\bibitem{pu} J. Park, S. \"Unver.  {\it Motivic cohomology of fat points in Milnor range.} Preprint.

\bibitem{sus} A. A. Suslin. {\it Reciprocity laws and the stable rank of rings of polynomials.} 
Izv. Akad. Nauk SSSR Ser. Mat. 43 (1979), no. 6, 1394-1429. 

%\bibitem{PU 0-cycle} J. Park, S. \"Unver, {\sl The Milnor range of the motivic cohomology of truncated polynomials over a field of any characteristic\/}, work in progress.

%\bibitem{PU general} J. Park, S. \"Unver, {\sl A theory of motivic cohomology for schemes over a field with arbitrary singularities\/}, work in progress

%\bibitem{Parshin} A. N. Parshin, {\sl On the arithmetic of two dimensional schemes, I. Distributions and residues}, Izv. Akad. Nauk SSSR Ser. Mat., \textbf{40}, (1976), no. 4, 736--773; English translation in Math. USSR Izv., \textbf{10}, (1976), no. 4, 695--729.



%\bibitem{R} K. R\"ulling, {\sl The generalized de Rham-Witt complex over a field is a complex of zero-cycles\/}, J. Alg. Geom., \textbf{16}, (2007), no. 1, 109--169.

%\bibitem{Suslin} A. A. Susllin, {\sl $K_3$ of a field, and the Bloch group\/}, Trudy Mat. Inst. Steklov., \textbf{183}, (1990), 180--199, 229. (Russian); Translation in Proc. Steklov Inst. Math. \textbf{1991}, no. 4, 217--239.


%\bibitem{Totaro} B. Totaro, {\sl Milnor $K$-theory is the simplest part of algebraic $K$-theory\/}, $K$-theory, \textbf{6}, (1992), no. 2, 177--189.


\bibitem{Unver} S. \"{U}nver. {\sl On the additive dilogarithm\/}, Algebra Number Theory, 3, (2009), no. 1, 1--34.

%\bibitem{Unver2} S. \"{U}nver, {\sl Additive polylogarithms and their functional equations\/}, Mathematische  Annalen, \textbf{348}, (2010), no. 4, 833--858.

%\bibitem{Voevodsky} V. Voevodsky, {\sl Motivic cohomology groups are isomorphic to higher Chow groups in any characteristic\/}, Int. Math. Res. Not. \textbf{2002}, no. 7, 351--355. 

%\bibitem{Yeku} A. Yekutieli, {\sl An explicit construction of the Grothendieck residue complex, with an appendix by P. Sastry\/}, Ast\'erisque, \textbf{208}, (1992), Soc. Math. France.




\end{thebibliography}
\end{document}